\documentclass[a4paper,11pt,reqno]{amsart}

\usepackage{
	amsfonts,
	amsmath,
	amssymb,
	bbold,
	dsfont,
	mathrsfs,
	mathtools,
	soul,
	subfig,
	xcolor
}

\usepackage[autostyle]{csquotes}
\usepackage{cite}
\usepackage{geometry}
\usepackage[hidelinks]{hyperref}
\usepackage[latin1]{inputenc}

\usepackage{tikz}
\usetikzlibrary{decorations.pathreplacing}
\usetikzlibrary{shapes.geometric,calc,decorations.markings,math}
\tikzstyle{nodo}=[circle,draw,fill,inner sep=0pt,minimum size=%
1.5mm]
\tikzstyle{infinito}=[circle,inner sep=0pt,minimum size=0mm]

\usetikzlibrary{shapes.geometric,calc,decorations.markings,math}
\tikzstyle{nodo}=[circle,draw,fill,inner sep=0pt,minimum size=0.5*width("k")]
\tikzstyle{infinito}=[circle,inner sep=0pt,minimum size=0mm]
\tikzset{every loop/.style={min distance=10mm,in=300,out=240,looseness=10}}
\tikzset{place/.style={circle,thick,draw=blue!75,fill=blue!20,minimum
		size=6mm}}
\tikzset{place2/.style={circle,thick,draw=red!75,fill=red!20,minimum
		size=6mm}}

\newtheorem{theorem}{Theorem}[section]
\newtheorem{proposition}[theorem]{Proposition}
\newtheorem{lemma}[theorem]{Lemma}

\theoremstyle{remark}
\newtheorem{remark}[theorem]{Remark}
\newtheorem*{remark*}{Remark}

\theoremstyle{definition}

\newcommand\N{{\mathbb N}}
\newcommand\R{{\mathbb R}}
\newcommand\Z{{\mathbb Z}}
\newcommand\NN{{\mathcal N}}
\newcommand\JJ{{\mathcal J}}
\newcommand\EE{{\mathcal E}}

\newcommand{\G}{\mathcal{G}}

\newcommand\wJ{{\widetilde J}}

\renewcommand\v{\textsc{v}}

\newcommand\f{\frac}

\newcommand\Hmu{{H_\mu^1}}

\title[Singular limit of periodic metric grids]{Singular limit of periodic metric grids}

\author[S. Dovetta]{Simone Dovetta}
\address[S. Dovetta]{Politecnico di Torino, Dipartimento di Scienze Matematiche ``G.L. Lagrange'', Corso Duca degli Abruzzi 24, 10129 - Torino, Italy.} 
\email{simone.dovetta@polito.it}

\begin{document}
	
\begin{abstract}
	We investigate the asymptotic behaviour of nonlinear Schr\"odinger ground states on $d$--dimensional periodic metric grids in the limit for the length of the edges going to zero. We prove that suitable piecewise--affine extensions of such states converge strongly in $H^1(\R^d)$ to the corresponding ground states on $\R^d$. As an application of such convergence results, qualitative properties of ground states and multiplicity results for fixed mass critical points of the energy on grids are derived. Moreover, we compare optimal constants in $d$--dimensional Gagliardo--Nirenberg inequalities on $\R^d$ and on grids. For $L^2$--critical and supercritical powers, we show that the value of such constants on grids is strictly related to that on $\R^d$ but, contrary to $\R^d$, constants on grids are not attained. The proofs of these results combine purely variational arguments with new Gagliardo--Nirenberg inequalities on grids.
\end{abstract}

\maketitle

\section{Introduction}

A $d$--dimensional periodic metric grid is a noncompact metric graph obtained by gluing together in a $\Z^d$--periodic pattern infinitely many copies of a given compact graph (see e.g. Figure \ref{fig:grid}). The peculiarity of such structures is the combination of a one--dimensional microscale, common to every metric graph, with a higher dimensional macroscale. Heuristically, this suggests that, in the limit for the length of the edges going to zero, periodic grids may give a locally--one dimensional approximation of $\R^d$, and one may imagine to exploit this fact to design one--dimensional models capable to mimic the dynamics of systems in higher dimension. Suppose we are interested in a model defined on $\R^d$ and we are able to construct another model on grids whose behaviour, when the edges of the graph are sufficiently small, is close to that of the original one. The graph approximation can then provide a low dimensional tool for the analysis of a $d$--dimensional phenomenon. This can be seen as an analogue of the classical approximation with point grids, with the possible advantage of a microscale finer than the distance between the vertices of the lattice. Though currently out of reach, in this sense one may even envisage tackling open problems in higher dimensions relying on the potential gain of a locally one--dimensional framework. Conversely, assume we are interested in a model on grids for which we can identify a counterpart in $\R^d$  somehow close to the original model when the edgelength of the grid is small enough. One can then interpret the $d$--dimensional model as an effective description of that on grids. When the behaviour in $\R^d$ is well--understood, this may be helpful to derive qualitative information on the dynamics on the network, whose direct investigation is often hindered by the complexity of the structure.  

Such a prospect raises the following question: given a model in $\R^d$ and a corresponding one on metric grids with vanishing edgelength, is it possible to prove that solutions of the latter converge in some sense to those of the former?

Of course, whether this can be done turns out to depend on the specific problem under exam. The present paper aims at initiating the investigations in this direction. Here we focus on a class of nonlinear PDEs that, also due to their prominent role in a variety of applications, have been widely studied by now both on metric graphs and on domains of $\R^d$: nonlinear Schr\"odinger equations.

In $\R^d$, it is well--known (see e.g. \cite{cazenave}) that, for every $p\in (2,2^*)$, where $2^*=\infty$ if $d=2$ and $2^*=\f{2d}{d-2}$ if $d\geq 3$, the stationary NonLinear Schr\"odinger (NLS) equation
\begin{equation}
\label{eq:NLSR2}
	\Delta u + |u|^{p-2}u=\omega u
\end{equation}
admits, for every $\omega>0$, a unique (up to translation) positive solution $u$ decaying at infinity. From a variational point of view, this solution can be characterized in at least two ways. For every $p\in(2,2^*)$, $u$ is an action ground state at frequency $\omega$, i.e. a global minimizer of the action
\[
J_{\omega,\R^d}(u):=\f12\|\nabla u\|_{L^2(\R^d)}^2+\f\omega 2\|u\|_{L^2(\R^d)}^2-\f1p\|u\|_{L^p(\R^d)}^p
\]
restricted to the associated Nehari manifold
\[
\begin{split}
\NN_{\omega,\R^d}:&=\left\{u\in H^1(\R^d)\,:\,J_{\omega,\R^d}'(u)u=0\right\}\\
&=\left\{u\in H^1(\R^d)\,:\,\|\nabla u\|_{L^2(\R^d)}^2+\omega\|u\|_{L^2(\R^d)}^2=\|u\|_{L^p(\R^d)}^p\right\}.
\end{split}
\]
Moreover, for every $p\in\left(2,2+\f4d\right)$, $u$ is also an energy ground state with mass $\mu$, i.e. a global minimizer of the energy
\[
E_{\R^d}(u):=\f12\|\nabla u\|_{L^2(\R^d)}^2-\f1p\|u\|_{L^p(\R^d)}^p
\]
in the mass constrained space
\[
H_\mu^1(\R^d):=\left\{u\in H^1(\R^d)\,:\, \|u\|_{L^2(\R^d)}^2=\mu\right\}\,.
\]
In this case $\omega$ plays the role of a Lagrange multiplier and  is in one--to--one correspondence with the mass $\mu$.

\begin{figure}
	\centering
	\subfloat[]{\begin{tikzpicture}[xscale= 0.5,yscale=0.5]
		\draw[step=2,thin] (0,0) grid (8,8);
		\foreach \x in {0,2,...,8} \foreach \y in {0,2,...,8} \node at (\x,\y) [nodo] {};
		\foreach \x in {0,2,...,8}
		{\draw[dashed,thin] (\x,8.2)--(\x,9.2) (\x,-0.2)--(\x,-1.2) (-1.2,\x)--(-0.2,\x)  (8.2,\x)--(9.2,\x); }
		\end{tikzpicture}}\qquad\qquad
	\subfloat[]{ 
	\begin{tikzpicture}
	[xscale= 0.5,yscale=0.5]
	\draw[step=3,thin] (0,0) grid (9,9);
	\foreach \x in {0,3,...,9} \foreach \y in {0,3,...,9} \node at (\x,\y) [nodo] {};
	
	\foreach \x in {0,3,...,9}
	{\draw[thin] (\x,9)--(\x,9.7) (\x,0)--(\x,-0.7) (-0.7,\x)--(0,\x)  (9,\x)--(9.7,\x); 
		\draw[dashed,thin] (\x,9.7)--(\x,10.2) (\x,-0.7)--(\x,-1.2) (-1.2,\x)--(-0.7,\x) (9.7,\x)--(10.2,\x);}
	
	\foreach \x in {0,3,...,9} \foreach \y in {0,3,...,9}
	{\draw[dashed,thin] (\x,\y)--(\x+2,\y+1.35);
		\draw[thin] (\x,\y)--(\x-0.5,\y-0.35);
		\draw[dashed,thin] (\x-0.5,\y-0.35)--(\x-1,\y-0.7);}
	
	\foreach \x in {1.5,4.5,...,10.5} \foreach \y in {1,4,...,10} \node at (\x,\y)
	[nodo] {};
	
	\foreach \x in {1.5,4.5,...,10.5} \foreach \y in {1,4,...,7} 
	{\draw[dashed,thin] (\x,\y)--(\x,\y+3);}
	
	\foreach \x in {1.5,4.5,...,10.5} 
	{\draw[dashed,thin] (\x,1)--(\x,0) (\x,10)--(\x,11);}
	
	\foreach \x in {1.5,4.5,...,7.5} \foreach \y in {1,4,...,10}
	{\draw[dashed,thin] (\x,\y)--(\x+3,\y);}
	
	\foreach \x in {1,4,...,10}
	{\draw[dashed,thin] (1.5,\x)--(0.5,\x) (10.5,\x)--(11.5,\x);}
	\end{tikzpicture}
}
\caption{The two--dimensional metric square grid (A) and the three--dimensional metric cubic grid (B).}
\label{fig:grid}
\end{figure}
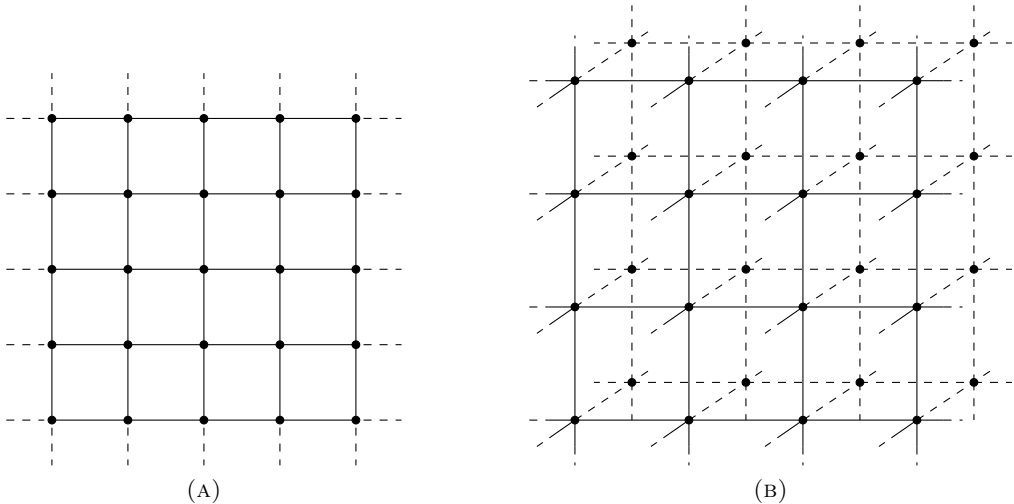
Recently, nonlinear Schr\"odinger equations have been considered extensively also on metric graphs. The literature on the subject is constantly growing, so that here we limit to redirect e.g. to \cite{ABD,BMP,BLS,BDL20,BDL21,BD22,BD21,BCJS,BCT,BCT2,CJS,DGMP,FMN,GSU21,KKLM,KMPX,KNP,NP,PS20,PSV} and references therein for some of the most recent developments, and to \cite{ABR,BK,KNP} for more comprehensive discussions. In particular, the existence of action and energy NLS ground states has been addressed on periodic grids in \cite{AD,ADST,DDGS} . When $\G$ is a $d$--dimensional periodic grid, it has been shown that the action
\[
J_{\lambda,\G}(u):=\f12\|u'\|_{L^2(\G)}^2+\f\lambda2\|u\|_{L^2(\G)}^2-\f1p\|u\|_{L^p(\G)}^p
\]
restricted to the Nehari manifold
\[
\begin{split}
\NN_{\lambda,\G}:&=\left\{u\in H^1(\G)\,:\,J_{\lambda,\G}'(u)u=0\right\}\\
&=\left\{u\in H^1(\G)\,:\,\|u'\|_{L^2(\G)}^2+\lambda\|u\|_{L^2(\G)}^2=\|u\|_{L^p(\G)}^p\right\}
\end{split}
\]
admits a ground state for every $\lambda>0$ if $p>2$, whereas the energy
\[
E_{\G}(u):=\f12\|u'\|_{L^2(\G)}^2-\f1p\|u\|_{L^p(\G)}^p
\]
constrained to the space of functions with fixed mass
\[
H_\mu^1(\G):=\left\{u\in H^1(\G)\,:\,\|u\|_{L^2(\G)}^2=\mu\right\}
\]
admits a ground state for every $\mu>0$ if $p\in\left(2,2+\f4d\right)$ (see \cite{DDGS} for action ground states, \cite{ADST,AD} for energy ground states with $d=2$ and $d=3$ respectively). Both ground states are (up to a change of sign) positive solutions of the NLS equation with homogeneous Kirchhoff conditions at the vertices
\[
\begin{cases}
u''+|u|^{p-2}u=\lambda u & \text{on every edge of }\G\\
\sum_{e\succ \v}\f{du}{dx_e}(\v)=0 & \text{for every vertex }\v\text{ of }\G.
\end{cases}
\]
Here, $e\succ \v$ means that the edge $e$ is incident at the vertex $\v$, whereas $\f{du}{dx_e}(\v)$ denotes the outward derivative of $u$ at $\v$ along $e$.

Comparing the above results for ground states on $\R^d$ and on grids unravels a strong analogy and triggers the question: do NLS ground states on grids somehow approximate those in $\R^d$ when the length of the edges is sufficiently small?

Another element of the similarity between grids and higher dimensional spaces, strictly related to the energy ground state problem described above, is the validity on $d$--dimensional grids of the Sobolev inequality
\begin{equation}
\label{eq:sobG}
\|u\|_{L^{\f d{d-1}}(\G)}\leq \mathcal{S}_{\G}\|u'\|_{L^1(\G)} \qquad\forall u\in W^{1,1}(\G)\,,
\end{equation}
which is the analogue of the well--known inequality in $\R^d$
\begin{equation}
\label{eq:sobRd}
\|u\|_{L^{\f d{d-1}}(\R^d)}\leq \mathcal{S}_{\R^d}\|\nabla u\|_{L^1(\R^d)}\qquad\forall u\in W^{1,1}(\R^d)
\end{equation}
(see \cite{ADST,AD} for a proof in the cases $d=2,3$).
Inequalities \eqref{eq:sobG}--\eqref{eq:sobRd} yield the following $d$--dimensional Gagliardo--Nirenberg inequalities
\begin{equation}
\label{eq:GNd}
\begin{split}
&\|u\|_{L^q(\G)}^q\leq K_{q,\G}\|u\|_{L^2(\G)}^{d+(2-d)\f q2}\|u'\|_{L^2(\G)}^{\left(\f q2 -1\right)d} \qquad\,\,\,\,\forall u\in H^1(\G)\\
&\|u\|_{L^q(\R^d)}^q\leq K_{q,\R^d}\|u\|_{L^2(\R^d)}^{d+(2-d)\f q2}\|\nabla u\|_{L^2(\R^d)}^{\left(\f q2 -1\right)d} \qquad\forall u\in H^1(\R^d)
\end{split}
\end{equation}
for every $q>2$ if $d=2$ and every $q\in(2,2^*)$ if $d\geq3$, that play a crucial role in giving lower boundedness of the energy functional in the mass constrained space and determining the threshold $2+\f4d$ on the nonlinearity power (for a wider discussion see e.g. \cite{DT22}). Analogous questions as for NLS ground states can then be raised for these functional inequalities on grids and $\R^d$: is there any relation between their optimal constants? As the length of the edges goes to zero, do optimizers of these inequalities on grids, if they exist, converge to those in $\R^d$?

\section{Setting and main results} 
To present our main results , we develop our discussion on $d$--dimensional cubic grids.
For every $d\geq2$ and $\varepsilon>0$, let $\G_\varepsilon^d\subset\R^d$ be the $d$--dimensional cubic grid with edgelength $\varepsilon$ in $\R^d$ (see Figure \ref{fig:grid} for $d=2,3$), i.e. the metric graph $\G_\varepsilon^d=\left(\mathbb{V}_{\G_\varepsilon^d},\mathbb{E}_{\G_\varepsilon^d}\right)$ given by
\[
\mathbb{V}_{\G_\varepsilon^d}=\varepsilon\Z^d,\qquad\mathbb{E}_{\G_\varepsilon^d}=\left\{(\v_1,\v_2)\in\mathbb{V}_{\G_\varepsilon^d}\times\mathbb{V}_{\G_\varepsilon^d}\,:\,|\v_1-\v_2|=\varepsilon\right\}.
\]
To compare functions defined on $\G_\varepsilon^d$ with those defined on the whole $\R^d$, we consider the following piecewise--affine extension procedure. For every $k=(k_1,\dots,k_d)\in\Z^d$, let 
\[
C_k:=[\varepsilon k_1,\varepsilon(k_1+1)]\times\dots\times[\varepsilon k_d,\varepsilon(k_d+1)]
\]
be the $d$--dimensional cube of edgelength $\varepsilon$ with edges on $\G_\varepsilon^d$ and $\varepsilon k$ as vertex with smallest coordinates, and write $C_k$ as
\[
C_k=\bigcup_{\sigma\in \mathbb{S}_d}S_{k,\sigma}\,,
\]
where the union runs over all the permutations $\sigma$ in the symmetric group $\mathbb{S}_d$ of the set $\left\{1,\dots,d\right\}$ and $S_{k,\sigma}$ is the $d$--simplex given by
\begin{equation}
\label{eq:simplex}
S_{k,\sigma}=\left\{(x_1,\dots,x_d)\in C_k\,:\,x_{\sigma(1)}-\varepsilon k_{\sigma(1)}\leq x_{\sigma(2)}-\varepsilon k_{\sigma(2)}\leq\dots\leq x_{\sigma(d)}-\varepsilon k_{\sigma(d)}\right\}.
\end{equation}
By construction, each $S_{k,\sigma}$ is the convex envelope of $d+1$ vertices of $C_k$ (see Figure \ref{fig:simpl} for $d=2,3$). Given $u:\G_\varepsilon^d\to\R$, we then define its piecewise--affine extension $\mathcal{A}u:\R^d\to\R$ as
\begin{equation}
\label{eq:Ad}
\mathcal{A}u(x):=\mathcal{A}_{k,\sigma}u(x)\quad\text{ if }x\in S_{k,\sigma},\text{ for some }k\in\Z^d\text{ and }\sigma\in\mathbb{S}_d\,,
\end{equation}
where $\mathcal{A}_{k,\sigma}u:S_{k,\sigma}\to\R$ is the affine interpolation of the values of $u$ at the vertices of $S_{k,\sigma}$. Note that $\mathcal{A}u$ is well--defined on the whole $\R^d$, because any non--empty intersection of two simplexes $S_{k,\sigma}$, $S_{k',\sigma'}$ is itself a simplex (of dimension smaller than $d$) contained in the boundary of $S_{k,\sigma}$ and $S_{k',\sigma'}$, so that $\mathcal{A}_{k,\sigma}u\equiv \mathcal{A}_{k',\sigma'}u$ on $S_{k,\sigma}\cap S_{k',\sigma'}$.

\begin{figure}[t]
	\centering
	\subfloat[][ ]{\includegraphics[width=0.24\columnwidth]{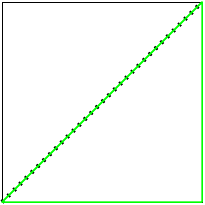}}
	\qquad\qquad\qquad\qquad
	\subfloat[][ ]{\includegraphics[width=0.3\columnwidth]{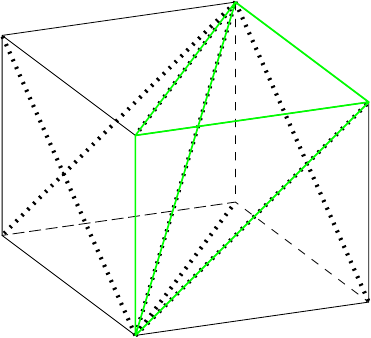}}
	\caption{The simplexes defined in \eqref{eq:simplex} for a two--dimensional square (A) and a three--dimensional cube (B). The edges of the boundary of the simplexes that do not coincide with edges of the square or the cube are denoted by bold dotted lines. The boundary of one of the simplexes is highlighted in green.}
	\label{fig:simpl}
\end{figure}

We can now state our main results, starting with the NLS ground state problems. Let us begin with action ground states. For every $p\in (2,2^*)$ and every $\lambda>0$, we introduce the action functional $\widetilde{J}_{\lambda,\G_\varepsilon^d}:H^1(\G_\varepsilon^d)\to\R$
\begin{equation}
	\label{eq:Jtilde}
	\widetilde{J}_{\lambda,\G_\varepsilon^d}(u):=\f12\|u'\|_{L^2(\G_\varepsilon^d)}^2+\f\lambda{2d}\|u\|_{L^2(\G_\varepsilon^d)}^2-\f1{dp}\|u\|_{L^p(\G_\varepsilon^d)}^p
\end{equation}
and the associated Nehari manifold
\[
	\begin{split}
	\widetilde{\NN}_{\lambda,\G_\varepsilon^d}:&=\left\{u\in H^1(\G_\varepsilon^d)\,:\,\widetilde{J}_{\lambda,\G_\varepsilon^d}'(u)u=0\right\}\\
	&=\left\{u\in H^1(\G_\varepsilon^d)\,:\,d\|u'\|_{L^2(\G_\varepsilon^d)}^2+\lambda\|u\|_{L^2(\G_\varepsilon^d)}^2=\|u\|_{L^p(\G_\varepsilon^d)}^p\right\}.
	\end{split}
\]
Letting
\[
	\widetilde{\JJ}_{\G_\varepsilon^d}(\lambda):=\inf_{v\in\widetilde{\NN}_{\lambda,\G_\varepsilon^d}}\wJ_{\lambda,\G_\varepsilon^d}(v)
\]
be the corresponding minimum problem, $u\in\widetilde{\NN}_{\lambda,\G_\varepsilon^d}$ is called a ground state of $\wJ_{\lambda,\G_\varepsilon^d}$ if $\wJ_{\lambda,\G_\varepsilon^d}(u)=\widetilde{\JJ}_{\G_\varepsilon^d}(\lambda)$. Since the existence results of \cite{DDGS} do not depend on the coefficients in the functional (provided the $L^p$ term is negative and the other ones positive), it follows that there always exist ground states of $\wJ_{\lambda,\G_\varepsilon^d}$ in $\widetilde{\NN}_{\lambda,\G_\varepsilon^d}$ for every $\lambda>0$, $\varepsilon>0$ and $p\in(2,2^*)$. Ground states are constant sign solutions of the NLS equation
\[
\begin{cases}
u''+\f1d|u|^{p-2}u=\f\lambda d u & \text{on every edge of }\G_\varepsilon^d\\
\sum_{e\succ \v}\f{du}{dx_e}(\v)=0 & \text{for every vertex }\v\text{ of }\G_\varepsilon^d\,.
\end{cases}
\]
The next theorem answers in the affirmative to the question whether action ground states on $d$--dimensional cubic grids with vanishing edgelength approximate ground states in the whole $\R^d$. In what follows, for every $\omega>0$ we make use of the shorthand notation
\[
	\JJ_{\R^d}(\omega):=\inf_{v\in\NN_{\omega,\R^d}}J_{\omega,\R^d}(v)\,.
\]
\begin{theorem}
	\label{thm:action}
	Let $d\geq 2$, $p\in(2,2^*)$ and $\omega>0$ be fixed.
	\begin{itemize}
		\item[(i)] If $d=2$ and $p>2$, or $d\geq3$ and $p\in\left(2,\f{2^*}2+1\right]$, then there exists $C_{d,p}>0$ such that
		\[
			\left|\varepsilon^{d-1}\widetilde{\JJ}_{\G_\varepsilon^d}(\omega)-\JJ_{\R^d}(\omega)\right|\leq C_{d,p}\varepsilon \qquad\text{as }\varepsilon\to0\,,
		\]
		whereas if $d\geq 3$ and $p\in\left(\f{2^*}2+1,2^*\right)$, then for every $\gamma>0$ there exists $C_{d,p,\gamma}>0$ such that
		\[
			\left|\varepsilon^{d-1}\widetilde{\JJ}_{\G_\varepsilon^d}(\omega)-\JJ_{\R^d}(\omega)\right|\leq C_{d,p,\gamma}\varepsilon^{\f{d-2}2(2^*-p)-\gamma}\qquad\text{as }\varepsilon\to0\,;
		\]
		
		\item[$(ii)$] for every positive ground state $u_\varepsilon$ of $\widetilde{J}_{\omega,\G_\varepsilon^d}$ in $\widetilde{\NN}_{\omega,\G_\varepsilon^d}$ there exists $x_\varepsilon\in\R^d$ such that
		\begin{equation*}
		\mathcal{A} u_\varepsilon(\cdot-x_\varepsilon) \xrightarrow[]{\varepsilon\to0} \varphi_\omega\quad\text{ in }H^1(\R^d)\,,
		\end{equation*}
		where $\varphi_\omega\in\NN_{\omega,\R^d}$ is the unique positive ground state of $J_{\omega,\R^d}$ attaining its $L^\infty$ norm at the origin. 
	\end{itemize}
\end{theorem}
An analogous result holds true for fixed mass ground states of the energy. For every $p\in\left(2,2+\f4d\right)$, we introduce the energy functional $\widetilde{E}_{\G_\varepsilon^d}:H^1(\G_\varepsilon^d)\to\R$
\begin{equation}
\label{eq:Etilde}
\widetilde{E}_{\G_\varepsilon^d}(u):=\f12\|u'\|_{L^2(\G_\varepsilon^d)}^2-\f1{dp}\|u\|_{L^p(\G_\varepsilon^d)}^p
\end{equation}	
and denote by
\[
\widetilde{\EE}_{\G_\varepsilon^d}(\mu):=\inf_{v\in H_\mu^1(\G_\varepsilon^d)}\widetilde{E}_{\G_\varepsilon^d}(v)
\]
the corresponding ground state problem at mass $\mu>0$.
As usual, $u\in H_\mu^1(\G_\varepsilon^d)$ is called a ground state of $\widetilde{E}_{\G_\varepsilon^d}$ at mass $\mu$ if $\widetilde{E}_{\G_\varepsilon^d}(u)=\widetilde{\EE}_{\G_\varepsilon^d}\left(\mu\right)$. If $u$ is a ground state of $\widetilde{E}_{\G_\varepsilon^d}$, then
\begin{equation*}
\label{eq:nlsetilde}
\begin{cases}
u''+\f1d|u|^{p-2}u=\mathcal{L}_{\G_\varepsilon^d}(u) u & \text{on every edge of }\G_\varepsilon^d\\
\sum_{e\succ \v}\f{du}{dx_e}(\v)=0 & \text{for every vertex }\v\text{ of }\G_\varepsilon^d\,,
\end{cases}
\end{equation*}
where
\begin{equation}
\label{eq:Lu}
\mathcal{L}_{\G_\varepsilon^d}(u):=\f{\f1d\|u\|_{L^p(\G_\varepsilon^d)}^p-\|u'\|_{L^2(\G_\varepsilon^d)}^2}{\|u\|_{L^2(\G_\varepsilon^d)}^2}\,.
\end{equation}
Adapting the analysis of \cite{ADST} ensures that there always exist ground states of $\widetilde{E}_{\G_\varepsilon^d}$ at mass $\mu$ for every $\varepsilon>0$, $\mu>0$ and $p\in\left(2,2+\f4d\right)$. For such ground states we have the following convergence result, where we also use the notation
\[
\EE_{\R^d}(\mu):=\inf_{u\in H_\mu^1(\R^d)}E_{\R^d}(v)\,.
\]
\begin{theorem}
	\label{thm:2dsquare}
	Let  $p\in\left(2,2+\f4d\right)$ and $\mu>0$ be fixed. 
	\begin{itemize}
		\item[$(i)$] If $d\in\left\{2,3,4\right\}$ and $p\in\left(2,2+\f4d\right)$, or $d\geq 5$ and $p\in\left(2,\f{2^*}2+1\right)$, then there exists $C_{d,p}>0$ such that
		\[
		\left|\varepsilon^{d-1}\widetilde{\EE}_{\G_\varepsilon^d}\left(\f d{\varepsilon^{d-1}}\mu\right)-\EE_{\R^d}(\mu)\right|\leq C_{d,p}\varepsilon\qquad\text{as }\varepsilon\to0, 
		\]
		whereas if $d\geq 5$ and $p\in\left(\f{2^*}2+1,2+\f4d\right)$, then for every $\gamma>0$ there exists $C_{d,p,\gamma}>0$ such that
		\[
			\left|\varepsilon^{d-1}\widetilde{\EE}_{\G_\varepsilon^d}\left(\f d{\varepsilon^{d-1}}\mu\right)-\EE_{\R^d}(\mu)\right|\leq C_{d,p,\gamma}\varepsilon^{\f{d-2}2(2^*-p)-\gamma}\qquad\text{as }\varepsilon\to0\,;
		\]
		\item[$(ii)$] for every positive ground state $u_\varepsilon$ of $\widetilde{E}_{\G_\varepsilon^d}$ in $H_{\f{d}{\varepsilon^{d-1}}\mu}^1(\G_\varepsilon^d)$ there exists $x_\varepsilon\in\R^d$ such that
		\begin{equation*}
		\mathcal{A} u_\varepsilon(\cdot-x_\varepsilon) \xrightarrow[]{\varepsilon\to0} \phi_\mu\quad\text{ in }H^1(\R^d)\,,
		\end{equation*}
		where $\phi_\mu\in H_\mu^1(\R^d)$ is the unique positive ground state of $E_{\R^d}$ at mass $\mu$ attaining its $L^\infty$ norm at the origin.
		Furthermore, 
			\begin{equation}
				\label{eq:convLu}
				\lim_{\varepsilon\to0}\mathcal{L}_{\G_\varepsilon^d}(u_\varepsilon)=\f{\omega_\mu}d\,,
			\end{equation} 
		where $\omega_\mu$ is the value of the parameter $\omega$ for which $\phi_\mu$ solves \eqref{eq:NLSR2}.
	\end{itemize}
\end{theorem}
Theorems \ref{thm:action}--\ref{thm:2dsquare} show that one has to consider slightly modified action and energy functionals on grids to recover ground states in $\R^d$ solving \eqref{eq:NLSR2}. This is no surprise. The scale factor $\varepsilon^{d-1}$ multiplying the ground state levels on grids is due to the different local dimensions of grids and $\R^d$, as one can see e.g. by comparing, for small $\varepsilon$, the volume of a ball of radius $\varepsilon$ in $\R^d$ (which is proportional to $\varepsilon^d$) with that of its restriction to $\G_\varepsilon^d$ (which goes like $\varepsilon$). Conversely, the coefficients $1/d$ and $d$ appearing in front of the norms in both problems are determined by the specific shape of the periodicity cell of $\G_\varepsilon^d$ (see also Section \ref{sec:gen} below).
Note that Theorems \ref{thm:action}(ii)--\ref{thm:2dsquare}(ii) do not require to pass to subsequences of ground states as $\varepsilon\to0$ by the uniqueness, up to symmetries, of the limit solution.

The proof of Theorems \ref{thm:action}--\ref{thm:2dsquare} combines purely variational arguments, based only on the minimality of ground states, with a deep analysis of the interaction between scales of different dimensions in the grid. This latter element is crucial to pass to the limit on the nonlinear term in the whole range $(2,2^*)$ for the nonlinearity power. Roughly, when $p\leq\f{2^*}2+1$, to prove the above convergence results it is enough to rely on the $d$--dimensional Gagliardo--Nirenberg inequalities \eqref{eq:GNd} on $\G_\varepsilon^d$. On the contrary, for larger values of $p$ we need to derive new Gagliardo--Nirenberg estimates on $\G_\varepsilon^d$ (see Proposition \ref{prop:GNint} below) interpolating between purely $d$--dimensional inequalities and purely two--dimensional ones. Such inequalities, that may perhaps be of some independent interest, suggest that $\G_\varepsilon^d$ shares a rich structure with features peculiar of any dimension between $1$ and $d$.  As so, these estimates have no analogue in $\R^d$.

The use of different estimates to deal with the cases $p\in\left(2,\f{2^*}2+1\right]$ and $p\in\left(\f{2^*}2+1,2^*\right)$ is also the reason for the rates of convergence on the ground state levels reported in Theorems \ref{thm:action}(i)--\ref{thm:2dsquare}(i). In the case of the energy, this difference becomes relevant only in dimension greater than or equal to $5$, since $2+\f4d\leq\f{2^*}2+1$ whenever $d=2,3,4$. We do not know whether the rates we obtain here are sharp. In particular, our method does not even allow us to understand whether the rate $o\left(\varepsilon^{\f{d-2}2(2^*-p)-\gamma}\right)$, for every $\gamma>0$, can be improved at least to $O\left(\varepsilon^{\f{d-2}2(2^*-p)}\right)$.
\begin{remark}
	\label{rem:gamma}
	Since Theorems \ref{thm:action}--\ref{thm:2dsquare} prove convergence of minimizers, one may wonder whether it is possible to recover the same result in the framework of $\Gamma$--convergence, a rather natural question also in view of the large  literature available on discrete--to--continuum problems (see e.g. \cite[Chapter 11]{braides} and references therein for a comprehensive overview on the subject). For instance, in the case of fixed mass ground states of the energy (the discussion in the action setting is analogous), given $\mu>0$, we could consider the functionals $F_\varepsilon:H^1(\G_\varepsilon^d)\to\R$
	\[
	F_\varepsilon(u):=\begin{cases}
		\varepsilon^{d-1}\widetilde{E}_{\G_\varepsilon^d}(u) & \text{if }u\in H_{\f d{\varepsilon^{d-1}}\mu}^1(\G_\varepsilon^d)\\
		+\infty & \text{otherwise}
	\end{cases}
	\]
	and try to understand whether $\Gamma-\lim_{\varepsilon\to0}F_\varepsilon=F$, with $F:H^1(\R^d)\to\R$ given by
	\[
	F(u):=\begin{cases}
		E_{\R^d}(u) & \text{if }u\in H_\mu^1(\R^d)\\
		+\infty & \text{otherwise}
	\end{cases}
	\]
	and the convergence of functions on $\G_\varepsilon^d$ to those on $\R^d$ as in Theorem \ref{thm:2dsquare}, i.e. strong convergence in $H^1(\R^d)$ of piecewise--affine extensions as in \eqref{eq:Ad}. Recall that, to obtain convergence of minimizers of $F_\varepsilon$ to those of $F$, such a $\Gamma$--convergence result would not be enough and should be coupled with the equicoercivity of the sublevel sets of $F_\varepsilon$ (see e.g. \cite[Theorem 2.9]{braides}). However, it is easy to see that in our setting here such equicoercivity does not hold (see Remark \ref{rem:noequi} below). 
\end{remark} 

As already pointed out, one can look at Theorems \ref{thm:action}--\ref{thm:2dsquare} in two ways: using ground states on grids as one--dimensional approximations of those on $\R^d$, or using the model in $\R^d$ to effectively describe that on grids. 

From the point of view of NLS equations in $\R^d$, the above theorems rigorously justify the approximation of $d$--dimensional ground states with their analogue on grids and open the way to numerical implementations in this spirit (see  \cite{BDL20,BDL21} for recent results on numerical schemes for ground states on graphs). Note that, since what we did here works for the action when $p\in(2,2^*)$ and for the energy when $p\in\left(2,2+\f4d\right)$, up to now we can only handle regimes of nonlinearities where the ground state problems on $\R^d$ are well--known. It would be interesting to investigate whether one can perform an analogous limit procedure in cases where the limit problem is not well--seated. Just to give an example, one can think for instance at equation \eqref{eq:NLSR2} with Sobolev critical power $p=2^*$ on bounded domains of $\R^d$, where even existence of positive solutions is often an open question (whereas \cite{DDGS} showed that action ground states always exist on grids for every $p>2$). Let us stress, however, that at the time being it is not clear to us whether our method can be extended to deal with $p=2^*$. We plan to further investigate this point in future works. 

Theorems \ref{thm:action}--\ref{thm:2dsquare} also say that ground states in $\R^d$ provide an effective description of suitable ground states on grids with small edgelength. Since, by simple scaling arguments (see Remark \ref{rem:GetoG1} below), such ground states on $\G_\varepsilon^d$ are in one--to--one correspondence with action ground states at small $\omega$ and energy ground states at small masses on the grid $\G_1^d$ with edges of length 1, this allows one to exploit ground states on $\R^d$ to derive qualitative properties of ground states on $\G_1^d$.  We briefly discuss here two explicit instances of this approach.

First, observe that action and energy ground states are positive solutions of the NLS equation that, in general, can be different from each other (for recent discussions on this point in full generality see e.g. \cite{DST22,JL}). However, since \eqref{eq:NLSR2} in $\R^d$ admits a unique positive solution decaying at infinity, these two notions on ground states do coincide in $\R^d$, in the sense that there is a one--to--one correspondence between $\omega>0$ and $\mu>0$ such that the action ground state $\varphi_\omega$ and the energy ground state $\phi_\mu$ are actually the same function. Furthermore, such solutions are always radially symmetric and decreasing in $\R^d$. On the contrary, on the grid nothing is known neither about the relation between action and energy ground states nor about their symmetry.  With respect to this, even though we are still not able to tackle these problems at any fixed $\omega$ and $\mu$, when $p\in\left(2,2+\f4d\right)$ Theorems \ref{thm:action}--\ref{thm:2dsquare} provide a first indication that these properties are recovered asymptotically in the limits for $\omega\to0$ and $\mu\to0$ (as both ground states, suitably scaled and extended to $\R^d$ through $\mathcal{A}$, converge to the same symmetric function). 

Second, as a by--product of the argument developed to prove Theorem \ref{thm:action}, we obtain the following multiplicity result for mass constrained critical points of the energy at large masses on $\G_1^d$.
\begin{proposition}
	\label{prop:mult}
	Let $d\in\left\{2,3\right\}$ and $p\in\left(2+\f4d,6\right)$ or $d\geq4$ and $p\in\left(2+\f4d,2^*\right)$ be fixed. Then there exists $(\mu_n)_n\subset\R$, with $\mu_n\to+\infty$ as $n\to+\infty$, such that $\widetilde{E}_{\G_1^d}$ has at least two different critical points in $H_{\mu_n}^1(\G_1^d)$ for every $n$.
\end{proposition}
The two critical points in Proposition \ref{prop:mult} are an energy and an action ground state. Note that existence of large mass energy ground states on $\G_1^d$ can be proved for every $p\in\left[2+\f4d,6\right)$ arguing as in \cite{ADST}. That the two solutions are actually different is a consequence of the fact that, when the mass is large enough, energy  ground states share a large Lagrange multiplier, whereas Theorem \ref{thm:action} guarantees that large mass action ground states correspond to $\omega$ close to $0$.  Observe that, since Theorem \ref{thm:action} applies only when $p\in(2,2^*)$, when $d\geq4$ the above multiplicity result does not cover the whole range $p\in\left(2+\f4d,6\right)$ where energy ground states at large masses do exist.

\smallskip
We now turn our attention to the comparison for Sobolev \eqref{eq:sobG}--\eqref{eq:sobRd} and Gagliardo--Nirenberg \eqref{eq:GNd} inequalities. In $\R^d$, the sharp constant in Sobolev inequality \eqref{eq:sobRd} is well--known (see \cite{talenti})
\[
\mathcal{S}_{\R^d}:=\sup_{u\in W^{1,1}(\R^d)}\f{\|u\|_{L^{\f d{d-1}}(\R^d)}}{\|\nabla u\|_{L^1(\R^d)}}=\f{\left(\Gamma\left(1+\f d2\right)\right)^{\f1d}}{d\sqrt{\pi}}\,,
\]
where $\Gamma$ is the gamma function, and it is not attained. On the $d$--dimensional grid of edgelength $1$, by \cite{Ham} it follows that
\[
\mathcal{S}_{\G_1^d}:=\sup_{u\in W^{1,1}(\G_1^d)}\f{\|u\|_{L^{\f d{d-1}}(\G_1^d)}}{\|u'\|_{L^1(\G_1^d)}}=\f1{(2d)^{\f 1d}}
\]
and it is easily seen that it is not attained too. Since elementary computations show that
\[
 \f1{(2d)^{\f 1d}}>\f{\left(\Gamma\left(1+\f d2\right)\right)^{\f1d}}{d\sqrt{\pi}}\qquad\forall d\geq2,
\]
this essentially exhausts the discussion on Sobolev inequalities.

The situation seems to be more involved for Gagliardo--Nirenberg inequalities. Letting
\begin{equation}
	\label{eq:defQ}
\begin{split}
Q_{q,\G_1^d}(u):=&\,\f{\|u\|_{L^q(\G_1^d)}^q}{\|u\|_{L^2(\G_1^d)}^{d+(2-d)\f q2}\|u'\|_{L^2(\G_1^d)}^{\left(\f q2-1\right)d}}, \qquad u\in H^1(\G_1^d)\,,\\
Q_{q,\R^d}(v):=&\,\f{\|v\|_{L^q(\R^d)}^q}{\|v\|_{L^2(\R^d)}^{d+(2-d)\f q2}\|\nabla v\|_{L^2(\R^d)}^{\left(\f q2-1\right)d}}, \qquad v\in H^1(\R^d)\,,
\end{split}
\end{equation}
and denoting by
\begin{equation}
\label{eq:optGn}
K_{q,\G_1^d}:=\sup_{u\in H^1(\G_1^d)}Q_{q,\G_1^d}(u),\qquad K_{q,\R^d}:=\sup_{v\in H^1(\R^d)}Q_{q,\R^d}(v)
\end{equation}
the best constants in \eqref{eq:GNd}, we have the next partial result.
\begin{theorem}
\label{thm:GN} 
For every $d\geq2$ and $q\in\left(2,2^*\right)$, there holds
\begin{equation}
\label{eq:KGgeqKR2}
K_{q,\G_1^d}\geq d^{\f{(d-2)(q-2)}4}K_{q,\R^d}\,.
\end{equation}
Furthermore, if $q\in\left[2+\f4d,2^*\right)$ then
\begin{equation}
	\label{eq:KG=KR2}
K_{q,\G_1^d}=d^{\f{(d-2)(q-2)}4}K_{q,\R^d}
\end{equation}
and $K_{q,\G_1^d}$ is not attained for every $q>2+\f4d$.
\end{theorem}
When $q\in\left[2+\f4d,2^*\right)$, on one side Theorem \ref{thm:GN} establishes a strong analogy between cubic grids and $\R^d$, with optimal constants being the same up to a factor depending only on the dimension. Interestingly, when $d=2$, the two constants are exactly the same. On the other side, it marks a difference, since such a constant is not attained by any function in $H^1(\G_1^d)$ for $q\in\left(2+\f4d,2^*\right)$, whereas existence of optimizers is well--known in $\R^d$ for every $q\in\left(2,2^*\right)$. As the proof of Theorem \ref{thm:GN} reveals, non--existence of optimizers on $\G_1^d$ is a consequence of the validity, for every $q>2$, of the Gagliardo--Nirenberg inequality (see e.g. \cite[Theorem 2.1]{ADST} in the two--dimensional setting)
\begin{equation}
\label{eq:gn1d}
\|u\|_{L^q(\G_1^d)}^q\leq C_q\|u\|_{L^2(\G_1^d)}^{\f q2+1}\|u'\|_{L^2(\G_1^d)}^{\f q2-1}\qquad\forall u\in H^1(\G_1^d)\,,
\end{equation}
where $C_q>0$ depends only on $q$.
Since \eqref{eq:gn1d} is based in the one--dimensional microscale of the graph, this is a peculiar feature of metric grids, that has no counterpart both in $\R^d$ and in the discrete setting $\Z^d$, where the one--dimensional local structure is absent \cite{weinstein}. 

Conversely, for $q\in\left(2,2+\f4d\right)$ we are not able at the moment to improve the upper bound \eqref{eq:KGgeqKR2}. It remains an open problem to understand whether in this regime it is possible that $K_{q,\G_1^d}$ be attained and, this being the case, whether one can recover a convergence result for optimizers of \eqref{eq:GNd} similar to those in Theorems \ref{thm:action}--\ref{thm:2dsquare}. Our best result in this direction at present is the following existence criterion.
\begin{proposition}
	\label{prop:ex24}
	Let $q\in\left(2,2+\f4d\right]$. If there exists $u\in H^1(\G_1^d)$ such that $Q_{q,\G_1^d}(u)\geq d^{\f{(d-2)(q-2)}4}K_{q,\R^d}$, then $K_{q,\G_1^d}$ is attained.
\end{proposition}

To conclude this section, let us stress that the method developed here is by no means limited to the specific case of cubic grids. It is in fact fairly easy to generalize our approach to grids with periodicity cell of different shape. For the sake of completeness, the final section of the paper briefly overviews how this can be done on two specific examples of non--square grids in the plane: the regular triangular grid (Figure \ref{fig:altre}(A)) and the regular hexagonal one (Figure \ref{fig:altre}(B)).

\begin{figure}[t]
	\centering
	\subfloat[][ ]{\includegraphics[width=0.35\columnwidth]{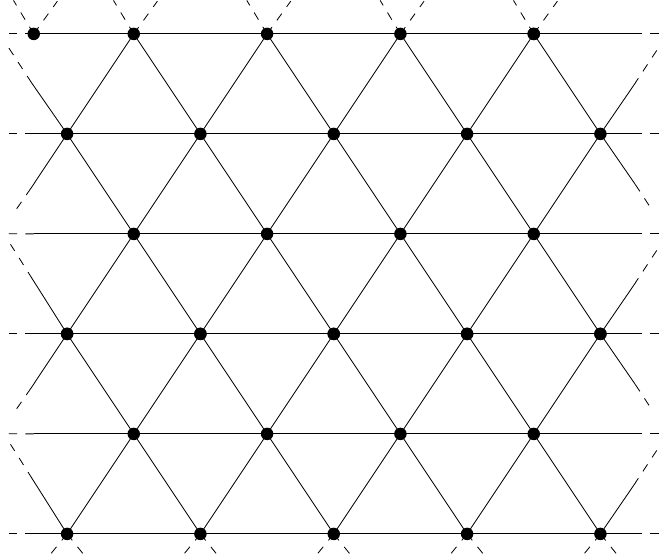}}
	\qquad\qquad
	\subfloat[][ ]{\includegraphics[width=0.35\columnwidth]{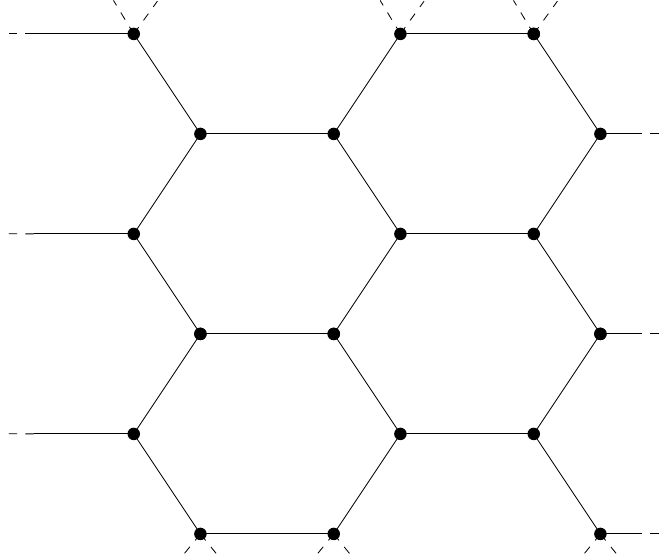}}
	\caption{Two--dimensional regular triangular (A) and hexagonal (B) grids.}
	\label{fig:altre}
\end{figure}

\smallskip
The remainder of the paper is organized as follows. Section \ref{sec:prelGN} establishes some results on Gagliardo--Nirenberg inequalities on grids. Section \ref{sec:prel} provides general estimates involving restrictions of functions from $\R^d$ to grids and extensions of functions from grids to $\R^d$ that will be widely used throughout the discussion. Section \ref{sec:apriori} is devoted to an a priori analysis of the ground state problems on grids with edges of fixed length, whereas the proof of Theorems \ref{thm:action}--\ref{thm:2dsquare} and of Proposition \ref{prop:mult} is given in Section \ref{sec:proof}. Section \ref{sec:gn} reports the proof of Theorem \ref{thm:GN} and Proposition \ref{prop:ex24}. 

\section*{Acknowledgements}
\noindent The author wishes to thank Riccardo Adami, Claudio Canuto, Enrico Serra, Paolo Tilli and Lorenzo Tentarelli for useful discussions and suggestions during the preparation of this work.
The work has been partially supported by the INdAM GNAMPA project 2022 ``Modelli matematici con singolarità per fenomeni di interazione".

%

\bigskip

\noindent{\em Notation:} in what follows, we will often use the symbols $\lesssim_{a,b,c}\,,\, \gtrsim_{a,b,c}$ to indicate that the corresponding estimates hold up to a multiplicative constant depending only on the parameters $a$, $b$, $c$. 

\section{Gagliardo--Nirenberg inequalities on $\G_\varepsilon^d$}
\label{sec:prelGN}

Given their importance throughout all the paper, we start with some results about Gagliardo--Nirenberg inequalities on $d$--dimensional cubic grids that will be used in the following. 

As anticipated in the Introduction, the analysis developed in dimensions two and three in \cite{AD,ADST} can be easily adapted to show that, for every $d\geq 2$, $q\in(2,2^*]$ and $\varepsilon>0$, there exists $K_{q,\G_\varepsilon^d}>0$, depending only on $d$ and $q$, such that the $d$--dimensional Gagliardo--Nirenberg inequality
\begin{equation}
\label{eq:dgnGe}
\|u\|_{L^q(\G_\varepsilon^d)}^q\leq K_{q,\G_\varepsilon^d}\|u\|_{L^2(\G_\varepsilon^d)}^{d+(2-d)\f q2}\|u'\|_{L^2(\G_\varepsilon^d)}^{\left(\f q2-1\right)d}
\end{equation}
holds true for every $u\in H^1(\G_\varepsilon^d)$. The first result of this section is a simple remark that relates the sharp constant in this inequality with the edgelength of the grid.

\begin{lemma}
	\label{lem:GNeps}
	Let $K_{q,\G_\varepsilon^d}$ be the sharp constant in \eqref{eq:dgnGe}. Then $K_{q,\G_\varepsilon^d}=\varepsilon^{\left(\f q2-1\right)(d-1)}K_{q,\G_1^d}$.
\end{lemma}
\begin{proof}
	Given $u\in H^1(\G_\varepsilon^d)$, let $v:\G_1^d\to\R$ be such that $u(x)=v(x/\varepsilon)$ for every $x\in \G_\varepsilon^d$. Then $v\in H^1(\G_1^d)$ and
	\[
	\|u\|_{L^r(\G_\varepsilon^d)}^r=\varepsilon\|v\|_{L^r(\G_1^d)}^r,\quad\forall r\geq1,\qquad\|u'\|_{L^2(\G_\varepsilon^d)}^2=\f1\varepsilon\|v'\|_{L^2(\G_1^d)}^2\,,
	\]
	so that, by \eqref{eq:dgnGe} on $\G_1^d$,
	\[
	\begin{split}
	\|u\|_{L^q(\G_\varepsilon^d)}^q=\varepsilon\|v\|_{L^q(\G_1^d)}^q\leq& \varepsilon K_{q,\G_1^d}\|v\|_{L^2(\G_1^d)}^{d+(2-d)\f q2}\|v'\|_{L^2(\G_1^d)}^{\left(\f q2-1\right)d}\\
	\leq& \varepsilon^{1-\f d2-(2-d)\f q4+\left(\f q4-\f12\right)d}K_{q,\G_1^d}\|u\|_{L^2(\G_\varepsilon^d)}^{d+(2-d)\f q2}\|u'\|_{L^2(\G_\varepsilon^d)}^{\left(\f q2-1\right)d}\\
	=&\varepsilon^{\left(\f q2 -1\right)(d-1)}K_{q,\G_1^d}\|u\|_{L^2(\G_\varepsilon^d)}^{d+(2-d)\f q2}\|u'\|_{L^2(\G_\varepsilon^d)}^{\left(\f q2-1\right)d}\,.
	\end{split}
	\]
	This shows that $K_{q,\G_\varepsilon^d}\leq \varepsilon^{\left(\f q2 -1\right)(d-1)}K_{q,\G_1^d}$, and arguing analogously inverting the role of $u$ and $v$ proves the reverse inequality.
\end{proof}
 Being peculiar of dimension $d$, \eqref{eq:dgnGe} is not available when $q>2^*$. However, on grids it is possible to exploit the coexistence of scales of different dimensions to obtain new families of Gagliardo--Nirenberg inequalities for powers larger than the $d$--dimensional Sobolev critical exponent. In particular, the next proposition, which to the best of our knowledge is new and can be of some independent interest, provides a result in this spirit interpolating between purely two--dimensional and $d$--dimensional inequalities. 
\begin{proposition}
	\label{prop:GNint}
	Let $d\geq3$. For every $\varepsilon>0$, $q>2^*$ and $\alpha\in(0,2]$, there exists $C>0$, depending only on $d$, $q$ and $\alpha$, such that
	\begin{equation}
	\label{eq:GNint}
	\|u\|_{L^q(\G_\varepsilon^d)}^q\leq C\varepsilon^{\f q2+1-\alpha}\|u\|_{L^2(\G_\varepsilon^d)}^\alpha\|u'\|_{L^2(\G_\varepsilon^d)}^{q-\alpha}\qquad\forall u\in H^1(\G_\varepsilon^d)\,.
	\end{equation}
\end{proposition}
	\begin{proof}
		Observe first that it is enough to prove \eqref{eq:GNint} when $\varepsilon=1$. Indeed, for every $\varepsilon\neq 1$ and $u\in H^1(\G_\varepsilon^d)$, letting $v\in H^1(\G_1^d)$ satisfy $u(x)=v(x/\varepsilon)$ for every $x\in\ \G_\varepsilon^d$, it follows
		\[
		\|u\|_{L^q(\G_\varepsilon^d)}^q=\varepsilon\|v\|_{L^q{(\G_1^d)}}^q\,\lesssim_{d,q,\alpha}\,\varepsilon \|v\|_{L^2(\G_1^d)}^\alpha\|v'\|_{L^2(\G_1^d)}^{q-\alpha}=\varepsilon^{\f q2+1-\alpha}\|u\|_{L^2(\G_\varepsilon^d)}^\alpha\|u'\|_{L^2(\G_\varepsilon^d)}^{q-\alpha}\,.
		\]
		The proof of \eqref{eq:GNint} on $\G_1^d$ is divided in two cases.
		
		\smallskip
		{\em Case 1: $\alpha=2$}. Note that, when $\alpha=2$, \eqref{eq:GNint} is the standard two--dimensional Gagliardo--Nirenberg inequality
		\begin{equation}
		\label{eq:GN2d}
		\|u\|_{L^q(\G_1^d)}^q\lesssim_{d,q} \|u\|_{L^2(\G_1^d)}^2\|u'\|_{L^2(\G_1^d)}^{q-2}\qquad\forall u\in H^1(\G_1^d)\,.
		\end{equation}
		As is well--known (see e.g. \cite[Theorem 2.3]{ADST}), \eqref{eq:GN2d} can be easily deduced by the two--dimensional Sobolev inequality
		\begin{equation}
		\label{eq:S2d}
		\|u\|_{L^2(\G_1^d)}\lesssim_d \|u'\|_{L^1(\G_1^d)}\qquad\forall u\in W^{1,1}(\G_1^d)\,.
		\end{equation}
		To prove \eqref{eq:S2d}, it is enough to think of $\G_1^d$ as 
		\[
		\G_1^d=\bigcup_{\substack{i,j\in\left\{1,\dots,d\right\}\\i<j}}\bigcup_{k\in\Z^{d-2}}\left(\G_1^d\cap P_{ij}^k\right)
		\]
		where, for every $i,j=1,\dots,d$ and $k\in\Z^{d-2}$, 
		\[
		P_{ij}^k=\left\{(x_1,\dots,x_d)\in\R^d\,:\,(x_1,\dots,x_{i-1},x_{i+1},\dots,x_{j-1},x_{j+1},\dots,x_d)= k\right\}.
		\]
		Clearly, $P_{ij}^k$ is a plane parallel to the coordinate planes of $\R^d$, $\G_1^d\cap P_{ij}^k$ is a two--dimensional grid contained in $P_{ij}^k$ and each edge of $\G_1^d$ belongs to $d-1$ of such two--dimensional grids. 
		
		If $u\in W^{1,1}(\G_1^d)$, its restriction $u_{\mid \G_1^d\cap P_{ij}^k}$ to $\G_1^d\cap P_{ij}^k$ is in $W^{1,1}(\G_1^d\cap P_{ij}^k)$, so that by \cite[Theorem 2.2]{ADST} it follows
		\[
		\|u\|_{L^2\left(\G_1^d\cap P_{ij}^k\right)}^2\leq \f14\|u'\|_{L^1\left(\G_1^d\cap P_{ij}^k\right)}^2
		\]
		and, summing over $i,j\in\left\{1,\dots,d\right\},i<j$ and $k\in\Z^{d-2}$,
		\[
		\begin{split}
		\|u\|_{L^2(\G_1^d)}^2=& \f1{d-1}\sum_{\substack{i,j\in\left\{1,\dots,d\right\}\\i<j}}\sum_{k\in\Z^{d-2}}\|u\|_{L^2\left(\G_1^d\cap P_{ij}^k\right)}^2\leq \f1{4(d-1)}\sum_{\substack{i,j\in\left\{1,\dots,d\right\}\\i<j}}\sum_{k\in\Z^{d-2}}\|u'\|_{L^1\left(\G_1^d\cap P_{ij}^k\right)}^2\\
		\leq&\f1{4(d-1)}\left(\sum_{\substack{i,j\in\left\{1,\dots,d\right\}\\i<j}}\sum_{k\in\Z^{d-2}}\|u'\|_{L^1\left(\G_1^d\cap P_{ij}^k\right)}\right)^2=\f{d-1}4\|u'\|_{L^1(\G_1^d)}^2\,,
		\end{split}
		\]
		that is \eqref{eq:S2d}.
		
		\smallskip
		{\em Case 2: $\alpha\in(0,2)$}. Given $q>2^*$, we prove \eqref{eq:GNint}  with $\varepsilon=1$ by interpolation between \eqref{eq:GN2d} and the $d$--dimensional Gagliardo--Nirenberg inequality \eqref{eq:dgnGe} at power $2^*$
		\begin{equation}
		\label{eq:GNd2*}
		\|u\|_{L^{2^*}(\G_1^d)}^{2^*}\leq K_{2^*,\G_1^d}\|u'\|_{L^2(\G_1^d)}^{2^*}\qquad\forall u\in\ H^1(\G_1^d)\,.
		\end{equation}
		Set $r:=\f2\alpha q+2^*\left(1-\f2\alpha\right)$, so that,  since $\alpha\in(0,2)$, we have $2^*<q<r$ and 
		\[
		\|u\|_{L^q(\G_1^d)}^q\leq \|u\|_{L^{2^*}(\G_1^d)}^{2^*\left(1-\f\alpha2\right)}\|u\|_{L^r(\G_1^d)}^{r\f\alpha2}\,,
		\]
		which coupled with \eqref{eq:GNd2*} and \eqref{eq:GN2d} for the $L^r$ norm entails
		\[
		\|u\|_{L^q(\G_1^d)}^q\lesssim_{d,q,\alpha}\|u'\|_{L^2(\G_1^d)}^{2^*\left(1-\f\alpha 2\right)}\|u\|_{L^2(\G_1^d)}^{\alpha}\|u'\|_{L^2(\G_1^d)}^{\f\alpha 2(r-2)}=\|u\|_{L^2(\G_1^d)}^\alpha\|u'\|_{L^2(\G_1^d)}^{q-\alpha}
		\]
		and completes the proof. 
	\end{proof}

\section{General restriction and extension estimates}
\label{sec:prel}
This section collects results frequently used in the following to compare norms of functions on $\R^d$ with those on grids. We start by recalling a straightforward relation between the norms of a given function in $\R^d$ and the ones of its restriction to grids with vanishing edgelength.

\begin{lemma}
	\label{lem:restr}
	Let $u\in C^1(\R^d)\cap H^2(\R^d)\cap L^\infty(\R^d)$. For every $\varepsilon>0$, let $u_\varepsilon:\G_\varepsilon^d\to\R$ be the restriction of $u$ to $\G_\varepsilon^d$. Then there exists $C>0$, depending on $u$ but not on $\varepsilon$, such that, as $\varepsilon\to0$, it holds
	\begin{align}
		&\left|\f{\varepsilon^{d-1}}d\|u_\varepsilon\|_{L^q(\G_\varepsilon^d)}^q-\|u\|_{L^q(\R^d)}^q\right|\leq C\varepsilon\,,\qquad \forall q\geq2\,,\label{eq:uRdG}\\
		&\left|\varepsilon^{d-1}\|u_\varepsilon'\|_{L^2(\G_\varepsilon^d)}^2-\|\nabla u\|_{L^2(\R^d)}^2\right|\leq C\varepsilon\,.\label{eq:u'RdG}
	\end{align}
\end{lemma}
\begin{proof}
	Here it is convenient to think of $\G_\varepsilon^d$ as
	\[
	\G_\varepsilon^d=\bigcup_{i\in\Z^{d-1}}\bigcup_{j=1}^d X_{\varepsilon,i}^j
	\]
	where, for every $i\in \Z^{d-1}$ and $j=1,\dots,d$,
	\[
	X_{\varepsilon,i}^j=\left\{(\varepsilon i_1,\dots, \varepsilon i_{j-1},s,\varepsilon i_j,\dots,\varepsilon i_{d-1})\,:\,s\in\R\right\}.
	\]
	Let us split the proof in two parts.
	
	\smallskip
	{\em Part 1: proof of \eqref{eq:uRdG}.} For every fixed $j\in\left\{1,\dots,d\right\}$, consider the partition of $\R^d$ given by the sets, for every $i\in\Z^{d-1}$,
	\[
	A_{\varepsilon,i}^j:=[\varepsilon i_1,\varepsilon (i_1+1))\times\dots\times[\varepsilon i_{j-1},\varepsilon(i_{j-1}+1))\times\R\times[\varepsilon i_{j},\varepsilon(i_j+1))\times\dots\times[\varepsilon i_{d-1},\varepsilon(i_{d-1}+1))
	\] 
	and define $w_{\varepsilon,j}:\R^d\to\R$ as
	\[
	w_{\varepsilon,j}(x_1,\dots,x_d):=u(\varepsilon i_1, \varepsilon i_{j-1},x_j,\varepsilon i_j,\dots,\varepsilon i_{d-1})\qquad\forall (x_1,\dots,x_d)\in A_{\varepsilon,i}^j,\,\text{ for some }i\in\Z\,.
	\]
	By definition, for every $q\geq 2$,
	\begin{equation}
		\label{eq:Lqwj}
		\|w_{\varepsilon,j}\|_{L^q(\R^d)}^q=\varepsilon^{d-1}\sum_{i\in\Z^{d-1}}\int_\R|u(\varepsilon i_1, \varepsilon i_{j-1},x_j,\varepsilon i_j,\dots,\varepsilon i_{d-1})|^q\,dx_j=\varepsilon^{d-1}\|u_\varepsilon\|_{L^q\left(\bigcup_{i\in\Z^{d-1}}X_{\varepsilon,i}^j\right)}^q\,.
	\end{equation}
	Moreover, since $u\in H^1(\R^d)$, if $\varepsilon$ is small enough we have
	\[
	\begin{split}
	\|u-w_{\varepsilon,j}\|_{L^2(\R^d)}^2\leq\sum_{i\in\Z^{d-1}}&\int_{\varepsilon i_1}^{\varepsilon(i_1+1)}\dots\int_{\varepsilon i_{d-1}}^{\varepsilon(i_{d-1}+1)}\int_\R\left|u(x_1,\dots,x_d)\right.\\
	&\left.-u(\varepsilon i_1, \varepsilon i_{j-1},x_j,\varepsilon i_j,\dots,\varepsilon i_{d-1})\right|^2dx_jdx_d\dots dx_{j+1}dx_{j-1}\dots dx_{1}\\
	\leq \varepsilon^2\sum_{i\in\Z^{d-1}}&\int_{A_{\varepsilon,i}^j}|\nabla u|^2(x_1,\dots,x_d)\,dx_1\dots dx_d=\varepsilon^2\|\nabla u\|_{L^2(\R^d)}^2\,,
	\end{split}
	\]	
	which coupled with \eqref{eq:Lqwj} yields
	\[
	\left|\|u\|_{L^2(\R^d)}^2-\varepsilon^{d-1}\|u_\varepsilon\|_{L^2\left(\bigcup_{i\in\Z^{d-1}}X_{\varepsilon,i}^j\right)}^2\right|\leq 2\varepsilon\|u\|_{L^2(\R^d)}\|\nabla u\|_{L^2(\R^d)}+\varepsilon^2\|\nabla u\|_{L^2(\R^d)}^2\,.
	\]
	Summing over $j=1,\dots, d$ proves \eqref{eq:uRdG} with $q=2$.
	
	If $q>2$, arguing similarly we obtain for every $j$
	\begin{equation}
	\label{stimachiave}
	\begin{split}
	\left|\|u\|_{L^q(\R^d)}^q-\|w_{\varepsilon,j}\|_{L^q(\R^d)}^q\right|\leq& \varepsilon\int_{\R^d}|\nabla (|u|^{q})|(x_1,\dots,x_d)dx_1\dots dx_d\\
	\lesssim_q&\,\varepsilon\int_{\R^d}|u|^{q-1}|\nabla u|(x_1,\dots,x_d)dx_1\dots dx_d\lesssim_q \varepsilon\|u\|_{L^{2(q-1)}(\R^d)}^{q-1}\|\nabla u\|_{L^2(\R^d)}\,,
	\end{split}
	\end{equation}
	so that, summing over $j$ and recalling \eqref{eq:Lqwj},
	\[
	\left|\|u\|_{L^q(\R^d)}^q-\f{\varepsilon^{d-1}}d\|u_\varepsilon\|_{L^q(\G_\varepsilon^d)}^q\right|\lesssim_q \varepsilon\|u\|_{L^{2(q-1)}(\R^d)}^{q-1}\|\nabla u\|_{L^2(\R^d)}\,,
	\]
	which implies \eqref{eq:uRdG} since $u\in H^1(\R^d)\cap L^\infty(\R^d)$ by assumption.
	
	\smallskip
	{\em Part 2: proof of \eqref{eq:u'RdG}.} For every $j=1,\dots,d$, let $v_{\varepsilon,j}:\R^d\to\R$ be
	\[
	v_{\varepsilon,j}(x_1,\dots,x_d):=\partial_{x_j}u(\varepsilon i_1, \varepsilon i_{j-1},x_j,\varepsilon i_j,\dots,\varepsilon i_{d-1})\qquad\forall (x_1,\dots,x_d)\in A_{\varepsilon,i}^j,\,\text{ for some }i\in\Z\,,
	\]
	where $A_{\varepsilon,i}^j$ are as in Part 1 above. Then
	\[
	\|v_{\varepsilon,j}\|_{L^2(\R^d)}^2=\varepsilon^{d-1}\sum_{i\in\Z^d}\int_{\R}|\partial_{x_j}u(\varepsilon i_1, \varepsilon i_{j-1},x_j,\varepsilon i_j,\dots,\varepsilon i_{d-1})|^2\,dx_j=\varepsilon^{d-1}\|u_\varepsilon'\|_{L^2\left(\bigcup_{i\in\Z^{d-1}}X_{\varepsilon,i}^j\right)}^2
	\]
	and, since $u\in C^1(\R^d)\cap H^2(\R^d)$, arguing as before
	\[
	\begin{split}
	\left|\|\partial_{x_j}u\|_{L^2(\R^d)}^2-\|v_{\varepsilon,j}\|_{L^2(\R^d)}^2\right|\leq&\, \varepsilon\sum_{i\in\Z^{d-1}}\int_{A_{\varepsilon,i}^j}|\nabla ((\partial_{x_j}u)^2)|(x_1,\dots,x_d)\,dx_1\dots dx_d\\
	\lesssim& \,\varepsilon\|\nabla u\|_{L^2(\R^d)}\|D^2u\|_{L^2(\R^d)}\,,
	\end{split}
	\]
	so that summing over $j$ gives \eqref{eq:u'RdG}.
\end{proof}

We then turn our attention to the relation between a given function $u\in H^1(\G_\varepsilon^d)$, its piecewise--affine extension $\mathcal{A}u$ as defined in \eqref{eq:Ad} and the restriction of $\mathcal{A}u$ to $\G_\varepsilon^d$, that from now on will be denoted by $\widetilde{u}$. By definition, on each edge of $\G_\varepsilon^d$ $\widetilde{u}$ is the linear interpolation of the values of $u$ at the corresponding vertices. Hence, identifying any edge $e\in\mathbb{E}_{\G_\varepsilon^d}$, $e=(\v_1,\v_2)$,  with $[0,\varepsilon]$, and denoting by $\widetilde{u}_e$ the restriction of $\widetilde{u}$ to $e$, we have
\begin{equation}
\label{eq:util}
\widetilde{u}_e(x)=\f{u(\v_2)-u(\v_1)}\varepsilon x+u(\v_1), \quad \forall x\in[0,\varepsilon]\,.
\end{equation}
It is readily seen that, if $u\in H^1(\G_\varepsilon^d)$, then $\widetilde{u}\in H^1(\G_\varepsilon^d)$ too and, by Jensen inequality,
\begin{equation}
\label{eq:util'}
\|\widetilde{u}'\|_{L^2(\G_\varepsilon^d)}\leq\|u'\|_{L^2(\G_\varepsilon^d)}\,.
\end{equation}
We now want to estimate the distance between the $L^q$ norms of $u$ and $\widetilde{u}$. The next lemma does the job in the case $q=2$.
\begin{lemma}
	\label{lem:uutL2}
	For every $u\in H^1(\G_\varepsilon^d)$ it holds, as $\varepsilon\to0$,
	\begin{equation}
		\label{eq:uutL2}
		\left|\|u\|_{L^2(\G_\varepsilon^d)}^2-\|\widetilde{u}\|_{L^2(\G_\varepsilon^d)}^2\right|\leq 3\varepsilon\|u\|_{H^1(\G_\varepsilon^d)}^2\,.
	\end{equation}
\end{lemma}
\begin{proof}
	Throughout the proof, we use the following convention: for every edge $e\in\mathbb{E}_{\G_\varepsilon^d}$ identified by an interval of length $\varepsilon$ in the $j$-th coordinate direction of $\R^d$, we let $\v_1$ be its vertex with smallest $j$-th coordinate and $\v_2$ be the one with largest $j$-th coordinate. 
	
	Note first that, for every edge $e$ of $\G_\varepsilon^d$, we have both
	\[
	\|u\|_{L^2(e)}^2-\varepsilon|u(\v_1)|^2=\int_e (u^2(x)-u^2(\v_1))\,dx=\int_e\int_0^x(u^2)'(y)\,dydx\leq 2 \varepsilon\|u\|_{L^2(e)}\|u'\|_{L^2(e)}
	\]
	and
	\[
	\|\widetilde{u}\|_{L^2(e)}^2-\varepsilon|u(\v_1)|^2=\int_e (\widetilde{u}^2(x)-\widetilde{u}^2(\v_1))\,dx\leq 2 \varepsilon\|\widetilde{u}\|_{L^2(e)}\|\widetilde{u}'\|_{L^2(e)}\,,
	\]
	since by definition $u(\v_1)=\widetilde{u}(\v_1)$ for every $\v_1\in\mathbb{V}_{\G_\varepsilon^d}$. Summing over $e\in\mathbb{E}_{\G_\varepsilon^d}$ and making use of Cauchy--Schwarz inequality and \eqref{eq:util'} then gives
	\begin{equation}
	\label{eq:1L2}
	\begin{split}
	\left|\|u\|_{L^2(\G_\varepsilon^d)}^2-\|\widetilde{u}\|_{L^2(\G_\varepsilon^d)}^2\right|\leq& \left|\|u\|_{L^2(\G_\varepsilon^d)}^2-\varepsilon\sum_{e\in\mathbb{E}_{\G_\varepsilon^d}}|u(\v_1)|^2\right|+ \left|\|\widetilde{u}\|_{L^2(\G_\varepsilon^d)}^2-\varepsilon\sum_{e\in\mathbb{E}_{\G_\varepsilon^d}}|u(\v_1)|^2\right|\\
	\leq& \varepsilon\sum_{e\in\mathbb{E}_{\G_\varepsilon^d}}\|u\|_{L^2(e)}^2+\|u'\|_{L^2(e)}^2+\|\widetilde{u}\|_{L^2(e)}^2+\|\widetilde{u}'\|_{L^2(e)}^2\\
	\leq& \varepsilon\left(\|u\|_{L^2(\G_\varepsilon^d)}^2+\|\widetilde{u}\|_{L^2(\G_\varepsilon^d)}^2+2\|u'\|_{L^2(\G_\varepsilon^d)}^2\right)\,.
	\end{split}
	\end{equation}
	Furthermore, recalling \eqref{eq:util}, for every $e\in\mathbb{E}_{\G_\varepsilon^d}$ we obtain
	\[
	\begin{split}
	\|\widetilde{u}\|_{L^2(e)}^2=&\,\int_0^\varepsilon\left|\f{u(\v_2)-u(\v_1)}\varepsilon x+u(\v_1)\right|^2\,dx\\
	=&\,\left(\f{u(\v_2)-u(\v_1)}\varepsilon\right)^2\f{\varepsilon^3}3+|u(\v_1)|^2\varepsilon+u(\v_1)\f{u(\v_2)-u(\v_1)}\varepsilon \varepsilon^2\,,
	\end{split}
	\]
	so that summing over all edges of $\G_\varepsilon^d$ yields
	\begin{equation}
	\label{eq1}
	\left|\|\widetilde{u}\|_{L^2(\G_\varepsilon^d)}^2-\varepsilon\sum_{e\in\mathbb{E}_{\G_\varepsilon^d}}|u(\v_1)|^2\right|\leq\varepsilon^2\sum_{e\in\mathbb{E}_{\G_\varepsilon^d}}|u(\v_1)|\left|\f{u(\v_2)-u(\v_1)}\varepsilon\right|+\f{\varepsilon^3}3\sum_{e\in\mathbb{E}_{\G_\varepsilon^d}}\left(\f{u(\v_2)-u(\v_1)}\varepsilon\right)^2.
	\end{equation}
	By Cauchy--Schwarz inequality, for any fixed $\delta\in(0,1/2)$
	\[
	|u(\v_1)|\left|\f{u(\v_2)-u(\v_1)}\varepsilon\right|\leq \f12\left(\f{|u(\v_1)|^2}{\varepsilon^{1-2\delta}}+\f{|u(\v_2)-u(\v_1)|^2}{\varepsilon^{1+2\delta}}\right),
	\]
	whereas by Jensen inequality
	\[
	\left(\f{u(\v_2)-u(\v_1)}\varepsilon\right)^2=\left(\f1{\varepsilon}\int_e u'\,dx\right)^2\leq\f1\varepsilon\|u'\|_{L^2(e)}^2\,.
	\]
	Plugging into \eqref{eq1} and making use of Jensen inequality again leads to
	\begin{equation}
	\label{est21}
	\begin{split}
	\left|\|\widetilde{u}\|_{L^2(\G_\varepsilon)}^2-\varepsilon\sum_{e\in\mathbb{E}_{\G_\varepsilon^d}}|u(\v_1)|^2\right|\leq &\,\f{\varepsilon^{1+2\delta}}2\sum_{e\in\mathbb{E}_{\G_\varepsilon^d}}|u(\v_1)|^2\\
	&+\f{\varepsilon^{1-2\delta}}2\sum_{e\in\mathbb{E}_{\G_\varepsilon^d}}|u(\v_2)-u(\v_1)|^2+\f{\varepsilon^2}3\sum_{e\in\mathbb{E}_{\G_\varepsilon^d}}\|u'\|_{L^2(e)}^2\\
	\leq&\,\f{\varepsilon^{1+2\delta}}2\sum_{e\in\mathbb{E}_{\G_\varepsilon^d}}|u(\v_1)|^2+\f{\varepsilon^{2-2\delta}}2\sum_{e\in\mathbb{E}_{\G_\varepsilon^d}}\|u'\|_{L^2(e)}^2+\f{\varepsilon^2}3\sum_{e\in\mathbb{E}_{\G_\varepsilon^d}}\|u'\|_{L^2(e)}^2\\
	=&\,\f{\varepsilon^{1+2\delta}}2\sum_{e\in\mathbb{E}_{\G_\varepsilon^d}}|u(\v_1)|^2+\varepsilon\|u'\|_{L^2(\G_\varepsilon^d)}^2\left(\f{\varepsilon^{1-2\delta}}2+\f\varepsilon3\right)\,.
	\end{split}
	\end{equation}
	Since, arguing as in \eqref{eq:1L2}, one has
	\begin{equation}
	\label{sumV2}
	\varepsilon\sum_{e\in\mathbb{E}_{\G_\varepsilon^d}}|u(\v_1)|^2\leq (1+\varepsilon)\|u\|_{L^2(\G_\varepsilon^d)}^2+\varepsilon\|u'\|_{L^2(\G_\varepsilon^d)}^2\,,
	\end{equation}
	combining with \eqref{est21} yields, for sufficiently small $\varepsilon$,
	\[
	\begin{split}
	\|\widetilde{u}\|_{L^2(\G_\varepsilon^d)}^2\leq& \varepsilon\sum_{e\in\mathbb{E}_{\G_\varepsilon^d}}|u(\v_1)|^2\left(1+\f{\varepsilon^{2\delta}}2\right)+\varepsilon\|u'\|_{L^2(\G_\varepsilon^d)}^2\left(\f{\varepsilon^{1-2\delta}}2+\f{\varepsilon}3\right)\leq
	2\left(\|u\|_{L^2(\G_\varepsilon^d)}^2+\varepsilon\|u'\|_{L^2(\G_\varepsilon^d)}^2\right)
	\end{split}
	\]
	and plugging into \eqref{eq:1L2} gives \eqref{eq:uutL2}.
\end{proof}
When $d\geq3$ and $2<q<2^*$, an estimate analogous to that of the previous lemma is harder to obtain and exploits the interpolating Gagliardo--Nirenberg inequalities of Proposition \ref{prop:GNint}.  

\begin{lemma}
	\label{lem:uutLq}
	For every $u\in H^1(\G_\varepsilon^d)$ we have that, as $\varepsilon\to0$,
	\begin{itemize}
		\item[(i)] if $d=2$ and $q>2$, or $d\geq3$ and $q\leq\f{2^*}2+1$, then
		\begin{equation}
		\label{eq:uutLq1}
		\left|\|u\|_{L^q(\G_\varepsilon^d)}^q-\|\widetilde{u}\|_{L^q(\G_\varepsilon^d)}^q\right|\leq C\left(
		\varepsilon^{\f{q-2}2(d-1)+1}\|u\|_{L^2(\G_\varepsilon^d)}^{\f d2+\f{2-d}2(q-1)}\|u'\|_{L^2(\G_\varepsilon^d)}^{\f{q-2}2 d +1}+\varepsilon^{\f12\left(\f{q-2}2 d +3\right)}\|u'\|_{L^2(\G_\varepsilon^d)}^q\right)
		\end{equation} 
		for a constant $C>0$ depending on $q$ and $d$ only;
		
		\smallskip
		\item[(ii)] if $d\geq 3$ and $q>\f{2^*}2+1$, then for every $\gamma\in(0,1]$
		\begin{equation}
		\label{eq:uutLq2}
		\left|\|u\|_{L^q(\G_\varepsilon^d)}^q-\|\widetilde{u}\|_{L^q(\G_\varepsilon^d)}^q\right|\leq C\left(\varepsilon^{\f q2+1-\gamma}\|u\|_{L^2(\G_\varepsilon^d)}^\gamma\|u'\|_{L^2(\G_\varepsilon^d)}^{ q-\gamma}+\varepsilon^{\f q2+1-\f\gamma2}\|u'\|_{L^2(\G_\varepsilon^d)}^{q}\right)
		\end{equation}
		for a constant $C>0$ depending on $q,d$ and $\gamma$ only.
	\end{itemize}
	
\end{lemma}
\begin{proof}
	We adopt the same convention of the previous proof for the vertices $\v_1,\v_2$ of each edge $e\in\mathbb{E}_{\G_\varepsilon^d}$.
	By H\"older inequality, we have
	\begin{equation}
	\label{eq21}
	\begin{split}
	\left|\|u\|_{L^q(\G_\varepsilon^d)}^q-\varepsilon\sum_{e\in\mathbb{E}_{\G_\varepsilon^d}}|u(\v_1)|^q\right|\leq&\,\sum_{e\in\mathbb{E}_{\G_\varepsilon^d}}\int_e||u|^q(x)-|u|^q(\v_1)|\,dx
	\leq\sum_{e\in\mathbb{E}_{\G_\varepsilon^d}}\int_e\left|\int_e(|u|^q)'(y)\,dy\right|dx\\
	=&\,q\varepsilon\sum_{e\in\mathbb{E}_{\G_\varepsilon^d}}\int_e|u(y)|^{q-1}|u'(y)|\,dy\leq q\varepsilon\|u\|_{L^{2(q-1)}(\G_\varepsilon^d)}^{q-1}\|u'\|_{L^2(\G_\varepsilon^d)}\,.
	\end{split}
	\end{equation}
	Now, if $d=2$ and $q>2$ or $d\geq3$ and $q\leq \f{2^*}2+1$, that is $2(q-1)\leq2^*$, by \eqref{eq:dgnGe} and Lemma \ref{lem:GNeps} it follows
	\[
	\|u\|_{L^{2(q-1)}(\G_\varepsilon^d)}^{q-1}\lesssim_{d,q}\varepsilon^{\f{\left( q-2\right)}2(d-1)}\|u\|_{L^2(\G_\varepsilon^d)}^{\f{d+(2-d)(q-1)}2}\|u'\|_{L^2(\G_\varepsilon^d)}^{\f{\left(q-2\right)}2d}
	\]
	that coupled with \eqref{eq21} gives
	\begin{equation}
	\label{eq1i}
	\left|\|u\|_{L^q(\G_\varepsilon^d)}^q-\varepsilon\sum_{e\in\mathbb{E}_{\G_\varepsilon^d}}|u(\v_1)|^q\right|\lesssim_{d,q} \varepsilon^{1+\f{\left( q-2\right)}2(d-1)}\|u\|_{L^2(\G_\varepsilon^d)}^{\f{d+(2-d)(q-1)}2}\|u'\|_{L^2(\G_\varepsilon^d)}^{\f{\left(q-2\right)}2d+1}\,.
	\end{equation}
	Analogously, since by construction $u(\v)=\widetilde{u}(\v)$ for every $\v\in\mathbb{V}_{\G_\varepsilon^d}$, arguing as above and combining with Jensen inequality, Lemma \ref{lem:uutL2} we obtain, provided $\varepsilon$ is small enough,
	\begin{equation}
	\label{eq2i}
	\begin{split}
	\left|\|\widetilde{u}\|_{L^q(\G_\varepsilon^d)}^q-\varepsilon\sum_{e\in\mathbb{E}_{\G_\varepsilon^d}}|u(\v_1)|^q\right|=&\,\left|\|\widetilde{u}\|_{L^q(\G_\varepsilon^d)}^q-\varepsilon\sum_{e\in\mathbb{E}_{\G_\varepsilon^d}}|\widetilde{u}(\v_1)|^q\right|\\
	\lesssim_{d,q}&\, \varepsilon^{1+\f{\left( q-2\right)}2(d-1)}\|\widetilde{u}\|_{L^2(\G_\varepsilon^d)}^{\f{d+(2-d)(q-1)}2}\|\widetilde{u}'\|_{L^2(\G_\varepsilon^d)}^{\f{\left(q-2\right)}2d+1}\\
	\lesssim_{d,q}& \,\varepsilon^{1+\f{\left( q-2\right)}2(d-1)}\left(\|u\|_{L^2(\G_\varepsilon^d)}^2+\varepsilon\|u'\|_{L^2(\G_\varepsilon^d)}^2\right)^{\f{d+(2-d)(q-1)}4}\|u'\|_{L^2(\G_\varepsilon^d)}^{\f{\left(q-2\right)}2d+1}\\
	\lesssim_{d,q}& \,\left(\varepsilon^{1+\f{\left( q-2\right)}2(d-1)}\|{u}\|_{L^2(\G_\varepsilon^d)}^{\f{d+(2-d)(q-1)}2}\|{u}'\|_{L^2(\G_\varepsilon^d)}^{\f{\left(q-2\right)}2d+1}+\varepsilon^{\f12\left(3+\f{q-2}2d\right)}\|u'\|_{L^2(\G_\varepsilon^d)}^q\right)\,,
	\end{split}
	\end{equation}
	where the last inequality makes use of the subadditivity on $\R^+$ of the map $s\mapsto s^{\f{d+(2-d)(q-1)}4}$, which is guaranteed for every $q>2$ when $d=2$ and every $q\in\left(2,\f{2^*}2+1\right]$ when $d\geq 3$. Coupling \eqref{eq1i} and \eqref{eq2i} gives \eqref{eq:uutLq1}.
	
	If $d\geq3$ and $q>\f{2^*}2+1$, i.e. $2(q-1)>2^*$, we combine \eqref{eq21} with Proposition \ref{prop:GNint} to get, for every $\alpha\in(0,2]$,
	\begin{equation}
		\label{eq1ii}
		\left|\|u\|_{L^q(\G_\varepsilon^d)}^q-\varepsilon\sum_{e\in\mathbb{E}_{\G_\varepsilon^d}}|u(\v_1)|^q\right|\lesssim_{d,q,\alpha}\varepsilon^{\f{q-\alpha}2+1}\|u\|_{L^2(\G_\varepsilon^d)}^{\f\alpha2}\|u'\|_{L^2(\G_\varepsilon^d)}^{q-\f\alpha2}\,.
	\end{equation}
	Arguing as above on $\widetilde{u}$ we  have also
	\[
	\begin{split}
	\left|\|\widetilde{u}\|_{L^q(\G_\varepsilon^d)}^q-\varepsilon\sum_{e\in\mathbb{E}_{\G_\varepsilon^d}}|u(\v_1)|^q\right|\lesssim_{d,q,\alpha} &\,\varepsilon^{\f{q-\alpha}2+1}\|\widetilde{u}\|_{L^2(\G_\varepsilon^d)}^{\f\alpha2}\|\widetilde{u}'\|_{L^2(\G_\varepsilon^d)}^{q-\f\alpha2}\\
	\lesssim_{d,q,\alpha} & \,\left(\varepsilon^{\f{q-\alpha}2+1}\|u\|_{L^2(\G_\varepsilon^d)}^{\f\alpha2}\|u'\|_{L^2(\G_\varepsilon^d)}^{q-\f\alpha2}+\varepsilon^{\f q2+1-\f\alpha4}\|u'\|_{L^2(\G_\varepsilon^d)}^{q}\right)\,,
	\end{split}
	\]
	which together with \eqref{eq1ii} yields \eqref{eq:uutLq2} with $\gamma=\alpha/2$.
\end{proof}
We conclude this section with some elementary estimates involving the extension operator $\mathcal{A}$.

\begin{lemma}
	\label{lem:normA}
	For every $u\in H^1(\G_\varepsilon^d)$ it holds, as $\varepsilon\to0$,
	\begin{align}
		&\|\mathcal{A}u\|_{L^2(\R^d)}^2\leq 2^{d}(d+1)\varepsilon^{d-1}\left(\|u\|_{L^2(\G_\varepsilon^d)}^2+\varepsilon\|u'\|_{L^2(\G_\varepsilon^d)}^2\right) \label{eq:AL2}\\
		&\|\nabla\mathcal{A}u\|_{L^2(\R^d)}^2\leq \varepsilon^{d-1}\|u'\|_{L^2(\G_\varepsilon^d)}^2\,.\label{eq:gradAL2}
	\end{align}
	Furthermore, if $d\geq 3$ and $q>2^*$, then 
	for every $\gamma\in(0,1]$
	\begin{equation}
		\label{eq:ALq}
		\|\mathcal{A}u\|_{L^q(\R^d)}^q\leq C\varepsilon^{d-1}\left(\|u\|_{L^q(\G_\varepsilon^d)}^q+\varepsilon^{\f q2+1-\gamma}\|u\|_{L^2(\G_\varepsilon^d)}^\gamma\|u'\|_{L^2(\G_\varepsilon^d)}^{q-\gamma}\right)
	\end{equation}
	for a constant $C>0$ depending only on $d,p$ and $\gamma$.
\end{lemma}
\begin{proof}
	For every $k\in\Z^d$, recall that we denote by $C_k$ the $d$--dimensional cube of edgelength $\varepsilon$ with edges on $\G_\varepsilon^d$ and vertex with smallest coordinates at $\varepsilon k$.  For every $\sigma\in\mathbb{S}_d$, on the simplex $S_{k,\sigma}$ given in \eqref{eq:simplex} by definition $\mathcal{A}u$ is a convex combination of the values of $u$ at the $d+1$ vertices of $C_k$ contained in $S_{k,\sigma}$. Hence, by Jensen inequality, for every $q\geq 2$
	\[
	\|\mathcal{A}u\|_{L^q(S_{k,\sigma})}^q\leq \f{(d+1)^{q-1}}{d!}\varepsilon^d\sum_{\v\in\mathbb{V}_{\G_\varepsilon^d}\cap S_{k,\sigma}}u(\v)^q\,,
	\] 
	so that summing over $\sigma\in\mathbb{S}_d$
	\[
	\begin{split}
	\|\mathcal{A}u\|_{L^q(C_k)}^q\leq&\sum_{\sigma\in\mathbb{S}_d}\f{(d+1)^{q-1}}{d!}\varepsilon^d\sum_{\v\in\mathbb{V}_{\G_\varepsilon^d}\cap S_{k,\sigma}}u(\v)^q\\
	\leq& \sum_{\v\in\mathbb{V}_{\G_\varepsilon^d}\cap C_k}u(\v)^q\sum_{\sigma\in\mathbb{S}_d}\f{(d+1)^{q-1}}{d!}\varepsilon^d=(d+1)^{q-1}\varepsilon^d\sum_{\v\in\mathbb{V}_{\G_\varepsilon^d}\cap C_k}u(\v)^q\,.
	\end{split}
	\]
	Since each vertex of $\G_\varepsilon^d$ belongs to $2^d$ such cubes, this entails
	\[
	\|\mathcal{A}u\|_{L^q(\R^d)}^q=\sum_{k\in\Z^d}\|\mathcal{A}u\|_{L^q(C_k)}^q\leq (d+1)^{q-1}\varepsilon^d\sum_{k\in\Z^d}\sum_{\v\in\mathbb{V}_{\G_\varepsilon^d}\cap C_k}u(\v)^q=2^d(d+1)^{q-1}\sum_{\v\in\mathbb{V}_{\G_\varepsilon^d}}|u(\v)|^q\,.
	\]
	When $q=2$, coupling with \eqref{sumV2} gives \eqref{eq:AL2}, whereas for $q>2^*$ combining with \eqref{eq1ii} yields \eqref{eq:ALq}.
	
	To prove \eqref{eq:gradAL2}, observe first that, for every $k\in\Z^d$, there is a one--to--one correspondence between the simplexes $S_{k,\sigma}$ as in \eqref{eq:simplex} contained in $C_k$ and the shortest paths in $\G_\varepsilon^d$ going from the vertex with smallest coordinates of $C_k$ to that with largest coordinates, since both sets are in one--to--one correspondence with the symmetric group $\mathbb{S}_d$. This immediately tells that each simplex contains exactly $d+1$ edges of $\G_\varepsilon^d$. Moreover, given any edge $e\in\mathbb{E}_{\G_\varepsilon^d}$, to find the total number $N(e)$ of simplexes $S_{k,\sigma}$ that contain $e$ it is enough to count the number of times $e$ belongs to a shortest path connecting the vertices of $C_k$ with smallest and largest coordinates, for some $k\in\Z^d$. 
	 With no loss of generality, let $e=(\v_1,\v_2)$, where $\v_1,\v_2\in\mathbb{V}_{\G_\varepsilon^d}$ are such that the $j$--th coordinate of $\v_1$ is smaller than that of $\v_2$, while all the other coordinates are the same. Then $e$ belongs to $C_k$ if and only if the $j$--th coordinate of $\v_1-\varepsilon k$ is equal to $0$ and, among the remaining $d-1$ coordinates, $i$ of them are equal to $\varepsilon$ and the other $d-i-1$ are equal to $0$, for some $i=0,\dots,d-1$. For fixed $i$, the number of such $k\in\Z^d$ is $\binom{d-1}{i}$. Morever, since $\varepsilon k$ is the vertex with smallest coordinates in $C_k$ and it has $i$ coordinates smaller than those of $\v_1$, there are $i!(d-1-i)!$ shortest paths in $C_k$ starting at $\varepsilon k$, passing through $e$ and ending at the vertex with largest coordinates of $C_k$.  Therefore, the total number of such paths $e$ belongs to is
	 \[
	 N(e)=\sum_{i=0}^{d-1}\binom{d-1}{i}i!(d-1-i)!=d!\,.
	 \]
	 Since a direct computation shows that, if $e_{k,\sigma}^1,\dots,e_{k,\sigma}^{d+1}$ are the $d+1$ edges of $\G_\varepsilon^d$ that belong to $S_{k,\sigma}$,
	 \[
	 \|\nabla\mathcal{A}u\|_{L^2(S_{k,\sigma})}^2=\f{\varepsilon^{d-1}}{d!}\sum_{i=1}^{d+1}\|\widetilde{u}'\|_{L^2(e_{k,\sigma}^i)}^2\,,
	 \]
	 we obtain
	 \[
	 \|\nabla\mathcal{A}u\|_{L^2(\R^d)}^2=\f{\varepsilon^{d-1}}{d!}\sum_{k\in\Z^d}\sum_{\sigma\in\mathbb{S}_d}\sum_{i=1}^{d+1}\|\widetilde{u}'\|_{L^2(e_{k,\sigma}^i)}^2=\f{\varepsilon^{d-1}}{d!}\sum_{e\in\mathbb{E}_{\G_\varepsilon^d}}N(e)\|\widetilde{u}'\|_{L^2(e)}^2=\varepsilon^{d-1}\|\widetilde{u}'\|_{L^2(\G_\varepsilon^d)}^2\,.
	 \]
	 Combining with \eqref{eq:util'}, we conclude.
\end{proof}

\section{A priori estimates on ground states of $\widetilde{\JJ}_{\omega,\G_\varepsilon^d}$ and  $\widetilde{\EE}_{\G_\varepsilon^d}$}
\label{sec:apriori}

Throughout this section we establish a priori estimates for ground states of $\wJ_{\omega,\G_\varepsilon^d}$ and $\widetilde{E}_{\G_\varepsilon^d}$ as in \eqref{eq:Jtilde}--\eqref{eq:Etilde}.  The argument being analogous, we develop the discussion of both problems in parallel. 

Since will be needed in the following, let us recall here some well--known facts about the action and energy ground state problems. As for the action, it is readily seen that, if $u\in\widetilde{\NN}_{\omega,\G_\varepsilon^d}$, then
\begin{equation}
	\label{eq:JsuNG}
\wJ_{\omega,\G_\varepsilon^d}(u)=\f\kappa d\|u\|_{L^p(\G_\varepsilon^d)}^p\,,\qquad\kappa=\f12-\f1p\,,
\end{equation}
so that
\[
	\widetilde{\JJ}_{\G_\varepsilon^d}(\omega)=\f\kappa d\inf_{v\in\widetilde{\NN}_{\omega,\G_\varepsilon^d}}\|v\|_{L^p(\G_\varepsilon^d)}^p\,.
\]
Moreover, there is a natural projection of $H^1(\G_\varepsilon^d)$ on $\widetilde{\NN}_{\omega,\G_\varepsilon^d}$, since, setting
\[
\widetilde{\pi}_\omega(u):=\left(\f{d\|u'\|_{L^2(\G_\varepsilon^d)}^2+\omega\|u\|_{L^2(\G_\varepsilon^d)}^2}{\|u\|_{L^p(\G_\varepsilon^d)}^p}\right)^{\f1{p-2}}\,,
\]
it holds $\widetilde{\pi}_\omega(u)u\in\widetilde{\NN}_{\omega,\G_\varepsilon^d}$ for every $u\in H^1(\G_\varepsilon^d)$.

Analogous properties hold true for the action problem on $\R^d$, as for every $u\in\NN_{\omega,\R^d}$ it holds
\[
J_{\omega,\R^d}(u)=\kappa\|u\|_{L^p(\R^d)}^p
\]
and setting, for every $u\in H^1(\R^d)$,
\[
\pi_\omega(u)=\left(\f{\|\nabla u\|_{L^2(\R^d)}^2+\omega\|u\|_{L^2(\R^d)}^2}{\|u\|_{L^p(\R^d)}^p}\right)^{\f1{p-2}}\,,
\]
then $\pi_\omega(u)u\in\NN_{\omega,\R^d}$.

Recall also that, by standard scaling arguments,  for every $p\in(2,2^*)$ and $\omega>0$,
\begin{equation}
	\label{eq:JRdw}
	\JJ_{\R^d}(\omega)=\JJ_{\R^d}(1)\omega^{\f{2p-d(p-2)}{2(p-2)}},\qquad\JJ_{\R^d}(1)>0\,,
\end{equation}whereas for every $p\in\left(2,2+\f4d\right)$ and $\mu>0$
\begin{equation}
	\label{eq:ERdm}
	\EE_{\R^d}(\mu)=\EE_{\R^d}(1)\mu^{\f{2p-d(p-2)}{4-d(p-2)}},\qquad\EE_{\R^d}(1)<0\,.
\end{equation}
We now state the main results of this section in the following two theorems.

\begin{theorem}
	\label{thm:stimeJe}
	For every $p\in(2,2^*)$ and $\omega>0$, there exists $\overline{\varepsilon}>0$ and $C>0$, depending only on $p$, $\omega$ and $d$, such that, for every $\varepsilon\in(0,\overline{\varepsilon})$, if $u\in \widetilde{\NN}_{\omega,\G_\varepsilon^d}$ is a ground state of $\wJ_{\omega,\G_\varepsilon^d}$, then
	\[
	\begin{split}
		\f1C\leq&\,\varepsilon^{d-1}\|u'\|_{L^2(\G_\varepsilon^d)}^2 \leq C\\
		\f1C\leq&\,\varepsilon^{d-1}\|u\|_{L^2(\G_\varepsilon^d)}^2\leq C\\
		\f1C\leq&\,\varepsilon^{d-1}\|u\|_{L^p(\G_\varepsilon^d)}^p\leq C\,.
		\end{split}
	\]
\end{theorem}
\begin{theorem}
	\label{thm:stimeGe}
	For every $p\in\left(2,2+\f4d\right)$ and $\mu>0$, there exists $\overline{\varepsilon}>0$ and $C>0$, depending only on $p$, $\mu$ and $d$, such that, for every $\varepsilon\in(0,\overline{\varepsilon})$, if $u\in H_{\f d{\varepsilon^{d-1}}\mu}^1(\G_\varepsilon^d)$ is a ground state of $\widetilde{E}_{\G_\varepsilon^d}$ at mass $\f{d}{\varepsilon^{d-1}}\mu$, then
	\[
	\begin{split}
		\f1C\leq&\,\varepsilon^{d-1}\|u'\|_{L^2(\G_\varepsilon^d)}^2 \leq C\\
		\f1C\leq&\,\varepsilon^{d-1}\|u\|_{L^p(\G_\varepsilon^d)}^p\leq C\,.
	\end{split}
\]
\end{theorem}

\begin{remark}
	\label{rem:GetoG1}
	To prove Theorems \ref{thm:stimeJe}--\ref{thm:stimeGe} we exploit the relation between the action and energy ground state problems on grids with different edgelength. Recall that, given a metric graph $\G$ and a function $u\in H^1(\G)$, setting for every $t>0$
	\begin{equation}
		\label{eq:scaling}
		\widehat{u}(x):=t^\alpha u(t^\beta x),\qquad \alpha=\f2{6-p},\quad\beta=\f{p-2}{6-p}\,,
	\end{equation}
	we have $\widehat{u}\in H^1(t^{-\beta} \G)$,
	\[
	\int_{t^{-\beta}\G}|\widehat{u}'|^2\,dx=t^{2\beta+1}\int_\G|u'|^2\,dx,\qquad\int_{t^{-\beta}\G}|\widehat{u}|^p\,dx=t^{2\beta+1}\int_\G|u|^p\,dx,\qquad\int_{t^{-\beta}\G}|\widehat{u}|^2\,dx=t\int_{\G}|u|^2\,dx\,.
	\]
	Hence, taking $\varepsilon>0$, $\G=\G_\varepsilon^d$ and $t=\varepsilon^{\f1\beta}$, it follows that $u\in H^1(\G_\varepsilon^d)$ if and only if $\widehat{u}\in H^1(\G_1^d)$ and
	\[
		\begin{split}
			\|\widehat{u}\|_{L^2(\G_1^d)}^2&\,=\varepsilon^{\f1\beta}\|u\|_{L^2(\G_\varepsilon^d)}^2\\
			\|\widehat{u}'\|_{L^2(\G_1^d)}^2&\,=\varepsilon^{2+\f1\beta}\|u'\|_{L^2(\G_\varepsilon^d)}^2\\ 
			\|\widehat{u}\|_{L^p(\G_1^d)}^p&\,=\varepsilon^{2+\f1\beta}\|u\|_{L^p(\G_\varepsilon^d)}^p\,,
		\end{split}
	\]
	As for the action problem, this shows that $u\in\widetilde{\NN}_{\omega,\G_\varepsilon^d}$ if and only if $\widehat{u}\in \widetilde{\NN}_{\varepsilon^2\omega,\G_1^d}$, and
	\begin{equation}
		\label{eq:scalJ}
		\wJ_{\varepsilon^2\omega,\G_1^d}(\widehat{u})=\varepsilon^{2+\f1\beta}\wJ_{\omega,\G_\varepsilon^d}(u)\,.
	\end{equation}
	As for the energy problem,  we obtain that $u\in H_{\f{d}{\varepsilon^{d-1}}\mu}^1(\G_\varepsilon^d)$ if and only if $\widehat{u}\in H_{d\varepsilon^{\f1\beta+1-d}\mu}^1(\G_1^d)$, and
	\begin{equation}
		\label{Etildeeps1}
		\widetilde{E}_{\G_1^d}(\widehat{u})=\varepsilon^{2+\f1\beta}\widetilde{E}_{\G_\varepsilon^d}\left(u\right)\,.
	\end{equation}
	Note that
	\[
		p\in\left(2,2+\f4d\right)\qquad\Longrightarrow\qquad \f1\beta+1-d=\f4{p-2}-d>0 \,,
	\]
	so that $\varepsilon^{\f1\beta+1-d}\to0$ as $\varepsilon\to0$ for every $2<p<2+\f4d$. 
\end{remark}

In view of Remark \ref{rem:GetoG1}, Theorems \ref{thm:stimeJe}--\ref{thm:stimeGe} are a direct consequence of the next two propositions. 
\begin{proposition}
	\label{prop:JG1}
	For every $p\in(2,2^*)$, there exists $\overline{\omega}>0$ and $C>0$, depending only on $p$ and $d$, such that, for every $\omega\in(0,\overline{\omega})$, if $u\in \widetilde{\NN}_{\omega,\G_1^d}$ is a ground state of $\wJ_{\omega,\G_1^d}$, then 
	\begin{align}
	\f1C\omega^{\f {2p-d(p-2)}{2(p-2)}}\leq \|u'\|_{L^2(\G_1^d)}^2\leq C\omega^{\f {2p-d(p-2)}{2(p-2)}}\label{eq:u'J1}\\
	\f1C\omega^{\f {4-d(p-2)}{2(p-2)}}\leq \|u\|_{L^2(\G_1^d)}^2\leq C\omega^{\f {4-d(p-2)}{2(p-2)}}\label{eq:uL2J1}\\
	\f1C\omega^{\f {2p-d(p-2)}{2(p-2)}}\leq \|u\|_{L^p(\G_1^d)}^p\leq C\omega^{\f {2p-d(p-2)}{2(p-2)}}\label{eq:uLpJ1}\,.
	\end{align}
\end{proposition}

\begin{proposition}
	\label{prop:estG1}
	For every $p\in\left(2,2+\f4d\right)$, there exists $\overline{m}>0$ and $C>0$, depending only on $p$ and $d$,  such that, for every $m\in(0,\overline{m})$, if $u\in H_m^1(\G_1^d)$ is a ground state of $\widetilde{E}_{\G_1^d}$ at mass $m$, then 
	\begin{align}
		\f1C m^{\f{2p-d(p-2)}{4-d(p-2)}}\leq &\,\|u'\|_{L^2(\G_1^d)}^2 \leq C m^{\f{2p-d(p-2)}{4-d(p-2)}}\label{eq:u'G1}\\
		\f1C m^{\f{2p-d(p-2)}{4-d(p-2)}}\leq &\,\|u\|_{L^p(\G_1^d)}^p \leq C m^{\f{2p-d(p-2)}{4-d(p-2)}}\,.\label{eq:uLpG1}
		\end{align} 
\end{proposition}

The proof of Propositions \ref{prop:JG1}--\ref{prop:estG1} makes use of the following preliminary estimates.
\begin{lemma}
	\label{lem:upest}
	For every $p\in(2,2^*)$ and $\omega>0$, there exists $C>0$, depending only on $p$, $\omega$ and $d$, such that
	\[
	\varepsilon^{d-1}\widetilde{\JJ}_{\G_\varepsilon^d}(\omega)\leq \JJ_{\R^d}(\omega)+ C\varepsilon\qquad\text{as }\varepsilon\to0\,.
	\]
	Furthermore, for every $p\in\left(2,2+\f4d\right)$ and $\mu>0$, there exists $C'>0$, depending only on $p$, $\mu$ and $d$, such that
	\[
	\varepsilon^{d-1}\widetilde{\EE}_{\G_\varepsilon^d}\left(\f d{\varepsilon^{d-1}}\mu\right)\leq \EE_{\R^d}(\mu)+ C'\varepsilon\qquad\text{as }\varepsilon\to0\,.
	\]
\end{lemma}
\begin{proof}
	It is a direct consequence of Lemma \ref{lem:restr} applied to the ground states of $J_{\omega,\R^d}$ and $E_{\R_d}$. 
	
	If $u\in\NN_{\omega,\R^d}$ is a ground state of $J_{\omega,\R^d}$, which is well--known to be in $C^1(\R^d)\cap H^2(\R^d)\cap L^\infty(\R^d)$, letting $u_\varepsilon\in H^1(\G_\varepsilon^d)$ be its restriction to $\G_\varepsilon^d$ and $v_\varepsilon:=\widetilde{\pi}_\omega(u_\varepsilon)u_\varepsilon$, we have $v_\varepsilon\in\widetilde{\NN}_{\omega,\G_\varepsilon^d}$ and, by Lemma \ref{lem:restr},
	\[
	\begin{split}
	\varepsilon^{d-1}\widetilde{\JJ}_{\G_\varepsilon^d}(\omega)\leq&\varepsilon^{d-1} \widetilde{J}_{\omega,\G_\varepsilon}(v_\varepsilon)=\kappa\f{\varepsilon^{d-1}}d\|v_\varepsilon\|_{L^p(\G_\varepsilon^d)}^p=\kappa\f{\varepsilon^{d-1}}d\widetilde{\pi}_{\omega}^{p}(u_\varepsilon)\|u_\varepsilon\|_{L^p(\G_\varepsilon^d)}^p\\
	\leq&\kappa\left(\f{\|\nabla u\|_{L^2(\R^2)}^2+\omega\|u\|_{L^2(\R^d)}^2+C\varepsilon}{\|u\|_{L^p(\R^d)}^p-C\varepsilon}\right)^{\f p{p-2}}\left(\|u\|_{L^p(\R^d)}^p+C\varepsilon\right)\\
	\leq &\kappa(1+C\varepsilon)\|u\|_{L^p(\R^d)}^p+C\varepsilon=\JJ_{\R^d}(\omega) +C\varepsilon
	\end{split}
	\]
	for suitable $C>0$ (not relabeled) and every $\varepsilon$ small enough. 
	
	Analogously, if $u\in \Hmu(\R^d)$ is a ground state of $E_{\R^d}$ at mass $\mu$, letting $u_\varepsilon\in H^1(\G_\varepsilon^d)$ be its restriction to $\G_\varepsilon^d$ and $v_\varepsilon:=\sqrt{\f d{\varepsilon^{d-1}\|u_\varepsilon\|_{L^2(\G_\varepsilon^d)}^2}\mu}\,u_\varepsilon$, we obtain $v_\varepsilon\in H_{\f d{\varepsilon^{d-1}}\mu}^1(\G_\varepsilon^d)$ and
	\[
	\varepsilon^{d-1}\widetilde{\EE}_{\G_\varepsilon^d}\left(\f d{\varepsilon^{d-1}}\mu\right)\leq \varepsilon^{d-1}\widetilde{E}_{\G_\varepsilon^d}(v_\varepsilon)\leq E_{\R^d}(u)+C\varepsilon=\EE_{\R^d}(\mu)+C\varepsilon
	\]
	for sufficiently small $\varepsilon$.
\end{proof}

\begin{proof}[Proof of Proposition \ref{prop:JG1}]
	Note first that, by \eqref{eq:scalJ} and Lemma \ref{lem:upest}, for every $\omega>0$ small enough
	\[
	\omega^{\f12\left(d-3-\f1\beta\right)}\widetilde{\JJ}_{\G_1^d}(\omega)=\omega^{\f{d-1}2}\widetilde{\JJ}_{\G_{\sqrt{\omega}}^d}(1)\leq \JJ_{\R^d}(1)+o(1)\,,
	\]
	so that
	\[
		\widetilde{\JJ}_{\G_1^d}(\omega)\lesssim_{d,p} \omega^{-\f12\left(d-3-\f1\beta\right)}=\omega^{\f{2p-d(p-2)}{2(p-2)}}\,.
	\]
If $u\in\widetilde{\NN}_{\omega,\G_1^d}$ is a ground state of $\wJ_{\omega,\G_1^d}$, this immediately proves the upper bound in \eqref{eq:uLpJ1} by \eqref{eq:JsuNG}, and by definition of $\widetilde{\NN}_{\omega,\G_1^d}$
\[
\begin{split}
	&\|u'\|_{L^2(\G_1^d)}^2<\f{d\|u'\|_{L^2(\G_1^d)}^2+\omega\|u\|_{L^2(\G_1^d)}^2}d=\f{\|u\|_{L^p(\G_1^d)}^p}d\lesssim_{d,p}\omega^{\f{2p-d(p-2)}{2(p-2)}}\\
	&\|u\|_{L^2(\G_1^d)}^2\leq \|u\|_{L^2(\G_1^d)}^2+\f{d\|u'\|_{L^2(\G_1^d)}^2}\omega=\f{\|u\|_{L^p(\G_1^d)}^p}\omega\lesssim_{d,p}\omega^{\f{4-d(p-2)}{2(p-2)}}\,,
\end{split}
\]
i.e. the upper bounds in \eqref{eq:u'J1} and \eqref{eq:uL2J1}. 

To prove the lower bounds, let us distinguish the cases $p\in\left(2,2+\f4d\right)$ and $p\in\left[2+\f4d,2^*\right)$. If $p\in\left(2,2+\f4d\right)$, by $u\in\widetilde{\NN}_{\omega,\G_1^d}$ and the $d$--dimensional Gagliardo--Nirenberg inequality \eqref{eq:dgnGe} 
\begin{equation}
	\label{dis1}
\|u'\|_{L^2(\G_1^d)}^2\leq \f{\|u\|_{L^p(\G_1^d)}^{p}}d\lesssim_{d,p}\|u\|_{L^2(\G_1^d)}^{d+(2-d)\f p2}\|u'\|_{L^2(\G_1^d)}^{\left(\f p2-1\right)d}\,,
\end{equation}
that is
\begin{equation}
	\label{mass1}
	\|u'\|_{L^2(\G_1^d)}^{\left(\f p2-1\right)d}\lesssim_{d,p}\|u\|_{L^2(\G_1^d)}^{\f{(p-2)d}{4-d(p-2)}\f{2p-d(p-2)}{2}}\,.
\end{equation}
Coupling \eqref{mass1} again with the definition of Nehari manifold and \eqref{eq:dgnGe} then yields
\[
\omega\|u\|_{L^2(\G_1^d)}^2\leq\|u\|_{L^1(\G_1^d)}^p\lesssim_{d,p}\|u\|_{L^2(\G_1^d)}^{d+(2-d)\f p2}\|u'\|_{L^2(\G_1^d)}^{\left(\f p2-1\right)d}\lesssim_{d,p}\|u\|_{L^2(\G_1^d)}^{d+(2-d)\f p2+\f{(p-2)d}{4-d(p-2)}\f{2p-d(p-2)}{2}}\,,
\]
that is the lower bound in \eqref{eq:uL2J1}. As a consequence, 
\[
\|u\|_{L^p(\G_1^d)}^p\geq\omega\|u\|_{L^2(\G_1^d)}^2\geq\omega^{\f{2p-d(p-2)}{2(p-2)}}
\]
and thus, together with \eqref{eq:dgnGe} and the upper bound in \eqref{eq:uL2J1},
\[
\omega^{\f{2p-d(p-2)}{2(p-2)}}\lesssim_{d,p}\|u\|_{L^2(\G_1^d)}^{d+(2-d)\f p2}\|u'\|_{L^2(\G_1^d)}^{\left(\f p2-1\right)d}\lesssim_{d,p}\omega^{\f{d+(2-d)\f p2}2\f{4-d(p-2)}{2(p-2)}}\|u'\|_{L^2(\G_1^d)}^{\left(\f p2-1\right)d}\,,
\]
completing the proof of the lower bounds in \eqref{eq:uLpJ1} and \eqref{eq:u'J1} respectively.

If $p\in\left[2+\f4d,2^*\right)$, then $\left(\f p2-1\right)d\geq2$, so that \eqref{dis1} and the upper bound in \eqref{eq:u'J1} imply
\[
\|u\|_{L^2(\G_1^d)}^{d+(2-d)\f p2}\gtrsim_{d,p} \|u'\|_{L^2(\G_1^d)}^{2-\left(\f p2-1\right)d}\gtrsim_{d,p} \omega^{\left(2-\left(\f p2-1\right)d\right)\f{2p-d(p-2)}{2(p-2)}}\,,
\]
which is once again the lower bound in \eqref{eq:uL2J1}. Arguing as above provides the lower bounds in \eqref{eq:u'J1} and \eqref{eq:uLpJ1} and concludes the proof.
\end{proof}

\begin{proof}[Proof of Proposition \ref{prop:estG1}]
	By \eqref{Etildeeps1} and Lemma \ref{lem:upest},
	\[
	m^{\f{\beta(d-3)-1}{1+\beta(1-d)}}\widetilde{\EE}_{\G_1^d}(m)=m^{(d-1)\f{\beta}{1+\beta(1-d)}}\widetilde{\EE}_{\G_{m^{\beta/(1+\beta(1-d))}}^d}\left(\f1{m^{(d-1)\f{\beta}{1+\beta(1-d)}}}\right)\leq\EE_{\R^d}\left(\f1d\right)+o(1)
	\]
	for sufficiently small $m>0$, so that (recall $\beta$ is as in \eqref{eq:scaling})
	\[
	\widetilde{\EE}_{\G_1^d}(m)\lesssim_{d,p}-m^{\f{2p-d(p-2)}{4-d(p-2)}}\,.
	\]
	If $u\in H_m^1(\G_1^d)$ is a ground state of $\widetilde{E}_{\G_1^d}$ at mass $m$, this means 
	\[
	\|u\|_{L^p(\G_1^d)}^p=\f{dp}2\|u'\|_{L^2(\G_1^d)}^2-dp\widetilde{\EE}_{\G_1^d}(m)\gtrsim_{d,p} m^{\f{2p-d(p-2)}{4-d(p-2)}}\,,
	\]
	i.e. the lower bound in \eqref{eq:uLpG1}, and coupling with the $d$--dimensional Gagliardo--Nirenberg inequality \eqref{eq:dgnGe} leads to
	\[
	m^{\f{2p-d(p-2)}{4-d(p-2)}}\lesssim_{d,p}m^{\f d2+(2-d)\f p4}\|u'\|_{L^2(\G_1^d)}^{\left(\f p2-1\right)d}\,,
	\]
	that is the lower bound in \eqref{eq:u'G1}. Moreover, the negativity of $\widetilde{E}_{\G_1^d}(u)$ and \eqref{eq:dgnGe} imply
	\[
	\|u'\|_{L^2(\G_1^d)}^2\leq\f2{dp}\|u\|_{L^p(\G_1^d)}^p\lesssim_{d,p}m^{\f d2+(2-d)\f p4}\|u'\|_{L^2(\G_1^d)}^{\left(\f p2-1\right)d}
	\]
	and rearranging terms yields the upper bound in \eqref{eq:u'G1}, that combined with \eqref{eq:dgnGe} again gives also the upper bound in \eqref{eq:uLpG1}.
\end{proof}
\begin{proof}[Proof of Theorems \ref{thm:stimeJe}--\ref{thm:stimeGe}]
	The desired estimates follow immediately by the combination of those in Propositions \ref{prop:JG1}--\ref{prop:estG1} and the scaling argument described in Remark \ref{rem:GetoG1}.
\end{proof}

\section{Convergence of ground states: proof of Theorems \ref{thm:action}--\ref{thm:2dsquare}  and Proposition \ref{prop:mult}}
\label{sec:proof}

Prior to prove our convergence results for ground states, we need the following lemma.
\begin{lemma}
	\label{lem:u->Au}
	Let $d\geq 2$ and $p\in(2,2^*)$ be fixed and, for every $\varepsilon>0$, let $u_\varepsilon\in H^1(\G_\varepsilon^d)$ be such that
	\begin{equation}
	\label{bound}
	\f 1C\leq\varepsilon^{d-1}\|u_\varepsilon\|_{L^2(\G_\varepsilon^d)}^2,\,\varepsilon^{d-1}\|u_\varepsilon'\|_{L^2(\G_\varepsilon^d)}^2,\,\varepsilon^{d-1}\|u_\varepsilon\|_{L^p(\G_\varepsilon^d)}^p\leq C
	\end{equation}
	for a suitable $C>0$ independent of $\varepsilon$. Then as $\varepsilon\to0$
	\begin{itemize}
		\item[(i)] there exists $C'>0$, independent of $\varepsilon$, such that
		\begin{equation}
		\label{eq:uAuL2}
		\left|\f{\varepsilon^{d-1}}d\|u_\varepsilon\|_{L^2(\G_\varepsilon^d)}^2-\|\mathcal{A}u_\varepsilon\|_{L^2(\R^d)}^2\right|\leq C'\varepsilon\,;
		\end{equation}
		
		\item[(ii)] if $d=2$ and $p>2$, or $d\geq3$ and $p\in\left(2,\f{2^*}2+1\right]$, there exists $C''>0$, depending only on $d$ and $p$, such that
		\begin{equation}
		\label{eq:uAuLp1}
		\left|\f{\varepsilon^{d-1}}d\|u_\varepsilon\|_{L^p(\G_\varepsilon^d)}^p-\|\mathcal{A}u_\varepsilon\|_{L^p(\R^d)}^p\right|\leq C''\varepsilon\,;
		\end{equation} 
		
		\item[(iii)] if $d\geq3$ and $p\in\left(\f{2^*}2+1,2^*\right)$, then for every $\gamma>0$ there exists $C'''>0$, depending only on $d$, $p$ and $\gamma$, such that
		\begin{equation}
		\label{eq:uAuLp2}
		\left|\f{\varepsilon^{d-1}}d\|u_\varepsilon\|_{L^p(\G_\varepsilon^d)}^p-\|\mathcal{A}u_\varepsilon\|_{L^p(\R^d)}^p\right|\leq C'''\varepsilon^{\f{d-2}2(2^*-p)-\gamma}\,.
		\end{equation} 
	\end{itemize}
\end{lemma}
\begin{proof}
	Recall that, as in Section \ref{sec:prel}, we always denote by $\widetilde{u}_\varepsilon$ the restriction of $\mathcal{A}u_\varepsilon$ to $\G_\varepsilon^d$. By \eqref{bound} and Lemma \ref{lem:uutL2} we obtain immediately
	\begin{equation}
	\label{pezzo1}
	\left|\f{\varepsilon^{d-1}}d\|u_\varepsilon\|_{L^2(\G_\varepsilon^d)}^2-\f{\varepsilon^{d-1}}d\|\widetilde{u}_\varepsilon\|_{L^2(\G_\varepsilon^d)}^2\right|\lesssim \varepsilon
	\end{equation}
	for sufficiently small $\varepsilon$. Moreover, by \eqref{bound} and Lemma \ref{lem:normA} it follows that $\left(\mathcal{A}u_\varepsilon\right)_\varepsilon$ is uniformly bounded in $H^1(\R^d)$, so that arguing as in Part 1 of the proof of Lemma \ref{lem:restr} yields
	\[
	\left|\f{\varepsilon^{d-1}}d\|\widetilde{u}_\varepsilon\|_{L^2(\G_\varepsilon^d)}^2-\|\mathcal{A}u_\varepsilon\|_{L^2(\R^d)}^2\right|\lesssim \varepsilon
	\]
	as $\varepsilon\to0$. Coupling with \eqref{pezzo1} proves \eqref{eq:uAuL2}.

	Observe that, applying to $\mathcal{A}u_\varepsilon$ the argument in Part 1 of the proof of Lemma \ref{lem:restr} up to \eqref{stimachiave}, we have
	\begin{equation}
	\label{utAuLp}
	\left|\f{\varepsilon^{d-1}}d\|\widetilde{u}_\varepsilon\|_{L^p(\G_\varepsilon^d)}^p-\|\mathcal{A}u_\varepsilon\|_{L^p(\R^d)}^p\right|\lesssim_p \varepsilon\|\mathcal{A}u_\varepsilon\|_{L^{2(p-1)}(\R^d)}^{p-1}\|\nabla\mathcal{A} u_\varepsilon\|_{L^2(\R^d)}\,.
	\end{equation}
	If $d=2$ and $p>2$, or $d\geq3$ and $p\in\left(2,\f{2^*}2+1\right]$, the boundedness in $H^1(\R^d)$ of $\left(\mathcal{A}u_\varepsilon\right)_\varepsilon$ ensures also the boundedness in $L^{2(p-1)}(\R^d)$, so that the previous estimate becomes
	\[
	\left|\f{\varepsilon^{d-1}}d\|\widetilde{u}_\varepsilon\|_{L^p(\G_\varepsilon^d)}^p-\|\mathcal{A}u_\varepsilon\|_{L^p(\R^d)}^p\right|\lesssim_p \varepsilon\qquad\text{as }\varepsilon\to0\,.
	\]
	Moreover, Lemma \ref{lem:uutLq}(i) and \eqref{bound} imply
	\[
	\left|\f{\varepsilon^{d-1}}d\|u_\varepsilon\|_{L^p(\G_\varepsilon^d)}^p-\f{\varepsilon^{d-1}}d\|\widetilde{u}_\varepsilon\|_{L^p(\G_\varepsilon^d)}^p\right|\lesssim_{d,p}\varepsilon
	\]
	provided $\varepsilon$ is small enough, which together with the previous estimates gives \eqref{eq:uAuLp1}.
	
	If $d\geq 3$ and $p\in\left(\f{2^*}2+1,2^*\right)$, combining \eqref{bound} with Lemma \ref{lem:uutLq}(ii) entails
	\begin{equation}
	\label{pezzo2}
	\left|\f{\varepsilon^{d-1}}d\|u_\varepsilon\|_{L^p(\G_\varepsilon^d)}^p-\f{\varepsilon^{d-1}}d\|\widetilde{u}_\varepsilon\|_{L^p(\G_\varepsilon^d)}^p\right|\lesssim_{d,p,\gamma}\varepsilon^{\f{d-2}2(2^*-p)-\gamma}
	\end{equation}
	for every $\gamma\in(0,1]$ and $\varepsilon$ sufficiently small. Furthermore, by Lemma \ref{lem:normA}, Proposition \ref{prop:GNint} (with $q=2(p-1)$ and $\alpha=\gamma$) and \eqref{bound}, we get
	\[
	\begin{split}
	\|\mathcal{A}u_\varepsilon\|_{L^{2(p-1)}(\R^d)}^{2(p-1)}&\lesssim_{d,p,\gamma} \varepsilon^{d-1}\left(\|u_\varepsilon\|_{L^{2(p-1)}(\G_\varepsilon^d)}^{2(p-1)}+\varepsilon^{p-\gamma}\|u_\varepsilon\|_{L^2(\G_\varepsilon^d)}^\gamma\|u_\varepsilon'\|_{L^2(\G_\varepsilon^d)}^{2(p-1)-\gamma}\right)\\
	&\lesssim_{d,p,\gamma}\varepsilon^{d-1}\varepsilon^{p-\gamma}\|u_\varepsilon\|_{L^2(\G_\varepsilon^d)}^\gamma\|u_\varepsilon'\|_{L^2(\G_\varepsilon^d)}^{2(p-1)-\gamma}\lesssim_{d,p,\gamma}\varepsilon^{p-(d-1)(p-2)-\gamma}\,,
	\end{split}
	\]
	and, plugging into \eqref{utAuLp},
	\[
	\left|\f{\varepsilon^{d-1}}d\|\widetilde{u}_\varepsilon\|_{L^p(\G_\varepsilon^d)}^p-\|\mathcal{A}u_\varepsilon\|_{L^p(\R^d)}^p\right|\lesssim_{d,p,\gamma} \varepsilon^{1+\f{p-(d-1)(p-2)-\gamma}2}=\varepsilon^{\f{d-2}2(2^*-p)-\f\gamma2}\,.
	\]
	Coupling with \eqref{pezzo2} implies \eqref{eq:uAuLp2} and completes the proof of the lemma.
\end{proof}

We are now in position to prove Theorems \ref{thm:action}--\ref{thm:2dsquare}.

\begin{proof}[Proof of Theorem \ref{thm:action}]
	Fix $p\in(2,2^*)$ and $\omega>0$. For every $\varepsilon>0$, let $u_\varepsilon\in\widetilde{\NN}_{\omega,\G_\varepsilon^d}$ be a positive ground state of $\widetilde{J}_{\omega,\G_\varepsilon^d}$ satisfying
	\begin{equation}
		\label{eq:cubi}
		\sup_{j\in\N}\|u_\varepsilon\|_{L^2(Q_j)}=\|u_\varepsilon\|_{L^2(Q_0)}\,,
	\end{equation}
	where $(Q_j)_{j\in\N}\subset\R^d$ is an $\varepsilon$--independent, countable family of open, unitary $d$--dimensional cubes in $\R^d$ such that $Q_i\cap Q_j=\emptyset$ for every $i\neq j$, $\overline{\bigcup_{j\in\N}Q_j}=\R^d$ and $Q_0$ is centered at the origin. It is clear that \eqref{eq:cubi} does not imply any loss of generality, since for every $u_\varepsilon$ there always exists $x_\varepsilon\in\R^d$ for which $u_\varepsilon(\cdot-x_\varepsilon)$ satisfies \eqref{eq:cubi}. Note also that, by Theorem \ref{thm:stimeJe}, Lemma \ref{lem:u->Au} applies to $u_\varepsilon$ as $\varepsilon\to0$. 
	
	Set $v_\varepsilon:=\pi_\omega(\mathcal{A}u_\varepsilon)\mathcal{A}u_\varepsilon$. It then follows that $v_\varepsilon\in\NN_{\omega,\R^d}$ and, by Lemmas \ref{lem:normA}--\ref{lem:u->Au}, if $d=2$ and $p>2$ or $d\geq3$ and $p\in\left(2,\f{2^*}2+1\right]$ it holds
	\begin{equation}
		\label{Jv1}
		J_{\omega,\R^d}(v_\varepsilon)=\kappa\pi_\omega(\mathcal{A}u_\varepsilon)^{p}\|\mathcal{A}u_\varepsilon\|_{L^p(\R^d)}^p\leq\f\kappa d\varepsilon^{d-1}\|u_\varepsilon\|_{L^p(\G_\varepsilon^d)}^p+C\varepsilon=\varepsilon^{d-1}\widetilde{\JJ}_{\G_\varepsilon^d}(\omega)+C\varepsilon\qquad\text{as }\varepsilon\to0,
	\end{equation}
	for a suitable constant $C>0$ depending only on $d$, $p$ and $\omega$, whereas if $d\geq3$ and $p\in\left(\f{2^*}2+1,2^*\right)$ we get for every $\gamma>0$
	\begin{equation}
		\label{Jv2}
		J_{\omega,\R^d}(v_\varepsilon)\leq\varepsilon^{d-1}\widetilde{\JJ}_{\G_\varepsilon^d}(\omega)+C'\varepsilon^{\f{d-2}2(2^*-p)-\gamma}\qquad\text{as }\varepsilon\to0,
	\end{equation}
	for $C'>0$ depending only on $d,p, \gamma$ and $\omega$. Since $\JJ_{\R^d}(\omega)\leq J_{\omega,\R^d}(v_\varepsilon)$ by definition, coupling \eqref{Jv1}--\eqref{Jv2} with Lemma \ref{lem:upest} proves Theorem \ref{thm:action}(i) (noting also that $\f{d-2}2(2^*-p)<1$ for every $p>\f{2^*}2+1$). Furthermore, $(v_\varepsilon)_\varepsilon\subset\NN_{\omega,\R^d}$ is a minimizing sequence for $\JJ_{\R^d}(\omega)$ and, since arguing as in the proof of Lemma \ref{lem:u->Au} it is readily seen that
	\[
	\lim_{\varepsilon\to0}\left|\|v_\varepsilon\|_{L^2(\Omega)}^2-\f{\varepsilon^{d-1}}d\|u_\varepsilon\|_{L^2(\Omega)}^2\right|=0
	\]
	for every given measurable $\Omega\subset\R^d$, by \eqref{eq:cubi} it satisfies 
	\begin{equation}
		\label{eq:vcubi}
		\lim_{\varepsilon\to0}\left|\|v_\varepsilon\|_{L^2(Q_0)}^2-\sup_{j\in\N}\|v_\varepsilon\|_{L^2(Q_j)}^2\right|=0\,.
	\end{equation}
	Since $\|v_\varepsilon-\mathcal{A}u_\varepsilon\|_{H^1(\R^d)}=o(1)$ as $\varepsilon\to0$, we are left to show that $v_\varepsilon$ converges strongly in $H^1(\R^d)$ to the ground state $\varphi_\omega$ of $J_{\omega,\R^d}$ in $\NN_{\omega,\R^d}$ attaining its $L^\infty$ norm at the origin. That such a convergence holds is a classic result, but for the sake of completeness we report here the details of the proof.
	
	As $(v_\varepsilon)_\varepsilon$ is bounded in $H^1(\R^d)$, there exists $v\in H^1(\R^d)$ such that, up to subsequences, $v_\varepsilon\rightharpoonup v$ in $H^1(\R^d)$ for $\varepsilon\to0$. Let $m:=\|v\|_{L^2(\R^d)}^2$, so that by semicontinuity $m\leq\liminf_{\varepsilon\to0}\|v_\varepsilon\|_{L^2(\R^d)}^2$. Observe that, if $m=\lim_{\varepsilon\to0}\|v_\varepsilon\|_{L^2(\R^d)}^2$, then the convergence of $v_\varepsilon$ to $v$ is strong in $L^2(\R^d)$ and thus, by Gagliardo--Nirenberg inequalities \eqref{eq:GNd}, strong in $L^p(\R^d)$. In particular, by semicontinuity again and $v_\varepsilon\in\NN_{\omega,\R^d}$,
	\[
	\pi_\omega(v)=\f{\|\nabla v\|_{L^2(\R^d)}^2+\omega\|v\|_{L^2(\R^d)}^2}{\|v\|_{L^p(\R^d)}^p}\leq \liminf_{\varepsilon\to0}	\f{\|\nabla v_\varepsilon\|_{L^2(\R^d)}^2+\omega\|v_\varepsilon\|_{L^2(\R^d)}^2}{\|v_\varepsilon\|_{L^p(\R^d)}^p}=1\,,
	\]
	so that
	\[
	\JJ_{\R^d}(\omega)\leq J_{\omega,\R^d}(\pi_\omega(v)v)\leq\kappa\pi_\omega(v)^p\|v\|_{L^p(\R^d)}^p\leq\kappa\lim_{\varepsilon\to0}\|v_\varepsilon\|_{L^p(\R^d)}^p=\lim_{\varepsilon\to0}J_{\omega,\R^d}(v_\varepsilon)=\JJ_{\R^d}(\omega)\,,
	\]
	that is $\pi_\omega(v)=1$, i.e. $v\in\NN_{\omega,\R^d}$ and it is a ground state of $J_{\omega,\R^d}$. By \eqref{eq:vcubi} and the fact that $Q_0$ is centered at the origin, $v$ is thus the unique ground state $\varphi_\omega$ attaining its $L^\infty$ norm at the origin, the convergence of $v_\varepsilon$ is not up to subsequences (by uniqueness of the limit) and it is strong in $H^1(\R^d)$.
	
	To complete the proof it is then enough to rule out the possibility that $m<\liminf_{\varepsilon\to0}\|v_\varepsilon\|_{L^2(\R^d)}^2$. To this end, let us distinguish the cases $m>0$ and $m=0$. 
	
	If $0<m<\liminf_{\varepsilon\to0}\|v_\varepsilon\|_{L^2(\R^d)}^2$, then $\liminf_{\varepsilon\to0}\|v_\varepsilon-v\|_{L^2(\R^d)}^2>0$. Let $\theta, \theta_\varepsilon\in\R$ be such that $v\in\NN_{\theta,\R^d}$, $v_\varepsilon-v\in\NN_{\theta_\varepsilon,\R^d}$ for every $\varepsilon$. By the weak convergence of $v_\varepsilon$ to $v$ in $H^1(\R^d)$, Brezis--Lieb Lemma \cite{BL} and $v_\varepsilon\in\NN_{\omega,\R^d}$, it follows
	\[
	\begin{split}
	\theta_\varepsilon=&\f{\|v_\varepsilon-v\|_{L^p(\R^d)}^p-\|\nabla v_\varepsilon-\nabla v\|_{L^2(\R^d)}^2}{\|v_\varepsilon-v\|_{L^2(\R^d)}^2}=\f{\|v_\varepsilon\|_{L^p(\R^d)}^p-\|v\|_{L^p(\R^d)}^p-\|\nabla v_\varepsilon\|_{L^2(\R^d)}^2+\|\nabla v\|_{L^2(\R^d)}^2+o(1)}{\|v_\varepsilon-v\|_{L^2(\R^d)}^2}\\
	=&\f{\omega\|v_\varepsilon\|_{L^2(\R^d)}^2-\theta\|v\|_{L^2(\R^d)}^2+o(1)}{\|v_\varepsilon-v\|_{L^2(\R^d)}^2}=\omega+\f{\|v\|_{L^2(\R^d)}^2}{\|v_\varepsilon-v\|_{L^2(\R^d)}^2}(\omega-\theta)+o(1)\qquad\text{as }\varepsilon\to0\,,
	\end{split}
	\]
	which implies that either $\theta>\omega$ or $\liminf_{\varepsilon\to0}\theta_\varepsilon>\omega$ or $\theta=\lim_{\varepsilon\to0}\theta_\varepsilon=\omega$. However, since by Brezis--Lieb Lemma \cite{BL} we also have
	\[
	\JJ_{\R^d}(\omega)=\kappa\lim_{\varepsilon\to0}\|v_\varepsilon\|_{L^p(\R^d)}^p=\kappa\lim_{\varepsilon\to0}\|v_\varepsilon-v\|_{L^p(\R^d)}^p+\kappa\|v\|_{L^p(\R^d)}^p\geq\lim_{\varepsilon\to0}\JJ_{\R^d}(\theta_\varepsilon)+\JJ_{\R^d}(\theta)\,,
	\]
	this provides a contradiction, as $\JJ_{\R^d}$ is nonnegative on $\R$ and strictly increasing on $\R^+$ by \eqref{eq:JRdw}.
	
	Assume then by contradiction that $m=0$, i.e. $v\equiv0$ on $\R^d$. Since the convergence of $v_\varepsilon$ to $v$ is locally strong in $L^2$, this implies that $\|v_\varepsilon\|_{L^2(Q_0)}\to0$ as $\varepsilon\to0$, so that by \eqref{eq:vcubi}
	\begin{equation}
		\label{eq:vto0}
	\lim_{\varepsilon\to0}\sup_{j\in\N}\|v_\varepsilon\|_{L^2(Q_j)}=0.
	\end{equation}
	Recall now that, by standard Gagliardo--Nirenberg inequalities on bounded subsets of $\R^d$, we have for every $j\in\N$
	\[
	\|v_\varepsilon\|_{L^{2+\f4d}(Q_j)}^{2+\f4d}\leq C\|v_\varepsilon\|_{L^2(Q_j)}^{\f4d}\|\nabla v_\varepsilon\|_{L^2(Q_j)}^2
	\]
	for a suitable $C>0$ independent of $j$ and $\varepsilon$. Summing over $j$ we obtain
	\[
	\|v_\varepsilon\|_{L^{2+\f4d}(\R^d)}^{2+\f4d}=\sum_{j\in\N}\|v_\varepsilon\|_{L^{2+\f4d}(Q_j)}^{2+\f4d}\leq C\sum_{j\in\N}\|v_\varepsilon\|_{L^2(Q_j)}^{\f4d}\|\nabla v_\varepsilon\|_{L^2(Q_j)}^2\leq C\sup_{j\in\N}\|v_\varepsilon\|_{L^{2+\f4d}(Q_j)}^{\f4d}\|\nabla v_\varepsilon\|_{L^2(\R^d)}^2
	\]
	and, combining with \eqref{eq:vto0},
	\begin{equation}
		\label{eq:pcritto0}
		\lim_{\varepsilon\to0}\|v_\varepsilon\|_{L^{2+\f4d}(\R^d)}=0\,.
	\end{equation}
	Since $(v_\varepsilon)_\varepsilon$ is uniformly bounded both in $L^2(\R^d)$ and in $L^{2^*}(\R^d)$, it then follows that $v_\varepsilon\to0$ strongly in $L^p(\R^d)$ when $\varepsilon\to0$. As this provides again a contradiction, given that $v_\varepsilon\in\NN_{\omega,\R^d}$ and thus $\|v_\varepsilon\|_{L^p(\R^d)}^p\gtrsim_{p}\JJ_{\R^d}(\omega)>0$ for every $\varepsilon$, we conclude.
\end{proof}

\begin{proof}[Proof of Theorem \ref{thm:2dsquare}]
	Let $p\in\left(2,2+\f4d\right)$ and $\mu>0$ be fixed. For every $\varepsilon>0$, let $u_\varepsilon\in H_{\f{d}{\varepsilon^{d-1}}\mu}^1(\G_\varepsilon^d)$ be a ground state of $\widetilde{E}_{\G_\varepsilon^d}$ with mass $\f d{\varepsilon^{d-1}}\mu$ satisfying \eqref{eq:cubi}. Set then $v_\varepsilon:=\sqrt{\f\mu{\|\mathcal{A}u_\varepsilon\|_{L^2(\R^d)}^2}}\,\mathcal{A}u_\varepsilon$, so that by definition $v_\varepsilon\in H_\mu^1(\R^d)$ for every $\varepsilon$. Since Lemma \ref{lem:u->Au} applies to $u_\varepsilon$ by Theorem \ref{thm:stimeGe}, arguing as in the proof of Theorem \ref{thm:action} it is readily seen that, if $d\in\left\{2,3,4\right\}$, or if $d\geq5$ and $p\in\left(2,\f{2^*}2+1\right]$ (which is strictly contained in $\left(2,2+\f4d\right)$), then 
	\[
	\EE_{\R^d}(\mu)\leq E_{\R_d}(v_\varepsilon)\leq \varepsilon^{d-1}\widetilde{E}_{\G_\varepsilon^d}(u_\varepsilon)+C\varepsilon=\varepsilon^{d-1}\widetilde{\EE}\left(\f d{\varepsilon^{d-1}}\mu\right)+C\varepsilon\qquad\text{as }\varepsilon\to0
	\]
	for a suitable $C>0$ depending only on $d,p$ and $\mu$, whereas if $d\geq 5$ and $p\in\left(\f{2^*}2+1,2^*\right)$, then for every $\gamma>0$ there exists $C>0$, depending only on $d, p,\gamma$ and $\mu$, such that
	\[
	\EE_{\R^d}(\mu)\leq E_{\R_d}(v_\varepsilon)\leq\varepsilon^{d-1}\widetilde{\EE}\left(\f d{\varepsilon^{d-1}}\mu\right)+C\varepsilon^{\f{d-2}2(2^*-p)-\gamma}\qquad\text{as }\varepsilon\to0\,.
	\]
	Since $\f{d-2}2(2^*-p)<1$ for every $p\in\left(\f{2^*}2+1,2^*\right)$, combining with Lemma \ref{lem:upest} proves Theorem \ref{thm:2dsquare}(i). Moreover, $(v_\varepsilon)_\varepsilon\subset H_\mu^1(\R^d)$ is a minimizing sequence for $E_{\R^d}$ which is bounded in $H^1(\R^d)$ and satisfies \eqref{eq:vcubi}. Hence, up to subsequences $v_\varepsilon\rightharpoonup v$ in $H^1(\R^d)$, for some $v\in H^1(\R^d)$. To conclude the proof of Theorem \ref{thm:2dsquare}, it is enough to prove that $v$ is the unique ground state $\phi_\mu$ of $E_{\R^d}$ at mass $\mu$ attaining its $L^\infty$ norm at the origin. To this end, it is sufficient to prove that $\|v\|_{L^2(\R^d)}^2=\mu$, because this implies that $v_\varepsilon$ tends to $v$ strongly in $L^q(\R^d)$ for every $q\in\left[2,2+\f4d\right]$, so that $v\in H_\mu^1(\R^d)$ and by semicontinuity
	\[
	\EE_{\R^d}(\mu)\leq E_{\R^d}(v)\leq\lim_{\varepsilon\to0}E_{\R^d}(v_\varepsilon)=\EE_{\R^d}(\mu)\,,
	\]
	i.e. $v$ is a ground state of $E_{\R^d}$ at mass $\mu$. That it attains its $L^\infty$ norm at the origin is then a consequence of \eqref{eq:vcubi}. Moreover, the convergence of $v_\varepsilon$ is strong in $H^1(\R^d)$ and does not depend on the subsequence (by uniqueness of the limit) and, since
	\begin{equation}
		\label{eq:LconE}
	\mathcal{L}_{\G_\varepsilon^d}(u_\varepsilon)=\left(1-\f2p\right)\f{\|u_\varepsilon\|_{L^p(\G_\varepsilon^d)}^p}{d\|u_\varepsilon\|_{L^2(\G_\varepsilon^d)}^2}-2\f{\widetilde{\EE}_{\G_\varepsilon^d}\left(\f d{\varepsilon^{d-1}}\mu\right)}{\|u_\varepsilon\|_{L^2(\G_\varepsilon^d)}^2}\,,
	\end{equation}
	\eqref{eq:convLu} follows by Theorem \ref{thm:2dsquare}(i), the strong convergence of $v_\varepsilon$ to $\phi_\mu$ in $L^p(\R^d)$, the fact that $\|v_\varepsilon-\mathcal{A}u_\varepsilon\|_{H^1(\R^d)}=o(1)$ as $\varepsilon\to0$ and Lemma \ref{lem:u->Au}.
	
	The argument showing that $m:=\|v\|_{L^2(\R^d)}^2=\mu$ is again classic. Assume first by contradiction that $m\in(0,\mu)$, so that $\liminf_{\varepsilon\to0}\|v_\varepsilon-v\|_{L^2(\R^d)}^2>0$.  Then, by $v_\varepsilon\rightharpoonup v$ in $H^1(\R^d)$, it follows
	\[
	\begin{split}
		\|v_\varepsilon-v\|_{L^2(\R^d)}^2=&\,\|v_\varepsilon\|_{L^2(\R^d)}^2-\|v\|_{L^2(\R^d)}^2+o(1)=\mu-m+o(1)\\
		\|\nabla v_\varepsilon-\nabla v\|_{L^2(\R^d)}^2=&\,\|\nabla v_\varepsilon\|_{L^2(\R^d)}^2-\|\nabla v\|_{L^2(\R^d)}^2+o(1)
	\end{split}
	\]
	and, by Brezis--Lieb Lemma \cite{BL},
	\[
	\|v_\varepsilon-v\|_{L^p(\R^d)}^p=\|v_\varepsilon\|_{L^p(\R^d)}^p-\|v\|_{L^p(\R^d)}^p+o(1)
	\]
	for every $\varepsilon$ small enough, so that
	\begin{equation}
		\label{eq:splitE}
		E_{\R^d}(v_\varepsilon)=E_{\R^d}(v_\varepsilon-v)+E_{\R^d}(v)+o(1)\qquad\text{as }\varepsilon\to0\,.
	\end{equation}
	Since $p>2$ and $m<\mu$,
	\begin{equation*}
		\EE_{\R^d}(\mu)\leq E_{\R^d}\left(\sqrt{\f\mu m}\,v\right)=\f\mu m\|\nabla v\|_{L^2(\R^d)}^2-\left(\f\mu m\right)^{\f p2}\|v\|_{L^p(\R^d)}^p< \f\mu m E_{\R^d}(v)\,,
	\end{equation*}
	in turn yielding
	\begin{equation}
		\label{eq:Ev}
		E_{\R^d}(v)>\f m\mu\EE_{\R^d}(\mu)\,.
	\end{equation}
	Moreover, since $m>0$
	\begin{equation*}
		\EE_{\R^d}(\mu)\leq E_{\R^d}\left(\sqrt{\f \mu {\|v_\varepsilon-v\|_{L^2(\R^d)}^2}}\,(v_\varepsilon-v)\right)< \f \mu {\|v_\varepsilon-v\|_{L^2(\R^d)}^2}E_{\R^d}(v_\varepsilon-v)
	\end{equation*}
	so that 
	\begin{equation}
		\label{eq:Eve-v}
		\liminf_{\varepsilon\to0}E_{\R^d}(v_\varepsilon-v)\geq\f{\mu-m}\mu\EE_{\R^d}(\mu)\,.
	\end{equation}
	Combining \eqref{eq:splitE}, \eqref{eq:Ev}, \eqref{eq:Eve-v} gives
	\[
	\EE_{\R^d}(\mu)=\lim_{\varepsilon\to0} E_{\R^d}(v_\varepsilon)\geq\liminf_{\varepsilon\to0}E_{\R^d}(v_\varepsilon-v)+E_{\R^d}(v)> \f{\mu-m}\mu\EE_{\R^d}(\mu)+\f m\mu\EE_{\R^d}(\mu)=\EE_{\R^d}(\mu)\,,
	\]
	i.e. a contradiction. Hence, either $m=0$ or $m=\mu$.
	
	Suppose then that $m=0$, i.e. $v\equiv0$ on $\R^d$.  Arguing as in the proof of Theorem \ref{thm:action}, it then follows that $(v_\varepsilon)_\varepsilon$ satisfies \eqref{eq:pcritto0}, which together with the boundedness in $L^2(\R^d)$ implies that $v_\varepsilon\to0$ strongly in $L^p(\R^d)$ as $\varepsilon\to0$. By semicontinuity, this leads to
	\[
	\EE_{\R^d}(\mu)=\lim_{\varepsilon\to0}E_{\R^d}(v_\varepsilon)\geq\liminf_{\varepsilon\to0}\f12\|\nabla v_\varepsilon\|_{L^2(\R^d)}^2\geq0\,,
	\]
	which is again a contradiction in view of \eqref{eq:ERdm}.  Therefore, $m=\mu$ and we conclude.
\end{proof}

\begin{remark}
	\label{rem:noequi}
	Theorem \ref{thm:2dsquare} can be used, inter alia, to show that the sublevel sets of the functionals $F_\varepsilon$ defined in Remark \ref{rem:gamma} are not equicoercive with respect to the strong convergence in $H^1(\R^d)$ of piecewise--affine extensions through the operator $\mathcal{A}$. To this end, consider for instance the following construction. Fix $\mu,\mu_1, \mu_2>0$ so that $\mu=\mu_1+\mu_2$ and, for every $\varepsilon>0$, let $u_{\varepsilon,1}\in H_{\f d{\varepsilon^{d-1}}\mu_1}^1(\G_\varepsilon^d)$, $u_{\varepsilon,2}\in H_{\f d{\varepsilon^{d-1}}\mu_2}^1(\G_\varepsilon^d)$ be compactly supported functions on $\G_\varepsilon^d$ satisfying $\widetilde{E}_{\G_\varepsilon^d}(u_{\varepsilon,1})\leq \widetilde{\EE}_{\G_\varepsilon^d}\left(\f d{\varepsilon^{d-1}}\mu_1\right)+\varepsilon$ and $\widetilde{E}_{\G_\varepsilon^d}(u_{\varepsilon,2})\leq \widetilde{\EE}_{\G_\varepsilon}\left(\f d{\varepsilon^{d-1}}\mu_2\right)+\varepsilon$ respectively. Exploiting the periodicity of $\G_\varepsilon^d$, we can then define the function $v_\varepsilon\in H_{\f d{\varepsilon^{d-1}}\mu}^1(\G_\varepsilon^d)$ as the disjoint union of (translated copies of) $u_{\varepsilon,1}$ and $u_{\varepsilon,2}$, in such a way that the $L^\infty$ norm of $u_{\varepsilon,1}$ is always attained inside the neighbourhood of radius $\varepsilon$ of the origin. Roughly, as $\varepsilon\to0$, $v_\varepsilon$ splits into a copy of $u_{\varepsilon,1}$ centered at the origin and a copy of $u_{\varepsilon,2}$ running away at infinity on $\G_\varepsilon^d$. Clearly, $\widetilde{E}_{\G_\varepsilon^d}(v_\varepsilon)= \widetilde{E}_{\G_\varepsilon^d}(u_{\varepsilon,1})+\widetilde{E}_{\G_\varepsilon^d}(u_{\varepsilon,2})=\widetilde{\EE}_{\G_\varepsilon^d}\left(\f d{\varepsilon^{d-1}}\mu_1\right)+\widetilde{\EE}_{\G_\varepsilon^d}\left(\f d{\varepsilon^{d-1}}\mu_2\right)+2\varepsilon$, so that by Theorem \ref{thm:2dsquare} one has $F_\varepsilon(v_\varepsilon)\leq \EE_{\R^d}(\mu_1)+\EE_{\R^d}(\mu_2)+o(1)$ as $\varepsilon$ is small enough. Nevertheless, since Theorem \ref{thm:2dsquare} implies that $\mathcal{A}u_{\varepsilon,i}\to \phi_{\mu_i}$ in $H^1(\R^d)$ for $i=1,2$, it is not possible to extract from $v_\varepsilon$ any subsequence whose corresponding extension $\mathcal{A}v_\varepsilon$ converges strongly in $H^1(\R^d)$.
\end{remark}

Ending this section, we give the proof of Proposition \ref{prop:mult}.

\begin{proof}[Proof of Proposition \ref{prop:mult}]
	
	For every $d\geq2$, $p\in\left[2+\f4d,2^*\right)$ and $\mu$ sufficiently large (depending on $d$ and $p$), existence of ground states of $\widetilde{E}_{\G_1^d}$ in $H_{\mu}^1(\G_1^d)$ can be proved adapting the argument developed when $d=2$ in the proof of  \cite[Theorem 1.2]{ADST}. Moreover,  when $\mu$ is large enough, there exists $w\in H_\mu^1(\R)$, compactly supported in $[0,1]$ and such that $\widetilde{E}_{\R}(w)\leq -\f{C_p}2\mu^{2\beta+1}$ (it is enough to consider compactly supported approximations of the ground state of $\widetilde{E}_{\R}$ at mass $\mu$). Since we can think of $w$ as a function in $H_{\mu}^1(\G_1^d)$ supported on a single edge, this implies
	\[
	\widetilde{\EE}_{\G_1^d}(\mu)\leq \widetilde{E}_{\G_1^d}(w)\leq -\f{C_p}2\mu^{2\beta+1}
	\]  
	for every $\mu$ large enough.  Recalling \eqref{eq:LconE}, it follows that, if $u\in H_\mu^1(\G_1^d)$ is a ground state of $\widetilde{E}_{\G_1^d}$, then
	\begin{equation}
		\label{eq:infty}
		\lim_{n\to+\infty}\mathcal{L}_{\G_1^d}(u)\geq\lim_{\mu\to+\infty}-\f{2\widetilde{\EE}_{\G_1^d}(\mu)}\mu=+\infty\,.
	\end{equation}
	Conversely, if $p\in\left(2+\f4d,2^*\right)$ and $v\in\widetilde{\NN}_{\omega,\G_1^d}$ is a ground state of $\widetilde{J}_{\omega,\G_1^d}$, Proposition \ref{prop:JG1} shows that $\|v\|_{L^2(\G_1^d)}^2\gtrsim_{d,p}\omega^{\f {4-d(p-2)}{2(p-2)}}\to+\infty$ as $\omega\to0^+$.  Therefore, there exist sequences $(\mu_n)_n, (\omega_n)_n\subset\R^+$ and  $(v_n)_n\subset H^1(\G_1^d)$ such that $\mu_n\to+\infty$ and $\omega_n\to0^+$ as $n\to+\infty$ and, for every $n$, $v_n\in\widetilde{\NN}_{\omega_n,\G_1^d}$ is a ground state of $\widetilde{J}_{\omega_n,\G_1^d}$. Since $v_n$ is thus a critical point of $\widetilde{E}_{\G_1^d}$ in $H_{\mu_n}^1(\G_1^d)$ with
	\[
	\mathcal{L}_{\G_1^d}(v_n)=\f{\omega_n}d\to0\qquad\text{as }n\to+\infty\,,
	\]
	comparing with \eqref{eq:infty} shows that $v_n$ is not the energy ground state of $\widetilde{E}_{\G_1^d}$ at mass $\mu_n$, thus showing the existence of two distinct critical points of the energy in $H_{\mu_n}^1(\G_1^d)$, provided $n$ is large enough.
\end{proof}

\section{Sharp constants in Gagliardo--Nirenberg inequalities: \\ proof of Theorem \ref{thm:GN} and Proposition \ref{prop:ex24}}
\label{sec:gn}

Let us focus here on the relation between $d$--dimensional Gagliardo--Nirenberg inequalities on $\R^d$ and on the grid $\G_1^d$ with edgelength 1. 
\begin{remark}
	\label{rem:gnr2}
	Exploiting the homogeneities of $Q_{q,\R^d}$, it is readily seen that $K_{q,\R^d}$ is attained for every $q\in\left(2,2^*\right)$, and standard regularity arguments show that optimizers are in $C^1(\R^d)\cap H^2(\R^d)\cap L^\infty(\R^d)$.
\end{remark}
As a preliminary step towards Theorem \ref{thm:GN}, we have the following lemma.
\begin{lemma}
	\label{lem:nonex}
	For every $q\in\left(2+\f4d,2^*\right)$, $K_{q,\G_1^d}$ is not attained.
\end{lemma}
\begin{proof}
	We argue by contradiction. Let $q\in\left(2+\f4d,2^*\right)$ be fixed and assume that there exists $u\in H^1(\G_1^d)$  such that $Q_{q,\G_1^d}(u)=K_{q,\G_1^d}$. By homogeneity, this yields $Q_{q,\G_1^d}(cu)=K_{q,\G_1^d}$ for every $c\in\R$. Let then $v:=cu$ and note that, by definition,
	\[
	\|v'\|_{L^2(\G_1^d)}^2\geq\|v\|_{L^q(\G_1^d)}^q
	\]
	provided $c>0$ is sufficiently small. Coupling with $Q_{q,\G_1^d}(v)=K_{q,\G_1^d}$ entails
	\[
	\|v'\|_{L^2(\G_1^d)}^2\geq K_{q,\G_1^d}\|v\|_{L^2(\G_1^d)}^{d+(2-d)\f q2}\|v'\|_{L^2(\G_1^d)}^{\left(\f q2-1\right)d},
	\]
	that is
	\begin{equation}
	\label{ass1}
	\|v'\|_{L^2(\G_1^d)}\gtrsim_{d,q}\|v\|_{L^2(\G_1^d)}^{\f{2d+(2-d)q}{4-(q-2)d}}\,.
	\end{equation}
	Viceversa, by the one--dimensional Gagliardo--Nirenberg inequality \eqref{eq:gn1d} and $Q_{q,\G_1^d}(v)=K_{q,\G_1^d}$  we obtain
	\[
	K_{q,\G_1^d}\|v\|_{L^2(\G_1^d)}^{d+(2-d)\f q2}\|v'\|_{L^2(\G_1^d)}^{\left(\f q2-1\right)d}\leq C_q\|v\|_{L^2(\G_1^d)}^{\f q2+1}\|v'\|_{L^2(\G_1^d)}^{\f q2-1}\,,
	\]
	so that 
	\[
	\|v'\|_{L^2(\G_1^d)}\lesssim_{d,q}\|v\|_{L^2(\G_1^d)}\,.
	\]
	Combining with \eqref{ass1} thus yields
	\[
	\|v\|_{L^2(\G_1^d)}^\f{2(2-q)}{4-(q-2)d}\gtrsim_{d,q}1\,,
	\]
	providing the contradiction we seek as soon as $c\to0^+$, since $q>2+\f4d$ and, by construction, $\|v\|_{L^2(\G_1^d)}=c\|u\|_{L^2(\G_1^d)}$.
\end{proof}
\begin{proof}[Proof of Theorem \ref{thm:GN}] 
		That $K_{q,\G_1^d}$ is not attained whenever $q\in\left(2+\f4d,2^*\right)$ is the content of Lemma \ref{lem:nonex}. We are thus left to show \eqref{eq:KGgeqKR2} and \eqref{eq:KG=KR2}. We split the proof in two parts.
		
		\smallskip
		{\em Part 1: proof of \eqref{eq:KGgeqKR2}}. Fix $q\in(2,2^*)$ and, according to Remark \ref{rem:gnr2}, let $u\in C^1(\R^d)\cap H^2(\R^d)\cap L^\infty(\R^d) $ be such that $Q_{q,\R^d}(u)=K_{q,\R^d}$. Denoting by $u_\varepsilon$ the restriction of $u$ to $\G_\varepsilon^d$ for every $\varepsilon>0$, Lemmas \ref{lem:GNeps}--\ref{lem:restr} entail
		\[
		\begin{split}
		K_{q,\G_1^d}=&\,\varepsilon^{-\left(\f q2-1\right)(d-1)}\sup_{v\in H^1(\G_\varepsilon^d)}\f{\|v\|_{L^q(\G_\varepsilon^d)}^q}{\|v\|_{L^2(\G_\varepsilon^d)}^{d+(2-d)\f q2}\|v'\|_{L^2(\G_\varepsilon^d)}^{\left(\f q2-1\right)d}}\geq \varepsilon^{-\left(\f q2-1\right)(d-1)}\f{\|u_\varepsilon\|_{L^q(\G_\varepsilon^d)}^q}{\|u_\varepsilon\|_{L^2(\G_\varepsilon^d)}^{d+(2-d)\f q2}\|u_\varepsilon'\|_{L^2(\G_\varepsilon^d)}^{\left(\f q2-1\right)d}}\\
		=&\,\varepsilon^{-\left(\f q2-1\right)(d-1)}\f{d\varepsilon^{1-d}\|u\|_{L^q(\R^d)}^q+o(1)}{\left(d\varepsilon^{1-d}\|u\|_{L^2(\R^d)}^2+o(1)\right)^{\f d2+(2-d)\f q4}\left(\varepsilon^{1-d}\|\nabla u\|_{L^2(\R^d)}^2+o(1)\right)^{\left(\f q2-1\right)\f d2}}\\
		=&\, d^{\f{(d-2)(q-2)}4}Q_{q,\R^d}(u)+o(1)=d^{\f{(d-2)(q-2)}4}K_{\R^d,q}+o(1)\,,
		\end{split}
		\]
		provided $\varepsilon$ is small enough, and passing to the limit for $\varepsilon\to0$ gives \eqref{eq:KGgeqKR2}.
		
		\smallskip
		{\em Part 2: proof of \eqref{eq:KG=KR2}}.  Note first that it is enough to prove \eqref{eq:KG=KR2} for every $q\in\left(2+\f4d,2^*\right)$, as the case $q=2+\f4d$ will then follow by continuity of the map $q\mapsto Q_{q,\G_1^d}(u)$, for every given $u\in H^1(\G_1^d)$. 
		
		Let now $q\in\left(2+\f4d,2^*\right)$ be fixed and $(u_n)_n\subset H_1^1(\G_1^d)$ be such that $Q_{q,\G_1^d}(u_n)=K_{q,\G_1^d}-\f1n$ for every $n$. Moreover, without loss of generality let $u_n$ attain its $L^\infty$ norm on the neighbourhood of radius 1 of the origin in $\G_1^d$, for every $n$. We first show that 
		\begin{equation}
			\label{eq:u'to0}
			\|u_n'\|_{L^2(\G_1^d)}\to0\qquad\text{as }n\to+\infty\,.
		\end{equation}
		Indeed, $Q_{q,\G_1^d}(u_n)=K_{q,\G_1^d}-\f1n$ and the one--dimensional Gagliardo--Nirenberg inequality \eqref{eq:gn1d} give, as in the proof of Lemma \ref{lem:nonex},
		\[
		\|u_n'\|_{L^2(\G_1^d)}\lesssim_{d,q} \|u_n\|_{L^2(\G_1^d)}=1\qquad\forall n\,,
		\]
		so that $(u_n)_n$ is bounded in $H^1(\G_1^d)$ and therefore, up to subsequences, $u_n\rightharpoonup u$ in $H^1(\G_1^d)$ as $n\to+\infty$, for some $u\in H^1(\G_1^d)$. Assume then by contradiction that $\liminf_{n\to+\infty}\|u_n'\|_{L^2(\G_1^d)}>0$.  This implies that $u\not\equiv 0$ on $H^1(\G_1^d)$. Indeed, if it were $u\equiv0$ on $\G_1^d$, then since $u_n$ always attains its $L^\infty$ on the same compact subset of $\G_1^d$ by assumption and $u_n\to u$ in $L_{\text{loc}}^\infty(\G_1^d)$, the boundedness in $L^2(\G_1^d)$ would yield
		\[
		\|u_n\|_{L^q(\G_1^d)}^q\leq\|u_n\|_{L^\infty(\G_1^d)}^{q-2}\|u_n\|_{L^2(\G_1^d)}^2\to0\qquad\text{as }n\to+\infty\,,
		\]
		so that
		\[
		Q_{q,\G_1^d}(u_n)=\f{\|u_n\|_{L^q(\G_1^d)}^q}{\|u_n'\|_{L^2(\G_1^d)}^{\left(\f q2-1\right)d}}\to0 \qquad \text{as }n\to+\infty\,,
		\]
		which is impossible since $Q_{q,\G_1^d}(u_n)=K_{q,\G_1^d}-\f1n$. 
		
		Let then $m:=\|u\|_{L^2(\G_1^d)}^2$, so that $m\in(0,1]$. Note that, if $m=1$, then the convergence of $u_n$ to $u$ is strong in $L^2(\G_1^d)$ and, by Gagliardo--Nirenberg inequalities \eqref{eq:GNd}, strong in $L^q(\G_1^d)$, so that the lower semicontinuity yields
		\[
		K_{q,\G_1^d}=\lim_{n\to+\infty}Q_{q,\G_1^d}(u_n)\leq Q_{q,\G_1^d}(u)\leq K_{q,\G_1^d}\,.
		\]
		Since this is impossible by Lemma \ref{lem:nonex}, it must be $m\in(0,1)$.  Possibly passing to a further subsequence (not relabeled), set now $\lambda:=\lim_{n\to+\infty}\f{\|u'\|_{L^2(\G_1^d)}^2}{\|u_n'\|_{L^2(\G_1^d)}^2}$.  Then, by weak convergence of $u_n$ to $u$,
		\[
		\begin{split}
			\lim_{n\to+\infty}\|u_n-u\|_{L^2(\G_1^d)}^2=\lim_{n\to+\infty}\|u_n\|_{L^2(\G_1^d)}^2-\|u\|_{L^2(\G_1^d)}^2=1-m\\
			\lim_{n\to+\infty}\f{\|u_n'-u\|_{L^2(\G_1^d)}^2}{\|u_n'\|_{L^2(\G_1^d)}^2}=\lim_{n\to+\infty}\f{\|u_n'\|_{L^2(\G_1^d)}^2-\|u\|_{L^2(\G_1^d)}^2}{\|u_n'\|_{L^2(\G_1^d)}^2}=1-\lambda\,,
		\end{split}
		\]
		so that, by Brezis--Lieb Lemma \cite{BL} and $u_n-u,u\in H^1(\G_1^d)$,
		\[
		\begin{split}
			K_{q,\G_1^d}=&\,\lim_{n\to+\infty}Q_{q,\G_1^d}(u_n)=\lim_{n\to+\infty}\left(\f{\|u_n-u\|_{L^q(\G_1^d)}^q}{\|u_n'\|_{L^2(\G_1^d)}^{\left(\f q2-1\right)d}}+\f{\|u\|_{L^q(\G_1^d)}^q}{\|u_n'\|_{L^2(\G_1^d)}^{\left(\f q2-1\right)d}}\right)\\
			=&\,\lim_{n\to+\infty}\|u_n-u\|_{L^2(\G_1^d)}^{d+\f{2-d}2 q}\left(\f{\|u_n'-u'\|_{L^2(\G_1^d)}}{\|u_n'\|_{L^2(\G_1^d)}}\right)^{\left(\f q2-1\right)d}Q_{q,\G_1^d}(u_n-u)\\
			&\qquad\qquad\qquad\qquad\qquad\qquad\qquad+\|u\|_{L^2(\G_1^d)}^{d+\f{2-d}2 q}Q_{q,\G_1^d}(u)\lim_{n\to+\infty}\left(\f{\|u'\|_{L^2(\G_1^d)}}{\|u_n'\|_{L^2(\G_1^d)}}\right)^{\left(\f q2-1\right)d}\\
			\leq&\, K_{q,\G_1^d}\left((1-m)^{\f d2+\f{2-d}4q}(1-\lambda)^{\left(\f q2-1\right)\f d2}+m^{\f d2+\f{2-d}4q}\lambda^{\left(\f q2-1\right)\f d2}\right)<K_{q,\G_1^d}\,,
		\end{split}
		\]
		since $(m,\lambda)\in(0,1)^2$ and the function $f(x,y):=x^{\f d2+\f{2-d}4q}y^{\left(\f q2-1\right)\f d2}+(1-x)^{\f d2+\f{2-d}4q}(1-y)^{\left(\f q2-1\right)\f d2}$ is strictly smaller than $1$ on $(0,1)^2$ for every $q>2$. This gives again a contradiction and proves \eqref{eq:u'to0}.
		
	Now, for every $n$, set $\varepsilon_n:=\|u_n'\|_{L^2(\G_1^d)}$ and $w_n(x):=u_n(x/\varepsilon_n)$ for every $x\in\G_{\varepsilon_n}^d$. Hence, $w_n\in H^1(\G_{\varepsilon_n}^d)$ and
	\begin{equation}
	\label{eq:normewn}
	\|w_n'\|_{L^2(\G_{\varepsilon_n}^d)}^2=\f{\|u_n'\|_{L^2(\G_1^d)}^2}{\varepsilon_n},\qquad\|w_n\|_{L^r(\G_{\varepsilon_n}^d)}^r=\varepsilon_n\|u_n\|_{L^r(\G_1^d)}^r\quad\forall r\geq1\,,
	\end{equation}
	so that
	\begin{equation}
		\label{eq:Qwn}
		\f{\|w_n\|_{L^q(\G_{\varepsilon_n}^d)}^q}{\|w_n\|_{L^2(\G_{\varepsilon_n}^d)}^{d+(2-d)\f q2}\|w_n'\|_{L^2(\G_{\varepsilon_n}^d)}^{\left(\f q2-1\right)d}}=\varepsilon_n^{\left(\f q2-1\right)(d-1)}Q_{q,\G_1^d}(u_n)=\varepsilon_n^{\left(\f q2-1\right)(d-1)}\left(K_{q,\G_1^d}-\f1n\right)\,.
	\end{equation}
	Let $\mathcal{A}w_n$ be the piecewise--affine extension of $w_n$ to $\R^d$ as in \eqref{eq:Ad} and $\widetilde{w}_n$ be the restriction of $\mathcal{A}w_n$ to $\G_{\varepsilon_n}^d$ as in Section \ref{sec:prel}. By Lemma \ref{lem:uutL2}, \eqref{eq:normewn}, the definition of $\varepsilon_n$ and \eqref{eq:u'to0}, we have
	\begin{equation}
	\label{QestL2}
	\begin{split}
		\|\widetilde{w}_n\|_{L^2(\G_{\varepsilon_n}^d)}^2\leq&\, \|w_n\|_{L^2(\G_{\varepsilon_n}^d)}^2+3\varepsilon_n\|w_n\|_{H^1(\G_{\varepsilon_n}^d)}^2=\left(1+3\varepsilon_n+3\varepsilon_n\f{\|w_n'\|_{L^2(\G_{\varepsilon_n}^d)}^2}{\|w_n\|_{L^2(\G_{\varepsilon_n}^d)}^2}\right)\,\|w_n\|_{L^2(\G_{\varepsilon_n}^d)}^2\\
		=&\,\left(1+3\varepsilon_n+\f3{\varepsilon_n}\f{\|u_n'\|_{L^2(\G_1^d)}^2}{\|u_n\|_{L^2(\G_1^d)}^2}\right)\,\|w_n\|_{L^2(\G_{\varepsilon_n}^d)}^2=(1+6\varepsilon_n)\|w_n\|_{L^2(\G_{\varepsilon_n}^d)}^2\qquad\text{as }n\to+\infty\,.
	\end{split}
	\end{equation}
	Furthermore, since $\widetilde{w}_n$ is the restriction of $\mathcal{A}w_n$ to $\G_{\varepsilon_n}^d$, arguing as in the proof of Lemma \ref{lem:restr} and recalling that, by construction, $\|\nabla \mathcal{A}w_n\|_{L^2(\R^d)}^2=\varepsilon_n^{d-1}\|\widetilde{w}_n'\|_{L^2(\G_{\varepsilon_n}^d)}^2$ gives
	\[
	\begin{split}
	\|\mathcal{A}w_n\|_{L^2(\R^d)}^2\leq&\f{\varepsilon_n^{d-1}}d\|\widetilde{w}_n\|_{L^2(\G_{\varepsilon_n}^d)}^2+\varepsilon_n\|\mathcal{A}w_n\|_{L^2(\R^d)}^2+(\varepsilon_n+\varepsilon_n^2)\|\nabla\mathcal{A}w_n\|_{L^2(\R^d)}^2\\
	\leq&\f{\varepsilon_n^{d-1}}d\|\widetilde{w}_n\|_{L^2(\G_{\varepsilon_n}^d)}^2+\varepsilon_n\|\mathcal{A}w_n\|_{L^2(\R^d)}^2+(\varepsilon_n+\varepsilon_n^2)\varepsilon_n^{d-1}\|\widetilde{w}_n'\|_{L^2(\G_{\varepsilon_n}^d)}^2\,,
	\end{split}
	\]
	that is 
	\begin{equation*}
	(1-\varepsilon_n)\|\mathcal{A}w_n\|_{L^2(\R^d)}^2\leq \f{\varepsilon_n^{d-1}}d\|\widetilde{w}_n\|_{L^2(\G_{\varepsilon_n}^d)}^2+(\varepsilon_n+\varepsilon_n^2)\varepsilon_n^{d-1}\|\widetilde{w}_n'\|_{L^2(\G_{\varepsilon_n}^d)}^2\,.
	\end{equation*}
	Coupling with \eqref{QestL2} and using Jensen inequality, \eqref{eq:normewn} and \eqref{eq:u'to0}, as $n\to+\infty$ we obtain
	\begin{equation}
	\label{AwL2}
	\begin{split}
	\|\mathcal{A}w_n\|_{L^2(\R^d)}^2\leq&\,\f{\varepsilon_n^{d-1}}d(1+o(1))\|w_n\|_{L^2(\G_{\varepsilon_n}^d)}^2+\varepsilon_n^d(1+o(1))\|w_n'\|_{L^2(\G_{\varepsilon_n}^d)}^2\\
	=&\,\f{\varepsilon_n^{d-1}}d\|w_n\|_{L^2(\G_{\varepsilon_n}^d)}^2\left(1+\f{d\|u_n'\|_{L^2(\G_1^d)}^2}{\varepsilon_n\|u_n\|_{L^2(\G_1^d)}^2}+o(1)\right)=\f{\varepsilon_n^{d-1}}d(1+o(1))\|w_n\|_{L^2(\G_{\varepsilon_n}^d)}^2\,.
	\end{split}
	\end{equation}
	Let now $d=2$ and $q>2$ or $d\geq3$ and $q\in\left(2,\f{2^*}2+1\right]$. By Lemma \ref{lem:uutLq} and $Q_{q,\G_1^d}(u_n)=K_{q,\G_1^d}-\f1n$ we have
	\begin{equation}
	\label{QestLq}
	\begin{split}
		\|\widetilde{w}_n\|_{L^q(\G_{\varepsilon_n}^d)}^q\geq&\,\|w_n\|_{L^q(\G_{\varepsilon_n}^d)}^q\left(1-C\varepsilon_n^{\f{q-2}2(d-1)+1}\f{\|w_n\|_{L^2(\G_{\varepsilon_n}^d)}^{\f {d+(2-d)(q-1)}2}\|w_n'\|_{L^2(\G_{\varepsilon_n}^d)}^{\f{q-2}2 d+1}}{\|w_n\|_{L^q(\G_{\varepsilon_n}^d)}^q}\right.\\
		&\qquad\qquad\qquad\qquad\qquad\qquad\qquad\qquad\qquad\qquad\left.-C\varepsilon_n^{\f12\left(\f{q-2}2 d+3\right)}\f{\|w_n'\|_{L^2(\G_{\varepsilon_n}^d)}^q}{\|w_n\|_{L^q(\G_{\varepsilon_n}^d)}^q}\right)\\
		=&\,\|w_n\|_{L^q(\G_{\varepsilon_n}^d)}^q\left(1-C\f{\|u_n\|_{L^2(\G_1^d)}^{\f {d+(2-d)(q-1)}2}\|u_n'\|_{L^2(\G_1^d)}^{\f{q-2}2 d+1}}{\|u_n\|_{L^q(\G_1^d)}^q}-C\varepsilon_n^{\f{q(d-2)}4-\f{d-1}2}\f{\|u_n'\|_{L^2(\G_1^d)}^q}{\|u_n\|_{L^q(\G_1^d)}^q}\right)\\
		=&\,\|w_n\|_{L^q(\G_{\varepsilon_n}^d)}^q\left(1-C\left(K_{q,\G_1}-\f1n\right)\f{\|u_n'\|_{L^2(\G_1^d)}}{\|u_n\|_{L^2(\G_1^d)}}\right.\\
		&\qquad\qquad\qquad\qquad\qquad\qquad\qquad\left.-C\left(K_{q,\G_1}-\f1n\right)\varepsilon_n^{\f{q(d-2)}4-\f{d-1}2}\f{\|u_n'\|_{L^2(\G_1^d)}^{d-\f{d-2}2 q}}{\|u_n\|_{L^2(\G_1^d)}^{d-\f{d-2}2 q}}\right)\\
		=&\,(1-o(1))\|w_n\|_{L^q(\G_{\varepsilon_n}^d)}^q\qquad\text{as }n\to+\infty\,,
	\end{split}
	\end{equation}
	since $\f{q(d-2)}4-\f{d-1}2+d-\f{d-2}2 q>0$. Moreover, applying \eqref{stimachiave} to $\mathcal{A}w_n$ and $\widetilde{w}_n$, making use of the $d$--dimensional Gagliardo--Nirenberg inequality \eqref{eq:GNd} in $L^{2(q-1)}(\R^d)$ (which is possible because $2(q-1)\leq2^*$ whenever $d=2$ or $d\geq3$ and $q\in\left(2,\f{2^*}2+1\right]$) and exploiting \eqref{AwL2}, \eqref{QestLq} and \eqref{eq:gradAL2},
	\begin{equation}
	\label{AwLq1}
	\begin{split}
	\|\mathcal{A}w_n\|_{L^q(\R^d)}^q\geq&\,\f{\varepsilon_n^{d-1}}d\|\widetilde{w}_n\|_{L^q(\G_{\varepsilon_n}^d)}^q\left(1-\varepsilon_n^{2-d}\f{d\|\mathcal{A}w_n\|_{L^{2(q-1)}(\R^d)}^{q-1}\|\nabla\mathcal{A}w_n\|_{L^2(\R^d)}}{\|\widetilde{w}_n\|_{L^q(\G_{\varepsilon_n}^d)}^q}\right)\\
	\geq&\,\f{\varepsilon_n^{d-1}}d\|\widetilde{w}_n\|_{L^q(\G_{\varepsilon_n}^d)}^q\left(1-C\varepsilon_n^{2-d}\f{\|\mathcal{A}w_n\|_{L^2(\R^d)}^{\f{d+(2-d)(q-1)}2}\|\nabla\mathcal{A}w_n\|_{L^2(\R^d)}^{\f{q-2}2 d +1}}{\|\widetilde{w}_n\|_{L^q(\G_{\varepsilon_n}^d)}^q}\right)\\
	\geq&\,\f{\varepsilon_n^{d-1}}d\|w_n\|_{L^q(\G_{\varepsilon_n}^d)}^q\left(1-C\varepsilon_n^{\f{q-2}2(d-1)+1}\f{\|w_n\|_{L^2(\G_{\varepsilon_n}^d)}^{\f{d+(2-d)(q-1)}2}\|w_n'\|_{L^2(\G_{\varepsilon_n}^d)}^{\f{q-2}2 d+1}}{\|w_n\|_{L^q(\G_{\varepsilon_n}^d)}^q}\right)\\
	=&\,\f{\varepsilon_n^{d-1}}d\|w_n\|_{L^q(\G_{\varepsilon_n}^d)}^q(1-o(1))\qquad\text{as }n\to+\infty\,,
	\end{split}
	\end{equation}
	where the constant $C>0$ is not relabeled line by line and it already takes into account the inessential $o(1)$.
	
	Let then $d\geq3$ and $q\in\left(\f{2^*}2+1,2^*\right)$ and take $\gamma\in(0,1)$ such that $d+\f{2-d}2 q-\f\gamma2>0$. Then, again by Lemma \ref{lem:uutLq} and $Q_{q,\G_1^d}(u_n)=K_{q,\G_1^d}-\f1n$,
	\begin{equation*}
		\begin{split}
		\|\widetilde{w}_n\|_{L^q(\G_{\varepsilon_n}^d)}^q\geq&\, \|w_n\|_{L^q(\G_{\varepsilon_n}^d)}^q\left(1-C\varepsilon_n^{\f q2+1-\gamma}\f{\|w_n\|_{L^2(\G_{\varepsilon_n}^d)}^\gamma \|w_n'\|_{L^2(\G_{\varepsilon_n}^d)}^{q-\gamma}}{\|w_n\|_{L^q(\G_{\varepsilon_n}^d)}^q}-C\varepsilon_n^{\f q2+1-\f\gamma2}\f{\|w_n'\|_{L^2(\G_{\varepsilon_n}^d)}^q}{\|w_n\|_{L^q(\G_{\varepsilon_n}^d)}^q}\right)\\
		=&\,\|w_n\|_{L^q(\G_{\varepsilon_n}^d)}^q\left(1-C\f{\|u_n\|_{L^2(\G_1^d)}^\gamma\|u_n'\|_{L^2(\G_1^d)}^{q-\gamma}}{\|u_n\|_{L^q(\G_1^d)}^q}+\varepsilon_n^{-\f\gamma 2}\f{\|u_n'\|_{L^2(\G_1^d)}^q}{\|u_n\|_{L^q(\G_1^d)}^q}\right)\\
		=&\,\|w_n\|_{L^q(\G_{\varepsilon_n}^d)}^q\left(1-C\f{\|u_n'\|_{L^2(\G_1^d)}^{d+\f{2-d}2 q-\gamma}}{\|u_n\|_{L^2(\G_1^d)}^{d+\f{2-d}2 q-\gamma}}-C\varepsilon_n^{-\f\gamma 2}\f{\|u_n'\|_{L^2(\G_1^d)}^{d+\f{2-d}2 q}}{\|u_n\|_{L^2(\G_1^d)}^{d+\f{2-d}2 q}}\right)\\
		=&\,\|w_n\|_{L^q(\G_{\varepsilon_n}^d)}^q\left(1-2C\varepsilon_n^{d-\f{d-2}2 q-\f\gamma 2}\right)=(1-o(1))\|w_n\|_{L^q(\G_{\varepsilon_n}^d)}^q\qquad\text{as }n\to+\infty\,.
		\end{split}
	\end{equation*}
	Combining with \eqref{stimachiave} on $\mathcal{A}w_n$, $2(q-1)>2^*$, Lemma \ref{lem:normA}, Proposition \ref{prop:GNint} and \eqref{AwL2} yields
	\begin{equation}
		\label{AwLq2}
		\begin{split}
		\|\mathcal{A}w_n\|_{L^q(\R^d)}^q\geq&\,\f{\varepsilon_n^{d-1}}d\|\widetilde{w}_n\|_{L^q(\G_{\varepsilon_n}^d)}^q\left(1-\varepsilon_n^{2-d}\f{d\|\mathcal{A}w_n\|_{L^{2(q-1)}(\R^d)}^{q-1}\|\nabla\mathcal{A}w_n\|_{L^2(\R^d)}}{\|\widetilde{w}_n\|_{L^q(\G_{\varepsilon_n}^d)}^q}\right)\\
		\geq&\,\f{\varepsilon_n^{d-1}}d\|w_n\|_{L^q(\G_{\varepsilon_n}^d)}^q\left(1-C\varepsilon_n^{\f q2+1-\f\gamma 2}\f{\|w_n\|_{L^2(\G_{\varepsilon_n}^d)}^{\f\gamma 2}\|w_n'\|_{L^2(\G_{\varepsilon_n}^d)}^{q-\f\gamma 2}}{\|w_n\|_{L^q(\G_{\varepsilon_n}^d)}^q}\right)\\
		=&\,\f{\varepsilon_n^{d-1}}d\|w_n\|_{L^q(\G_{\varepsilon_n}^d)}^q(1-o(1))\qquad\text{as }n\to+\infty\,.
		\end{split}
	\end{equation}
	Since $(\mathcal{A}w_n)_n\subset H^1(\R^d)$, by \eqref{AwL2}, \eqref{AwLq1}, \eqref{AwLq2}, Lemma \ref{lem:normA}, and recalling \eqref{eq:optGn} and \eqref{eq:Qwn}, it follows then
	\[
	\begin{split}
	K_{q,\R^d}\geq\limsup_{n\to+\infty}Q_{q,\R^d}(\mathcal{A}w_n)\geq&\,\limsup_{n\to+\infty}\f{\f{\varepsilon_n^{d-1}}d\|w_n\|_{L^q(\G_{\varepsilon_n}^d)}^q}{\left(\f{\varepsilon_n^{d-1}}d\|w_n\|_{L^2(\G_{\varepsilon_n}^d)}^2\right)^{\f d2+\f{2-d}4 q}\left(\varepsilon_n^{d-1}\|w_n'\|_{L^2(\G_{\varepsilon_n}^d)}^2\right)^{\left(\f q2-1\right)\f d2}}\\
	=&\,d^{-\f{(d-2)(q-2)}4}\limsup_{n\to+\infty}\varepsilon_n^{-\left(\f q2-1\right)(d-1)}\f{\|w_n\|_{L^q(\G_{\varepsilon_n}^d)}^q}{\|w_n\|_{L^2(\G_{\varepsilon_n}^d)}^{d+\f{2-d}2 q}\|w_n'\|_{L^2(\G_{\varepsilon_n}^d)}^{\left(\f q2-1\right)d}}\\
	=&\,d^{-\f{(d-2)(q-2)}4}K_{q,\G_1^d}\,,
	\end{split}
	\]
	which coupled with \eqref{eq:KGgeqKR2} concludes the proof of \eqref{eq:KG=KR2}.
\end{proof}
\begin{proof}[Proof of Proposition \ref{prop:ex24}]
	If $K_{q,\G_1^d}=K_{q,\R^d}$, the result is trivially true. Assume then $q\in\left(2,2+\f4d\right)$, $K_{q,\G_1^d}>K_{q,\R^d}$ and let $(u_n)_n\subset H_1^1(\G_1^d)$ be such that $Q_{q,\G_1^d}(u_n)=K_{q,\G_1^d}-\f1n$ for every $n$. Moreover, with no loss of generality we can take $u_n$ to attain its $L^\infty$ norm inside the neighbourhood of radius 1 of the origin, for every $n$. Arguing as in the proofs of Lemma \ref{lem:nonex} and Theorem \ref{thm:GN}, the one--dimensional Gagliardo--Nirenberg inequality \eqref{eq:gn1d} ensures that $(u_n)_n$ is bounded in $H^1(\G_1^d)$. Hence, up to subsequences, $u_n\rightharpoonup u$ in $H^1(\G_1^d)$ as $n\to+\infty$, for some $u\in H^1(\G_1^d)$ that, by semicontinuity, satisfies $m:=\|u\|_{L^2(\G_1^d)}^2\in [0,1]$. On the one hand, if we assume $m\in(0,1)$, setting (possibly passing to a further subsequence) $\lambda:=\lim_{n\to+\infty}\f{\|u'\|_{L^2(\G_1^d)}^2}{\|u_n'\|_{L^2(\G_1^d)}^2}$ and arguing as in the proof of Theorem \ref{thm:GN} leads to a contradiction. On the other hand, if $m=0$, then up to subsequences $u_n\rightharpoonup 0$ in $H^1(\G_1^d)$ and $u_n\to 0$ in $L_{\text{\normalfont loc}}^\infty(\G_1^d)$ as $n\to+\infty$. Since, by assumption, $u_n$ attains its $L^\infty$ norm inside a fixed compact subset of $\G_1^d$ for every $n$, this entails that $u_n\to0$ in $L^\infty(\G_1^d)$, which itself implies $u_n\to 0$ in $L^q(\G_1^d)$, as $\|u_n\|_{L^2(\G_1^d)}=1$ for every $n$. Recalling that $Q_{q,\G_1^d}(u_n)=K_{q,\G_1^d}-\f1n>0$, this gives 
	\[
		\|u_n'\|_{L^2(\G_1^d)}\to0\qquad\text{as }n\to+\infty\,.
	\]
	Hence, setting $\varepsilon_n:=\|u_n'\|_{L^2(\G_1^d)}$, $w_n(x):=u_n(x/\varepsilon_n)$ for every $x\in \G_{\varepsilon_n}^d$,
	and arguing exactly as in the last part of the proof of Theorem \ref{thm:GN}, we obtain
	\[
	K_{q,\R^d}\geq\limsup_{n\to+\infty}\f{\|\mathcal{A}w_n\|_{L^q(\R^d)}^q}{\|\mathcal{A}w_n\|_{L^2(\R^d)}^{d+\f{2-d}2 q}\|\nabla\mathcal{A}w_n\|_{L^2(\R^d)}^{\left(\f q2-1\right)d}}\geq d^{-\f{(d-2)(q-2)}4}\lim_{n\to+\infty}Q_{q,\G_1^d}(u_n)=d^{-\f{(d-2)(q-2)}4}K_{q,\G_1^d}\,,
	\]
	which is again a contradiction. Therefore, it must be $m=1$, in turn implying that the convergence of $u_n$ to $u$ is strong in $L^2(\G_1^d)$ and, by \eqref{eq:gn1d} again, strong in $L^q(\G_1^d)$. By lower semicontinuity, this is enough to see that $Q_{q,\G_1^d}(u)=K_{q,\G-1^d}$, concluding the proof.
\end{proof} 

\section{Generalizations}
\label{sec:gen}
This section discusses the generalization of the method developed so far to non--cubic periodic grids. In general, 
to recover the results of Theorems \ref{thm:action}--\ref{thm:2dsquare}--\ref{thm:GN} and Proposition \ref{prop:ex24}, it is enough to modify the definition \eqref{eq:Ad} of the extension operator $\mathcal{A}$ according to the periodicity cell of the grid under exam. As a matter of fact, this will not change the scale factor $\varepsilon^{d-1}$ in front of the functionals, but will affect the $\varepsilon$--independent coefficients in front of the various terms. Hence, once a suitable definition of $\mathcal{A}$ is given, we just need to compute these new coefficients and then repeat the arguments of the previous sections with no significant modifications. 

Of course, the identification of the extension operator has to be performed case by case. Giving up on any vain ambition to treat every grid in one shot, here we limit to consider two explicit examples of non--square two--dimensional grids: the regular triangular grid and the regular hexagonal one. For the sake of brevity, when speaking of ground states, in what follows we describe explicitly only the extension to these grids of our results on energy ground states, that for those of the action being identical.

\subsection{The regular triangular grid}
If $\G_\varepsilon$ is the two--dimensional regular triangular grid in $\R^2$ with edgelength $\varepsilon$ as in Figure \ref{fig:altre}(A), with one vertex at the origin, we can write 
\[
\G_\varepsilon=\bigcup_{i,j\in\Z} \partial T_{\varepsilon,ij}\,,
\]
where the couples of indices $(i,j)\in\Z^2$ are in one--to--one correspondence with the triangles $T_{\varepsilon,ij}$ with edges of length $\varepsilon$ in which $\G_\varepsilon$ divides the plane, and $\partial T_{\varepsilon,ij}$ denotes their boundary in $\R^2$. Given $u:\G_\varepsilon\to\R$, we define $\mathcal{A}u:\R^2\to\R$ as 
\begin{equation}
\label{eq:defAtri}
\mathcal{A}u(x,y):=\mathcal{A}_{\varepsilon,ij}u(x,y)\qquad\text{if }(x,y)\in T_{\varepsilon,ij},\text{ for some }i,j\in\Z\,,
\end{equation}
with $\mathcal{A}_{\varepsilon,ij}u:T_{\varepsilon,ij}\to\R$ being the affine interpolation on $T_{\varepsilon,ij}$ of the values of $u$ at the vertices of $T_{\varepsilon,ij}$. Moreover, we define $\widetilde{E}_{\G_\varepsilon}:H^1(\G_\varepsilon)\to\R$ as
\begin{equation}
\label{eq:Etildetri}
\widetilde{E}_{\G_\varepsilon}(u):=\f1{2\sqrt{3}}\|u'\|_{L^2(\G_\varepsilon)}^2-\f1{2\sqrt{3}p}\|u\|_{L^p(\G_\varepsilon)}^p
\end{equation}
and 
\begin{equation}
	\label{eq:Ltri}
	\mathcal{L}_{\G_\varepsilon}(u):=\f{\f12\|u\|_{L^p(\G_\varepsilon)}^p-\|u'\|_{L^2(\G_\varepsilon)}^2}{\|u\|_{L^2(\G_\varepsilon)}^2}\,.
\end{equation}
Since the analysis of \cite{ADST} generalizes to two--dimensional regular triangular grids, existence of ground states of $\widetilde{E}_{\G_\varepsilon}$ is guaranteed for every $p\in(2,4)$ and $\mu>0$, and they are one--signed solutions of
\begin{equation}
	\label{eq:nlsG2}
\begin{cases}
	u''+\f12|u|^{p-2}u=\mathcal{L}_{\G_\varepsilon}(u)u & \text{on every edge of }\G_\varepsilon\\
	\sum_{e\succ \v}\f{du}{dx_e}(\v)=0 & \text{for every vertex }\v\text{ of }\G_\varepsilon\,.
\end{cases}
\end{equation}
We then have the next theorem, that is the analogue of Theorem \ref{thm:2dsquare} for regular two--dimensional triangular grids.
\begin{theorem}
	\label{thm:gstri}
	Let $p\in(2,4)$ and $\mu>0$ be fixed. For every $\varepsilon>0$, let $\G_\varepsilon$ be the two--dimensional regular triangular grid with edgelength $\varepsilon$ and $\widetilde{E}_{\G_\varepsilon}$ be as in \eqref{eq:Etildetri}. Then 
	\begin{itemize}
		\item[$(i)$] 
		there exists $C_p>0$, depending only on $p$, such that
		\begin{equation*}
		\left|\varepsilon\widetilde{\EE}_{\G_\varepsilon}\left(\f{2\sqrt{3}\mu}\varepsilon\right)-\EE_{\R^2}(\mu)\right|\leq C_p\varepsilon\qquad\text{as }\varepsilon\to0\,;
		\end{equation*}
		\item[$(ii)$] for every positive ground state $u_\varepsilon$ of $\widetilde{E}_{\G_\varepsilon}$ in $H_{\f{2\sqrt{3}\mu}\varepsilon}^1(\G_\varepsilon)$ there exists $x_\varepsilon\in\R^2$ such that
		\begin{equation*}
		\mathcal{A} u_\varepsilon(\cdot-x_\varepsilon) \xrightarrow[]{\varepsilon\to0} \phi_\mu\quad\text{ in }H^1(\R^2)\,,
		\end{equation*}
		where the extension operator $\mathcal{A}$ is as in \eqref{eq:defAtri} and $\phi_\mu$ is the unique positive ground state of $E_{\R^2}$ at mass $\mu$ attaining its $L^\infty$ norm at the origin. Furthermore, 
		\begin{equation*}
			\lim_{\varepsilon\to0}\mathcal{L}_{\G_\varepsilon}(u_\varepsilon)=\f{\omega_\mu}{2\sqrt{3}}\,,
		\end{equation*}
		where $\mathcal{L}_{\G_\varepsilon}(u_\varepsilon)$ is defined as in \eqref{eq:Ltri} and $\omega_\mu$ is the value of $\omega$ for which $\phi_\mu$ solves \eqref{eq:NLSR2}.
	\end{itemize}
\end{theorem}
As for sharp constants in two--dimensional Gagliardo--Nirenberg inequalities, we have the following.
\begin{theorem}
	\label{thm:GNtri}
	Let $\G_1$ be the two--dimensional regular triangular grid with edgelength 1. For every $q>2$, let $K_{q,\G_1}$ be the sharp constant in the two--dimensional Gagliardo--Nirenberg inequality \eqref{eq:GNd} on $\G_1$. Then
	\begin{equation*}
	K_{q,\G_1}\geq 3^{\f{2-q}4}K_{q,\R^2}\qquad\forall q>2\,.
	\end{equation*}
	If $q\geq4$, then
	\begin{equation*}
	K_{q,\G_1}=3^{\f{2-q}4}K_{q,\R^2}
	\end{equation*}
	and $K_{q,\G_1}$ is not attained for every $q>4$. Furthemore, if $q\in(2,4)$ and there exists $u\in H^1(\G_1)$ such that $Q_{q,\G_1}(u)\geq3^{\f{2-q}4}K_{q,\R^2}$ (where $Q_{q,\G_1}$ is defined as in \eqref{eq:defQ}), then $K_{q,\G_1}$ is attained. 
\end{theorem}
As anticipated, the unique difference between Theorems \ref{thm:2dsquare}--\ref{thm:GN}, Proposition \ref{prop:ex24} and Theorems \ref{thm:gstri}--\ref{thm:GNtri} is in the numerology.  The above results follow repeating the same argument as for the square grid, coupled whenever needed with the next two lemmas.
\begin{lemma}
	Let $u\in C^1(\R^2)\cap H^2(\R^2)$. For every $\varepsilon>0$, let $u_\varepsilon:\G_\varepsilon\to\R$ be the restriction of $u$ to the regular triangular grid $\G_\varepsilon$ with edgelength $\varepsilon$. Then there exists $C>0$, depending on $u$ but not on $\varepsilon$, such that, as $\varepsilon\to0$, it holds
	\begin{align}
		&\left|\f\varepsilon{2\sqrt{3}}\|u_\varepsilon\|_{L^q(\G_\varepsilon)}^q-\|u\|_{L^q(\R^2)}^q\right|\leq C\varepsilon\,,\qquad \forall q\geq2,\label{eq:uR2Gtri}\\
		&\left|\f\varepsilon{\sqrt{3}}\|u_\varepsilon'\|_{L^2(\G_\varepsilon)}^2-\|\nabla u\|_{L^2(\R^2)}^2\right|\leq C\varepsilon\,.\label{eq:u'R2Gtri}
	\end{align}
\end{lemma}
\begin{proof}
	The argument being analogous to that in Lemma \ref{lem:restr}, we just sketch the proof of \eqref{eq:uR2Gtri} with $q=2$ and of \eqref{eq:u'R2Gtri}. In the following, we will always denote by $C$ a suitable positive constant possibly depending only on $p$ and $u$, without renaming it even when varying line by line.
	
	Here it is convenient to think of $\G_\varepsilon$ as
	\[
	\G_\varepsilon=\bigcup_{i\in\Z}\left(H_{\varepsilon,i}\cup L_{\varepsilon,i}\cup R_{\varepsilon,i}\right),
	\]
	where
	\[
	\begin{split}
	H_{\varepsilon,i}:=&\bigcup_{j\in\Z}[\varepsilon j,\varepsilon(j+1)]\times\left\{\f{\sqrt{3}}2 \varepsilon i\right\}\\
	L_{\varepsilon,i}:=&\left\{(x,\sqrt{3}(x-\varepsilon i))\,:\,x\in\R\right\}\\
	R_{\varepsilon,i}:=&\left\{(x,-\sqrt{3}(x-\varepsilon i))\,:\,x\in\R\right\}\,.
	\end{split}
	\]
	To prove \eqref{eq:uR2Gtri} with $q=2$, consider first the function $w_\varepsilon:\R^2\to\R$ given by
	\[
	w_\varepsilon(x,y):=u\left(x,\f{\sqrt{3}}2\varepsilon i\right)\qquad\forall (x,y)\in\R\times\left[\f{\sqrt{3}}2\varepsilon i,\f{\sqrt{3}}2\varepsilon(i+1)\right),\text{ for some }i\in\Z\,.
	\]
	Arguing as in Part 1 of the proof of Lemma \ref{lem:restr}, it follows that $\left|\|w_\varepsilon\|_{L^2(\R^2)}^2-\|u\|_{L^2(\R^2)}^2\right|\leq C\varepsilon$, and a direct computation shows that $\|w_\varepsilon\|_{L^2(\R^2)}^2=\f{\sqrt{3}}2\varepsilon\|u_\varepsilon\|_{L^2(\bigcup_{i\in\Z}H_{\varepsilon,i})}^2$, so that
	\begin{equation}
		\label{eq:utri1}
		\left|\f{\sqrt{3}}2\varepsilon\|u_\varepsilon\|_{L^2(\bigcup_{i\in\Z}H_{\varepsilon,i})}^2-\|u\|_{L^2(\R^2)}^2\right|\leq C\varepsilon\,.
	\end{equation}
	Define then $v_\varepsilon:\R^2\to\R$ as
	\[
	v_\varepsilon(x,y):=u_\varepsilon(x,\sqrt{3}(x-\varepsilon i))\qquad\forall (x,y)\in A_{\varepsilon,i},\text{ for some }i\in\Z\,,
	\]
	where $A_{\varepsilon,i}:=\left\{(x,y)\in\R^2\,:\,\sqrt{3}(x-\varepsilon i)\leq y<\sqrt{3}(x-\varepsilon(i+1))\right\}$. 
	Again, $\left|\|v_\varepsilon\|_{L^2(\R^2)}^2-\|u\|_{L^2(\R^2)}^2\right|\leq C \varepsilon$. Furthermore, 
	\[
	\begin{split}
	\|v_\varepsilon\|_{L^2(\R^2)}^2=&\sum_{i\in\Z}\int_{\R}\int_{\sqrt{3}(x-\varepsilon i)}^{\sqrt{3}(x-\varepsilon(i+1))}|u_\varepsilon(x,\sqrt{3}(x-\varepsilon i))|^2\,dydx\\
	=&\sqrt{3}\varepsilon\sum_{i\in\Z}\sum_{j\in\Z}\int_{\f\varepsilon2 j}^{\f\varepsilon2(j+1)}|u_\varepsilon(x,\sqrt{3}(x-\varepsilon i))|^2dx=\f{\sqrt{3}}2\varepsilon\sum_{i\in\Z}\|u_\varepsilon\|_{L^2(L_{\varepsilon,i})}^2\,,
	\end{split}
	\]
	so that
	\begin{equation}
		\label{eq:utri2}
		\left|\f{\sqrt{3}}2\varepsilon\|u_\varepsilon\|_{L^2(\bigcup_{i\in\Z}L_{\varepsilon,i})}^2-\|u\|_{L^2(\R^2)}^2\right|\leq C\varepsilon\,.
	\end{equation}
Since an analogous computation yields
\[
\left|\f{\sqrt{3}}2\varepsilon\|u_\varepsilon\|_{L^2(\bigcup_{i\in\Z}R_{\varepsilon,i})}^2-\|u\|_{L^2(\R^2)}^2\right|\leq C\varepsilon\,,
\]
summing with \eqref{eq:utri1} and \eqref{eq:utri2} gives \eqref{eq:uR2Gtri} with $q=2$.

Let us now focus on \eqref{eq:u'R2Gtri}. Arguing as in Part 2 of the proof of Lemma \ref{lem:restr}, we immediately obtain that
\begin{equation}
\label{eq:u'tri1}
\left|\f{\sqrt{3}}2\varepsilon\|u_\varepsilon'\|_{L^2(\bigcup_{i\in\Z}H_{\varepsilon,i})}^2-\|\partial_x u\|_{L^2(\R^2)}^2\right|\leq C\varepsilon\,.
\end{equation}
Let $g_\varepsilon:\R^2\to\R$ be the function
\[
g_\varepsilon(x,y):=\nabla u (x,\sqrt{3}(x-\varepsilon i))\cdot (1,\sqrt{3})\qquad\forall (x,y)\in A_{\varepsilon,i}, \text{ for some }i\in\Z\,,
\]
where $A_{\varepsilon,i}$ is as above. Then $\left|\|g_\varepsilon\|_{L^2(\R^2)}^2-\|\partial_x u+\sqrt{3}\partial_y u\|_{L^2(\R^2)}^2\right|\leq C\varepsilon$ and
\[
\begin{split}
\|g_\varepsilon\|_{L^2(\R^2)}^2=&\sum_{i\in\Z}\int_\R\int_{\sqrt{3}(x-\varepsilon i)}^{\sqrt{3}(x-\varepsilon(i+1))}|\nabla u (x,\sqrt{3}(x-\varepsilon i))\cdot (1,\sqrt{3})|^2\,dydx\\
=&4\sqrt{3}\varepsilon\sum_{i\in\Z}\sum_{j\in\Z}\int_{\f\varepsilon2 j}^{\f\varepsilon2(j+1)}\left|\nabla u (x,\sqrt{3}(x-\varepsilon i))\cdot \left(\f12,\f{\sqrt{3}}2\right)\right|^2\,dx=2\sqrt{3}\varepsilon\sum_{i\in\Z}\|u_\varepsilon'\|_{L^2(L_{\varepsilon,i})}^2\,,
\end{split}
\]
so that
\begin{equation}
	\label{u'tri2}
	\left|2\sqrt{3}\varepsilon\|u_\varepsilon'\|_{L^2(\bigcup_{i\in\Z}L_{\varepsilon,i})}^2-\|\partial_x u+\sqrt{3}\partial_y u\|_{L^2(\R^2)}^2\right|\leq C\varepsilon\,.
\end{equation}
Similarly,
\[
\left|2\sqrt{3}\varepsilon\|u_\varepsilon'\|_{L^2(\bigcup_{i\in\Z}R_{\varepsilon,i})}^2-\|\partial_x u-\sqrt{3}\partial_y u\|_{L^2(\R^2)}^2\right|\leq C\varepsilon\,,
\]
which summing with \eqref{u'tri2} gives
\[
\left|\sqrt{3}\varepsilon\|u_\varepsilon'\|_{L^2(\bigcup_{i\in\Z}(L_{\varepsilon,i}\cup R_{\varepsilon,i}))}^2-\|\partial_x u\|_{L^2(\R^2)}^2-3\|\partial_y u\|_{L^2(\R^2)}^2\right|\leq C\varepsilon\,.
\]
Coupling with \eqref{eq:u'tri1} leads to \eqref{eq:u'R2Gtri}.
\end{proof}

\begin{lemma}
	Let $\G_\varepsilon$ be the regular triangular grid with edgelength $\varepsilon$. For every $u\in H^1(\G_\varepsilon)$, it holds
	\[
	\|\nabla\mathcal{A}u\|_{L^2(\R^2)}^2=\f{\varepsilon}{\sqrt{3}}\|\widetilde{u}'\|_{L^2(\G_\varepsilon)}^2\,,
	\]
	where $\mathcal{A}u$ is as in \eqref{eq:defAtri} and $\widetilde{u}$ denotes the restriction of $\mathcal{A}u$ to $\G_\varepsilon$.
\end{lemma}
\begin{proof}
	It is a straightforward consequence of the fact that, by definition of $\mathcal{A}u$ and $\widetilde{u}$, 
	\[
	\|\nabla \mathcal{A}u\|_{L^2(T_{\varepsilon,ij})}^2=\f\varepsilon{2\sqrt{3}}\|\widetilde{u}'\|_{L^2(\partial T_{\varepsilon,ij})}^2\,,\qquad\forall (i,j)\in\Z^2\,,
	\]
	and that each edge of $\G_\varepsilon$ belongs to the boundary of two of the triangles $T_{\varepsilon,ij}$.
\end{proof}
\subsection{The regular hexagonal grid}
If $\G_\varepsilon$ is the two--dimensional regular hexagonal grid with edgelength $\varepsilon$ as in Figure \ref{fig:altre}(B), with one vertex at the origin, we can write
\[
\G_\varepsilon:=\bigcup_{i,j\in\Z}\partial H_{\varepsilon,ij}\,,
\]
where the couples of indices $(i,j)\in\Z^2$ are in one--to--one correspondence with the hexagons $H_{\varepsilon,ij}$ with edges of length $\varepsilon$ in which $\G_\varepsilon$ divides the plane, and $\partial H_{\varepsilon,ij}$ denotes their boundary in $\R^2$. For each $i,j$, we then consider $H_{\varepsilon,ij}=T_{\varepsilon,ij}^1\cup T_{\varepsilon,ij}^2\cup T_{\varepsilon,ij}^3\cup T_{\varepsilon,ij}^4$, where $T_{\varepsilon,ij}^1$ is the triangle given by the three leftmost vertices of $H_{\varepsilon,ij}$, $T_{\varepsilon,ij}^2$ is the triangle given by the bottom--left vertex and the two top vertices of $H_{\varepsilon,ij}$, $T_{\varepsilon,ij}^3$ is the triangle given by the two bottom vertices and the top--right vertex of $H_{\varepsilon,ij}$ and $T_{\varepsilon,ij}^4$ is the triangle given by the three rightmost vertices of $H_{\varepsilon,ij}$ (Figure \ref{fig:Ahex}). Given $u:\G_\varepsilon\to\R$, we then define $\mathcal{A}u:\R^2\to\R$ as 
\begin{equation}
\label{eq:defAhex}
\mathcal{A}u(x,y):=\mathcal{A}_{\varepsilon,ij}^k u\qquad\text{if }(x,y)\in T_{\varepsilon,ij}^k,\text{ for some }i,j\in\Z,\,k\in\left\{1,2,3,4\right\},
\end{equation}
with $\mathcal{A}_{\varepsilon,ij}^ku:T_{\varepsilon,ij}^k\to\R$ being the affine interpolation of the values of $u$ at the vertices of $T_{\varepsilon,ij}^k$. Moreover, we define $\widetilde{E}_{\G_\varepsilon}:H^1(\G_\varepsilon)\to\R$ as
\begin{equation}
\label{eq:Etildehex}
\widetilde{E}_{\G_\varepsilon}(u):=\f{\sqrt{3}}2\|u'\|_{L^2(\G_\varepsilon)}^2-\f{\sqrt{3}}{2p}\|u\|_{L^p(\G_\varepsilon)}^p
\end{equation}
and $\mathcal{L}_{\G_\varepsilon}(u)$ as in \eqref{eq:Ltri}. Existence of energy ground states for every $p\in(2,4)$ and $\mu>0$ on hexagonal grids can be found in \cite{ADR} and computing the associated Euler--Lagrange equations show that they are constant sign solutions to \eqref{eq:nlsG2} on $\G_\varepsilon$. 

\begin{figure}[t]
	\centering
	\includegraphics[width=0.25\textwidth]{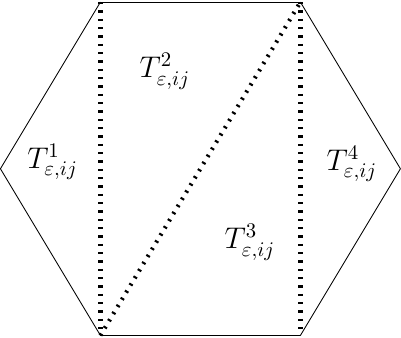}
	\caption{The splitting of an hexagonal cell in the definition of the extension operator \eqref{eq:defAhex} on two--dimensional regular hexagonal grid.}
	\label{fig:Ahex}
\end{figure}

\begin{theorem}
	\label{thm:gshex}
	Let $p\in(2,4)$ and $\mu>0$ be fixed. For every $\varepsilon>0$, let $\G_\varepsilon$ be the two--dimensional regular hexagonal grid with edgelength $\varepsilon$ and $\widetilde{E}_{\G_\varepsilon}$ be as in \eqref{eq:Etildehex}. Then 
	\begin{itemize}
		\item[$(i)$] there exists $C_p>0$, depending only on $p$, such that
		\begin{equation*}
		\left|\varepsilon\widetilde{\EE}_{\G_\varepsilon}\left(\f{2\mu}{\sqrt{3}\varepsilon}\right)-\EE_{\R^2}(\mu)\right|\leq C_p\varepsilon\qquad\text{as }\varepsilon\to0\,;
		\end{equation*}
		\item[$(ii)$] for every positive ground state $u_\varepsilon$ of $\widetilde{E}_{\G_\varepsilon}$ in $H_{\f{2\mu}{\sqrt{3}\varepsilon}}^1(\G_\varepsilon)$ there exists $x_\varepsilon\in\R^2$ such that
		\begin{equation*}
		\mathcal{A} u_\varepsilon(\cdot-x_\varepsilon) \xrightarrow[]{\varepsilon\to0} \phi_\mu\quad\text{ in }H^1(\R^2)\,,
		\end{equation*}
		where the extension operator $\mathcal{A}$ is as in \eqref{eq:defAhex} and $\phi_\mu$ is the unique positive ground state of $E_{\R^2}$ at mass $\mu$ attaining its $L^\infty$ norm at the origin.  Furthemore, 
		\begin{equation*}
			\lim_{\varepsilon\to0}\mathcal{L}_{\G_\varepsilon}(u_\varepsilon)=\f{\sqrt{3}\omega_\mu}2\qquad\text{as }\varepsilon\to0\,,
		\end{equation*}
		where $\mathcal{L}_{\G_\varepsilon}(u_\varepsilon)$ is defined as in \eqref{eq:Ltri} and $\omega_\mu$ is the value of $\omega$ for which $\phi_\mu$ solves \eqref{eq:NLSR2}.
	\end{itemize}
\end{theorem}
\begin{theorem}
	\label{thm:GNhex}
	Let $\G_1$ be the two--dimensional regular hexagonal grid with edgelength 1. For every $q>2$, let $K_{q,\G_1}$ be the sharp constant in the two--dimensional Gagliardo--Nirenberg inequality \eqref{eq:GNd} on $\G_1$. Then
	\begin{equation*}
	K_{q,\G_1}\geq 3^{\f{q-2}4}K_{q,\R^2}\qquad\forall q>2\,.
	\end{equation*}
	If $q\geq4$, then
	\begin{equation*}
	K_{q,\G_1}=3^{\f{q-2}4}K_{q,\R^2}
	\end{equation*}
	and $K_{q,\G_1}$ is not attained for every $q>4$. Furthemore, if $q\in(2,4)$ and there exists $u\in H^1(\G_1)$ such that $Q_{q,\G_1}(u)\geq3^{\f{q-2}4}K_{q,\R^2}$ (where $Q_{q,\G_1}$ is defined as in \eqref{eq:defQ}), then $K_{q,\G_1}$ is attained. 
\end{theorem}
As in the previous case, the proof of these results combines the discussion performed for square grids with the next two lemmas.
\begin{lemma}
	Let $u\in C^1(\R^2)\cap H^2(\R^2)$. For every $\varepsilon>0$, let $u_\varepsilon:\G_\varepsilon\to\R$ be the restriction of $u$ to the regular hexagonal grid $\G_\varepsilon$ with edgelength $\varepsilon$. Then there exists $C>0$, depending on $u$ but not on $\varepsilon$, such that, as $\varepsilon\to0$, it holds
	\begin{align*}
		&\left|\f{\sqrt{3}}2\varepsilon\|u_\varepsilon\|_{L^q(\G_\varepsilon)}^q-\|u\|_{L^q(\R^2)}^q\right|\leq C\varepsilon\,,\qquad \forall q\geq2,\\
		&\left|\sqrt{3}\varepsilon\|u_\varepsilon'\|_{L^2(\G_\varepsilon)}^2-\|\nabla u\|_{L^2(\R^2)}^2\right|\leq C\varepsilon\,.
	\end{align*}
\end{lemma}
\begin{proof}
	It is evident from the fact that the regular hexagonal grid is the subset of the regular  triangular one given by the removal of one edge over three on each $H_{\varepsilon,i}$, $L_{\varepsilon,i}$ and $R_{\varepsilon,i}$, for every $i\in\Z$.
\end{proof}

\begin{lemma}
	Let $\G_\varepsilon$ be the regular hexagonal grid with edgelength $\varepsilon$. For every $u\in H^1(\G_\varepsilon)$, it holds
	\[
	\|\nabla\mathcal{A}u\|_{L^2(\R^2)}^2=\sqrt{3}\varepsilon\|\widetilde{u}'\|_{L^2(\G_\varepsilon)}^2\,,
	\]
	where $\mathcal{A}u$ is as in \eqref{eq:defAhex} and $\widetilde{u}$ denotes the restriction of $\mathcal{A}u$ to $\G_\varepsilon$.
\end{lemma}
\begin{proof}
	It is again a direct consequence of the definition of $\mathcal{A}u$ and the fact that each edge of $\G_\varepsilon$ belongs to exactly two of the hexagons in which the grid divides the plane.
\end{proof}

\end{document}